\documentclass[journal]{IEEEtran}
\usepackage{bbm}
\usepackage{mathtools}
\usepackage{amsmath, amsfonts, amssymb}
\usepackage{amsthm}
\usepackage{booktabs}
\usepackage{amsmath}
\usepackage{diagbox}
\usepackage{fontawesome5}
\usepackage{algorithm,algorithmic}
\usepackage{hyperref}
\usepackage{bbm}
\usepackage{dsfont} 
\usepackage{bbm}
\usepackage{dsfont}
\usepackage{subfigure}
\usepackage{xcolor}         % colors
\usepackage[table]{xcolor}
\definecolor{sotagreen}{RGB}{220,245,225}
\newtheorem{theorem}{Theorem}

\newtheorem{Lemma}{Lemma}

\newtheorem{definition}{Definition}

\newtheorem{remark}[theorem]{Remark}

\newcommand{\bd}{\boldsymbol}

\allowdisplaybreaks

\usepackage{aliascnt}
\usepackage{bm}
\usepackage{amsfonts}
\usepackage{mathrsfs}
\usepackage{amsmath}

\def\bx{\boldsymbol{x}}
\def\by{\boldsymbol{y}}
\def\bz{\boldsymbol{z}}
\def\bu{\boldsymbol{u}}

\def\bg{\boldsymbol{g}}
\def\bq{\boldsymbol{q}}
\def\bs{\boldsymbol{s}}

\def\bd{\boldsymbol{d}}
\def\bm{\boldsymbol{m}}

\def\bpi{\boldsymbol{\pi}}

\DeclareMathOperator{\mX}{{\scriptstyle\boldsymbol{\mathcal{X}}}}

\DeclareMathOperator{\mD}{{\scriptstyle\boldsymbol{\mathcal{D}}}}

\DeclareMathOperator{\mY}{{\scriptstyle\boldsymbol{\mathcal{Y}}}}
\DeclareMathOperator{\mG}{{\scriptstyle\boldsymbol{\mathcal{G}}}}
\DeclareMathOperator{\mZ}
{{\scriptstyle\boldsymbol{\mathcal{Z}}}}

\DeclareMathOperator{\mhE}{{\scriptstyle\boldsymbol{\mathcal{E}}}}

\DeclareMathOperator{\mS}
{{\scriptstyle\boldsymbol{\mathcal{S}}}}
\DeclareMathOperator{\mR}
{{\scriptstyle\boldsymbol{\mathcal{R}}}}

\def\mQ{\mathcal{Q}}
\def\mT{\mathcal{T}}
\def\mU{\mathcal{U}}

\def\mB{{\mathcal{B}}}

\def\hnabla{\widehat{\nabla}}
\def\bxi{\boldsymbol{\xi}}
\def\bzeta{\boldsymbol{\zeta}}
\DeclareMathOperator{\mA}{{\mathcal{A}}}

\def\mC{\mathcal{C}}

\def\mP{\mathcal{P}}
\def\mR{\mathbb{R}}

\def\bxi{\boldsymbol{\xi}}

\def\mE{\mathbb{E}}

\newtheorem{Theorem}{Theorem}
\newtheorem{Assumption}{Assumption}
\newtheorem{Corollary}{Corollary}

\usepackage{booktabs}
\usepackage{tabularx}
\usepackage{makecell}
\usepackage[table]{xcolor}
\usepackage{threeparttable}
\usepackage{pifont} % for checkmarks/crosses
\definecolor{tableShade}{RGB}{245,248,252}
\ifCLASSINFOpdf
\else
\fi

\begin{document}
\title{A Unified Zeroth-Order Approach for Decentralized Minimax Optimization}

\author{Haoyuan Cai, Yike Zhao,
        Aleksandar Armacki,
        Jie Chen
        and~Ali H.~Sayed,
        \thanks{Haoyuan Cai, Yike, Zhao, Aleksandar Armacki, and Ali H.~Sayed are with the School of Engineering, École Polytechnique Fédérale de Lausanne, Switzerland. 
Emails: \{haoyuan.cai, yike.zhao, aleksandar.armacki, ali.sayed\}@epfl.ch}
\thanks{Jie Chen is with the School of Artificial Intelligence, Northwestern Polytechnical University, Xi'an, China. 
Email: jie.chen@nwpu.edu.cn}
}

% The paper headers
\markboth{}%
{Shell \MakeLowercase{\textit{et al.}}: Bare Demo of IEEEtran.cls for IEEE Journals}

% make the title area
\maketitle

\begin{abstract}
We propose ZOMA, a unified Zeroth-Order decentralized accelerated MinimAx framework for multi-agent nonconvex Polyak--\L{}ojasiewicz minimax optimization.
The proposed framework only requires evaluating the function value and, as such, is tailored to gradient-free environments,  where exact gradient information is either unavailable or computationally prohibitive to obtain. 
A central contribution of our \textbf{ZOMA} framework is a multi-level unification, along the following directions: 
(i) \emph{estimator} - our framework adopts a hybrid zeroth-order estimator, which accommodates, among others, both coordinate-wise and randomized uniform smoothing estimators;
 (ii) \emph{bias correction} - our framework subsumes a wide range of bias-correction strategies, including gradient tracking (GT), exact diffusion (ED), and EXTRA and (iii) \emph{acceleration} - our framework facilitates a broad class of acceleration techniques, including zeroth-order versions of STORM, PAGE, and L2S.
The general nature of \textbf{ZOMA} leads to many novel decentralized zeroth-order minimax methods and allows us to establish unified convergence guarantees, matching the performance of state-of-the-art centralized zeroth-order minimax methods, while providing benefits, such as linear speed-up in the number of users. The unified framework also provides a systematic way to assess algorithmic suitability by specializing the convergence rates to specific problem structures and method designs.
We validate the performance of the proposed algorithms via numerical simulations.
\end{abstract}


\section{Introduction}
\IEEEPARstart{S}{tochastic} zeroth-order (ZO) optimization has received growing attention due to its effectiveness in gradient-free applications, including black-box adversarial attacks \cite{chen2017zoo,tu2019autozoom} and memory-efficient fine-tuning of large language models (LLMs) \cite{malladi2023fine,zhang2024revisiting},
where the sheer size of the model makes it prohibitively expensive to compute the gradient. Another issue with such models is the huge volume of data required for training, further leading to high storage demands. As such, decentralized training represents a natural alternative, with data stored in many smaller chunks across a network of users.
Despite growing interest, existing ZO methods are largely limited to centralized minimax problems, leaving {\em decentralized minimax} optimization unexplored. Such
decentralized minimax formulations arise naturally in a wide range of applications, including decentralized generative adversarial networks \cite{liu2020decentralized}, AUC maximization \cite{guo2020communication}, and reinforcement learning \cite{li2019robust},
where learning is performed over networked multi-agent systems.
Existing approaches to these problems predominantly rely on first-order (FO) gradient information, which might be inaccessible due to the black-box nature, memory constraints, or privacy restrictions of the problem.
These challenges motivate the use of decentralized ZO minimax methods. 
Therefore,
we study the following minimax problems in a decentralized gradient-free setting
\begin{align}
&\min_{x \in \mathbb{R}^{d_1}}
\max_{y \in \mathbb{R}^{d_2}}
J(x, y)
= \frac{1}{K}\sum_{k=1}^K J_k(x,y), \quad
\label{main:sec1:problem}
\end{align}
where $K\ge2$ is the number of agents, $J_k(x,y)$ is the local cost of agent $k$, and the global cost function $J(x,y)$ is assumed to be smooth nonconvex in $x$ and nonconcave in $y$, satisfying the $\nu$-Polyak–Łojasiewicz (PL) condition; see Assumption \ref{main:assumption:costfunction} ahead.
We address two important learning scenarios, namely the stochastic and finite-sum settings, i.e.,
\begin{align}
 J_k(x, y) = 
\begin{cases}
\mathbb{E}_{\bzeta_k \sim \mathbb{P}_k}[Q_k(x, y; \bzeta_k)] & \textbf{(Stochastic)}\\
\frac{1}{N} \sum_{s=1}^{N} Q_{k}(x, y; \bzeta_{k}(s)) & \textbf{(Finite-sum)}, \notag
\end{cases}
\end{align}
where the local loss $Q_k(x,y;\bzeta_k)$ is computed with  $\bzeta_k$  drawn from an unknown distribution $\mathbb{P}_k$ 
in the stochastic setting, and $Q_k(x,y;\bzeta_k(s))$ is a component function $\forall s= [N]$ in the finite-sum setting.
We assume agents access a stochastic ZO oracle to perform local updates.
This gradient-free constraint significantly complicates algorithmic design and analysis, as stochastic ZO oracles typically exhibit higher noise variance than their FO counterparts \cite{ji2019improved,liu2020min,liu2018zeroth}.
Such difficulties can be further exacerbated in decentralized settings due to network-induced errors.
Recently, considerable progress has been made on designing FO minimax algorithms in both centralized  \cite{lin2020gradient,diakonikolas2021efficient,xu2023unified,yang2022faster,zheng2023universal,cai2024accelerated,huang2025enhanced,lin2025two,li2022tiada} and decentralized settings \cite{xian2021faster,gao2022decentralized,mancino2023variance,zhang2024jointly,xu2024decentralized,cai2025communication,cai2025dama1,cai2025dama2}.
In contrast, decentralized ZO minimax optimization remains much less understood, with several key questions still open. We next review related works to highlight these gaps.

\subsection{Related works}

\textbf{ZO methods}.
References  \cite{ghadimi2013stochastic,nesterov2017random,gao2018information, liu2018zerothaistat,ji2019improved,liu2020primer,liu2018zeroth,zhang2024revisiting,chen2017zoo,malladi2023fine,chen2025zeroth, yi2022zeroth,tang2020distributed,sahu2018distributed,sahu2020decentralized} have studied gradient-free minimization problems.  
For the  centralized case, references
\cite{ghadimi2013stochastic,nesterov2017random} studied the Gaussian smoothing-based ZO gradient technique. Subsequent works \cite{liu2018zeroth,ji2019improved} developed variance-reduced ZO methods. More recently, the works 
\cite{malladi2023fine,zhang2024revisiting} applied the ZO method in deep learning tasks for memory-efficient fine-tuning of LLMs using only forward function evaluations. 
On the {\em decentralized} side, the work
\cite{sahu2018distributed} addressed the decentralized ZO convex optimization, while references \cite{tang2020distributed, yi2022zeroth, chen2025zeroth} focused on the 
decentralized nonconvex optimization problems.
Due to the rapid growth of this area, we refer readers to the aforementioned works and the references therein for details.
In contrast to the single-variable minimization setting, we study the two-variable minimax optimization problems, which are more challenging due to the coupling between the two optimization variables.

\textbf{Centralized minimax methods.}
References \cite{lin2020gradient,lin2025two, yang2022faster,zheng2023universal,zhang2020single,huang2025enhanced,cai2024accelerated,luo2020stochastic}
focused on FO methods
for solving nonconvex strongly concave/PL minimax optimization problems.  Among these studies, the works \cite{lin2020gradient,lin2025two} analyzed the convergence of two-time-scale gradient descent–ascent (GDA) methods.  The works \cite{yang2022faster,zheng2023universal,zhang2020single} proposed smoothing variants of GDA to improve the geometry of the optimization problem. Variance-reduced techniques have been investigated in \cite{huang2025enhanced,cai2024accelerated,luo2020stochastic}.
In parallel, several studies have investigated single-machine ZO methods for minimax optimization \cite{liu2020min,wang2023zeroth,huang2022accelerated,xu2020gradient,xu2023zeroth,an2024robust}. For instance, reference \cite{wang2023zeroth} studied stochastic ZO-GDA  using the Gaussian smoothing technique under a nonconvex strongly-concave setting, while references \cite{huang2022accelerated, xu2020gradient, xu2023zeroth} proposed variance reduction techniques to accelerate the randomized ZO estimator in an online scenario, which are crucial for addressing ZO gradient noise.
However, existing accelerated ZO methods are designed specifically for the online stochastic setting and lack a unified acceleration framework that handles both stochastic and finite-sum cases, resulting in a notable gap.

\textbf{Decentralized minimax methods}.
Decentralized nonconvex strongly concave/PL minimax optimization has attracted increasing attention in recent years \cite{xian2021faster,huang2023near,cai2025communication, gao2022decentralized,zhang2024jointly,mancino2023variance, zhang2025federated,cai2025dama1,cai2025dama2}. Several works studied momentum–based gradient tracking (GT) methods \cite{xian2021faster,huang2023near,cai2025communication}, while variance reduction-based GT approaches have also been explored by the works \cite{gao2022decentralized,zhang2024jointly,mancino2023variance}. 
Although the  GT method can mitigate data heterogeneity, it is often outperformed by other bias-correction approaches in sparse networks (e.g., \cite{alghunaim2022unified,sayed2022inference}).
In addition, reference \cite{chen2024efficient} studied an efficient loopless variance-reduction method and a fast mixing scheme for stochastic and finite-sum settings. 
However, fast-mixing typically requires multiple-step communication for each update. 
More recently, references \cite{cai2025dama1,cai2025dama2} introduced a unified framework that subsumes multiple decentralized learning as special cases, with improved performance over sparse networks.

Despite the substantial progress in decentralized FO minimax optimization, 
{\em decentralized ZO minimax algorithms remain unexplored}, leaving the development of strong baselines and a unified theoretical understanding in this regime an open problem.
This motivates the following research questions.

\setlength{\tabcolsep}{2pt}  
\begin{table*}[t]
\centering
\caption{\footnotesize \normalfont Comparison of ZO minimax algorithms for reaching an $\varepsilon$-stationary point under  nonconvex strongly-concave/PL setups.
Communication and ZO oracle complexity refer to per-agent results. 
\textbf{Notation:}
ZO-CW: ZO coordinate-wise.
ZO-RU: ZO randomized uniform smoothing.
ZO-RG: ZO randomized Gaussian smoothing.
Accelerated schemes: STORM \cite{cutkosky2019momentum}, PAGE \cite{li2021page}, and L2S \cite{li2020convergence}. 
$d$: problem dimension, $d= d_1+d_2$.
$K$: the number of agents. $N$: sample size in finite-sum setting.
$\kappa$: condition number, $\kappa \triangleq L_f/\nu$, where $L_f$ is the smooth constant and $\nu$ is the strong concavity/PL parameter. $1-\lambda$: network spectral gap (approaching 
$0$
for sparse and 
$1$ for dense networks).
We observe that several of our methods achieve order-wise better ZO oracle complexity than centralized/single-machine methods, while maintaining linear speedup.
We note that EXTRA-based strategies achieve guarantees comparable to ED-based ones; details are omitted for brevity.}
\label{tab:zo-minimax-unified-comparison}

\begin{tabular}{@{} l l c c  @{}}
\toprule
\textbf{Regime} & \textbf{Algorithm} &
\textbf{Communication complexity} &
\textbf{ZO oracle complexity}  \\
\midrule
% ---------------- Centralized baseline ----------------
\multicolumn{4}{@{}l}{{\em Centralized (single-machine) }} \\
\midrule
Stochastic & ZO-RU, ZO-Min-Max \cite{liu2020min}&
------ &$\mathcal{O}(d\kappa^4\varepsilon^{-6})$  \\
Stochastic & ZO-RG+ZO-CW, ZO-VRGDA \cite{xu2020gradient}
&------& $\mathcal{O}(d\kappa^3\varepsilon^{-3})$
\\
Stochastic &ZO-RU, Acc-ZOMDA \cite{huang2022accelerated}& ------& 
$\widetilde{\mathcal{O}}(d^{\frac{3}{4}}\kappa^{4.5}\varepsilon^{-3})$
\\
stochastic & ZO-RG, ZO-SGDA \cite{wang2023zeroth} &------& $\mathcal{O}(\kappa^5\varepsilon^{-4})$  
\\
stochastic& ZO-RG, ZO-VRAGDA \cite{xu2023zeroth}&------& $\mathcal{O}(d_1d_2\kappa^3\varepsilon^{-3})$
\\
stochastic & ZO-RU, ZO-GDEGA \cite{an2024robust} &  ------ & $\mathcal{O}(\kappa^2(d_1+d_2\kappa)\varepsilon^{-4})$ 
\\
\midrule
% ---------------- tau = 0 ----------------
\multicolumn{4}{@{}l}{{\em Decentralized} --- \textbf{ZOMA} \textbf{(\color{red} Ours; includes the following as special cases)}.}  \\
\midrule
\multicolumn{4}{c}{ZO-CW$^{\scalebox{1.2}{$\star$}}$}
\\
\midrule
\rowcolor{sotagreen}
Stochastic
& \textsc{ZO-STORM-ED} (Corollary \ref{corollary:online:storm+ed}) & $ \mathcal{O}\Big(\frac{\kappa^3\varepsilon^{-3}}{K} + \frac{\kappa^2\varepsilon^{-2}}{(1-\lambda)^2} \Big)$ & $\mathcal{O}\Big(\frac{d\kappa^3\varepsilon^{-3}}{K} + \frac{d\kappa^2\varepsilon^{-2}}{(1-\lambda)^2} +  \frac{d\kappa \varepsilon^{-1}}{K} \Big)$ \\Stochastic
& \textsc{ZO-STORM-GT} (Corollary \ref{corollary:online:storm+gt}) & $\mathcal{O}\Big(\frac{\kappa^3\varepsilon^{-3}}{K} + \frac{\kappa^2\varepsilon^{-2}}{(1-\lambda)^3} \Big)$ & $\mathcal{O}\Big(\frac{d\kappa^3\varepsilon^{-3}}{K} + \frac{d\kappa^2\varepsilon^{-2}}{(1-\lambda)^3} +  \frac{d\kappa \varepsilon^{-1}}{K} \Big)$\\
\rowcolor{sotagreen}
Finite-sum
& \textsc{ZO-PAGE-ED} (Corollary \ref{corollary:offline:page+ed})& $ \mathcal{O}\Big(\frac{\kappa^2\varepsilon^{-2}}{(1-\lambda)^{1.5}}
+ \kappa^2\varepsilon^{-2} \Big)
$ & $ \mathcal{O}\Big(
\frac{d\kappa^2\sqrt{N}\varepsilon^{-2}}{\sqrt{K}(1-\lambda)^{1.5}} + d\sqrt{\frac{N}{K}}
\Big)$ \\ Finite-sum
& \textsc{ZO-PAGE-GT} (Corollary \ref{corollary:offline:page+atc-gt})& $\mathcal{O}\Big(\frac{\kappa^2\varepsilon^{-2}}{(1-\lambda)^{2}}
+ \kappa^2\varepsilon^{-2} \Big)$ & $\mathcal{O}\Big(
\frac{d\kappa^2\sqrt{N}\varepsilon^{-2}}{\sqrt{K}(1-\lambda)^{2}} + d\sqrt{\frac{N}{K}}
\Big)$ \\ 
\rowcolor{sotagreen}
Finite-sum
& \textsc{ZO-L2S-ED} (Corollary \ref{corollary:offline:L2S+ed}) & $ 
\mathcal{O}\Big(
\frac{\kappa^2\sqrt{N}\varepsilon^{-2}}{\sqrt{K}}
+\frac{\kappa^2\epsilon^{-2}}{(1-\lambda)^2}+ \frac{\kappa^2\sqrt{K}\varepsilon^{-2}}{\sqrt{N}(1-\lambda)^3}
\Big)
$ & $\mathcal{O}\Big(
\frac{d\kappa^2\sqrt{N}\varepsilon^{-2}}{\sqrt{K}} +\frac{d\kappa^2\varepsilon^{-2}}{(1-\lambda)^2} + d\sqrt{\frac{N}{K}}
\Big)$  \\Finite-sum
& \textsc{ZO-L2S-GT} (Corollary \ref{corollary:offline:L2S+atc-gt}) & $ 
\mathcal{O}\Big(
\frac{\kappa^2\sqrt{N}\varepsilon^{-2}}{\sqrt{K}}
+\frac{\kappa^2\epsilon^{-2}}{(1-\lambda)^3} + \frac{\kappa^2\sqrt{K}\varepsilon^{-2}}{\sqrt{N}(1-\lambda)^4}
\Big)
$ & $\mathcal{O}\Big(
\frac{d\kappa^2\sqrt{N}\varepsilon^{-2}}{\sqrt{K}} +\frac{d\kappa^2\varepsilon^{-2}}{(1-\lambda)^3} + d\sqrt{\frac{N}{K}}
\Big)$  \\
\midrule
\multicolumn{4}{c}{ZO-CW + ZO-RU$^{\scalebox{1.2}{$\star$}}$}
\\
\midrule
\rowcolor{sotagreen}
Stochastic
& \textsc{ZO-STORM-ED} (Corollary \ref{corollary:online:storm+ed:rd}) & $\mathcal{O}\Big(\frac{ d^{\frac{3}{4}}\kappa^3\varepsilon^{-3}}{K} + \frac{\kappa^2\varepsilon^{-2}}{(1-\lambda)^2} \Big)$ & $\mathcal{O}\Big(\frac{d^{\frac{3}{4}}\kappa^3\varepsilon^{-3}}{K} + \frac{\kappa^2\varepsilon^{-2}}{(1-\lambda)^2} \Big)$  \\ Stochastic
& \textsc{ZO-STORM-GT}  (Corollary \ref{corollary:online:atc+gt:rd})  & $\mathcal{O}\Big(\frac{ d^{\frac{3}{4}}\kappa^3\varepsilon^{-3}}{K} + \frac{\kappa^2\varepsilon^{-2}}{(1-\lambda)^3} \Big)$ & $\mathcal{O}\Big(\frac{d^{\frac{3}{4}}\kappa^3\varepsilon^{-3}}{K} + \frac{\kappa^2\varepsilon^{-2}}{(1-\lambda)^3}\Big)$  \\ 
\rowcolor{sotagreen}
Finite-sum$^{\scalebox{1.1}{$\dagger$}}$
& \textsc{ZO-PAGE-ED} (Corollary \ref{corollary:offline:page+ed:rd})  & $\mathcal{O}\Big(\frac{\sqrt{d}\kappa^2\varepsilon^{-2}}{(1-\lambda)^{1.5}}
+ \kappa^2\varepsilon^{-2} \Big)$ & $\mathcal{O}\Big(
\frac{d^{1-\frac{c}{2}}\kappa^2\sqrt{N}\varepsilon^{-2}}{\sqrt{K}(1-\lambda)^{1.5}} + d^{\frac{3-c}{2}}\sqrt{\frac{N}{K}}
\Big)$ \\ 
Finite-sum$^{\scalebox{1.1}{$\dagger$}}$
& \textsc{ZO-PAGE-GT} (Corollary \ref{corollary:offline:page+atc-gt:rd}) & $\mathcal{O}\Big(\frac{\sqrt{d}\kappa^2\varepsilon^{-2}}{(1-\lambda)^{2}}
+ \kappa^2\varepsilon^{-2} \Big)$ & $\mathcal{O}\Big(
\frac{d^{1-\frac{c}{2}}\kappa^2\sqrt{N}\varepsilon^{-2}}{\sqrt{K}(1-\lambda)^{2}} + d^{\frac{3-c}{2}}\sqrt{\frac{N}{K}}
\Big)$ \\ 
\rowcolor{sotagreen}
Finite-sum$^{\scalebox{1.1}{$\dagger$}}$ 
& \textsc{ZO-L2S-ED}  (Corollary \ref{corollary:offline:L2S+ed:rd}) & $ 
\mathcal{O}\Big(
\frac{d^{1-\frac{c}{2}}\kappa^2\sqrt{N}\varepsilon^{-2}}{K^{\frac{c+1}{2}}} 
+\frac{\kappa^2\epsilon^{-2}}{(1-\lambda)^2}+ \frac{\kappa^2K^{\frac{c+1}{2}} \varepsilon^{-2}}{\sqrt{N}(1-\lambda)^3}
\Big)$ & $\mathcal{O}\Big(
\frac{d^{1-\frac{c}{2}}\kappa^2\sqrt{N}\varepsilon^{-2}}{K^{\frac{c+1}{2}}} +\frac{\kappa^2\varepsilon^{-2}}{(1-\lambda)^2} + \frac{\kappa^2K^{\frac{c+1}{2}} \varepsilon^{-2}}{\sqrt{N}(1-\lambda)^3}
\Big)$  \\ Finite-sum$^{\scalebox{1.1}{$\dagger$}}$
& \textsc{ZO-L2S-GT} (Corollary \ref{corollary:offline:L2S+atc-gt:rd})  & $ 
\mathcal{O}\Big(
\frac{d^{1-\frac{c}{2}}\kappa^2\sqrt{N}\varepsilon^{-2}}{K^{\frac{c+1}{2}}} 
+\frac{\kappa^2\epsilon^{-2}}{(1-\lambda)^3}+ \frac{\kappa^2K^{\frac{c+1}{2}} \varepsilon^{-2}}{\sqrt{N}(1-\lambda)^4}
\Big)$ & $\mathcal{O}\Big(
\frac{d^{1-\frac{c}{2}}\kappa^2\sqrt{N}\varepsilon^{-2}}{K^{\frac{c+1}{2}}} +\frac{\kappa^2\varepsilon^{-2}}{(1-\lambda)^3}+ \frac{\kappa^2K^{\frac{c+1}{2}} \varepsilon^{-2}}{\sqrt{N}(1-\lambda)^4}
\Big)$  \\
\bottomrule
\end{tabular}%

\begin{tablenotes}[flushleft]
\footnotesize
\item \textbf{Notes:}
$^{\scalebox{1.0}{$\dagger$}}$ 
We demonstrate guarantees of hybrid ZO methods under two regimes: i)  $N\le \tilde{\mathcal{O}}(\varepsilon^{-2})$ in which case $c=0$,
and ii) $N\ge \tilde{\mathcal{O}}(\varepsilon^{-2})$ in which case $c=1$. The latter regime implies that an improved complexity is attainable when $N$ is large.
$^{\scalebox{1.2}{$\star$}}$ ZO-CW: pure coordinate ZO estimator is adopted; ZO-CW+ZO-RU: hybrid ZO estimators.
\end{tablenotes}
\end{table*}

\textbf{Q1:}{ \em 
Can we design a general decentralized ZO minimax framework that accommodates both stochastic and finite-sum settings and is resilient to the loss of information stemming from the use of a ZO estimator?}

\textbf{Q2:}{ \em 
  Can we establish unified performance guarantees for such a framework, thereby enabling quantitative comparisons across different special cases?}

\textbf{Q3:}{ \em 
 Can we identify algorithmic instances that simultaneously achieve optimal communication and function query efficiency, among others?
If so, can their function query efficiency match that of centralized methods?}

Answering these questions is highly nontrivial due to both algorithmic and theoretical challenges, as we highlight next.

\subsection{Novelty and challenges}
Compared with existing centralized ZO and decentralized FO minimax algorithms, we encounter challenges on both the algorithmic and theoretical sides:

$\bullet$ 
To the best of our knowledge, existing ZO minimax methods primarily focus on the centralized setting~\cite{liu2020min,wang2023zeroth,huang2022accelerated,xu2020gradient,xu2023zeroth,an2024robust}. In such settings, the main goal is to improve function-query complexity, while communication efficiency is not part of the algorithmic design. In contrast, decentralized ZO minimax optimization requires information exchange over a network, making the balance between communication efficiency and function-query efficiency a key issue.
Moreover, function heterogeneity among agents becomes more difficult to handle in the decentralized ZO setting due to information loss from gradient approximation. To address these challenges, we propose a unified framework equipped with a bias-correction technique and a flexible variance-reduction scheme tailored to the ZO setting; see subsections \ref{main:subsection: accelerated_ZO} and \ref{main:subsection:unified}. The analysis further requires a transformed recursion that is absent from the centralized ZO literature; see Appendix~\ref{appendix:sec:transformation}.
Moving from FO minimax methods to ZO minimax methods introduces additional difficulties ~\cite{gao2022decentralized,chen2024efficient,zhang2024jointly,mancino2023variance,cai2025dama1}: ZO gradient estimators often suffer from high variance, which can degrade performance even in regimes where simple decentralized gradient methods are effective, necessitating efficient variance-reduction design.  The choice of the ZO estimator is also crucial, as different estimators can lead to distinct trade-offs between function-query complexity and approximation accuracy. We exploit the complementary strengths of coordinate-wise and randomized smoothing estimators to balance efficiency and convergence.

$\bullet$ The theoretical analysis is nontrivial due to the combined effects of weaker assumptions, decentralization, variance reduction, and the use of multiple ZO estimators. Compared with existing non-variance-reduced ZO methods~\cite{xu2023unified,an2024robust}, we adopt several acceleration schemes and need to control gradient approximations at both the averaged and network levels. To address this, we exploit the unified structure of these acceleration schemes and establish Lemmas~\ref{appendix:lemma:coorinate_variance} and~\ref{appendix:lemma:coorinate_Nxc_Nyc}. Moreover, unlike several existing ZO works that assume bounded FO gradients~\cite{liu2020min}, bounded ZO variance~\cite{an2024robust,huang2022accelerated, xu2023zeroth}, Lipschitz component gradients~\cite{huang2022accelerated}, or convex component functions~\cite{xu2020gradient}, our analysis is carried out under substantially weaker conditions. As a result, we need to control expected differences between ZO gradient estimators. We address this difficulty by exploiting a mean-value representation of the gradient difference; see Lemma~\ref{appendix:lemma:random:exey}. We also need to handle gradient-dependent ZO variance, which is controlled by relating it to the target gradient; see  Lemmas~\ref{appendix:lemma:coorinate_Nxc_Nyc}--\ref{appendix:lemma:random:exey}. These difficulties are distinct from those in decentralized FO minimax methods~\cite{gao2022decentralized,chen2024efficient,zhang2024jointly,mancino2023variance,cai2025dama1}, where gradient differences can often be controlled directly. Furthermore, our hybrid ZO framework employs multiple estimators that approximate different FO surrogate gradients, requiring a careful error decomposition to relate the resulting estimation errors; we tackle these challenges in Lemmas~\ref{appendix:lemma:cw_coupled} and~\ref{appendix:lemma:random:exey}.

\subsection{Contributions}

To address the aforementioned gaps,
we develop a unified \textbf{ZO} decentralized accelerated \textbf{M}inim\textbf{A}x framework, termed \textbf{ZOMA}.
Our main contributions are listed below.

$\bullet$ Our general framework \textbf{ZOMA} provides a broad unification, along several levels: 
(i) \emph{estimator} - it adopts a hybrid ZO estimator, which accommodates, among others, both coordinate-wise (CW) and randomized uniform (RU) smoothing estimators;
 (ii) \emph{bias correction} - it subsumes a wide range of bias-correction strategies, including GT \cite{nedic2017achieving}, 
 exact diffusion (ED) \cite{yuan2018exact}, and EXTRA \cite{shi2015extra} and (iii) \emph{acceleration} - it facilitates a broad class of acceleration techniques, including zeroth-order versions of STORM \cite{cutkosky2019momentum}, PAGE \cite{li2021page} and L2S \cite{li2020convergence}.
% Jointly unifying these components is highly nontrivial, as it requires meticulous co-design of both algorithmic and analytical frameworks.
To the best of our knowledge, there are no prior works on decentralized ZO  minimax optimization, making our work the first ZO framework for decentralized minimax optimization. 

$\bullet$ We establish unified performance guarantees for \textbf{ZOMA} under two ZO design choices: (i) a pure ZO coordinate-wise (ZO-CW) strategy, and (ii) a hybrid strategy that combines ZO-CW and ZO randomized uniform (ZO-RU) smoothing technique. These designs
enable a trade-off between estimation accuracy of the ZO  gradient estimator and function query efficiency.
In Table~\ref{tab:zo-minimax-unified-comparison}, we summarize the complexity results of special cases of \textbf{ZOMA} and compare them with existing centralized benchmarks, notably showing that our decentralized methods perform on par with (or better than) their centralized counterparts, while bringing benefits like linear speed-up.
In particular,
the coordinate-wise variant of ZO-STORM-ED/GT matches the centralized complexity $\mathcal{O}(d\kappa^3\epsilon^{-3})$ achieved by \cite{xu2020gradient} (see Table~\ref{tab:zo-minimax-unified-comparison} for definition of $\kappa$ and $d$), while additionally attaining linear speedup with respect to the number of agents $K$. 
The hybrid ZO variant 
 of ZO-STORM-ED/GT
 further improves this result, reducing the complexity by a factor of $\mathcal{O}(d^{\frac{1}{4}})$.
Compared with Acc-ZOMDA proposed by \cite{huang2022accelerated}, our hybrid method removes the logarithmic dependency and improves the condition number dependence by a factor of $\mathcal{O}(\kappa^{1.5})$.
 
$\bullet$ Our theoretical analysis provides several insights into the relative strengths of different algorithmic instances under various regimes.
First, ZO-ED-based algorithms consistently outperform their ZO-GT-based counterparts over sparse networks, indicating that the ED strategy is advantageous in gradient-free settings.
Second, our results reveal an inherent trade-off between communication and ZO oracle complexities, {\em as no single algorithm simultaneously achieves the best performance in both metrics}. Third, by leveraging the finite-sum structure,  we establish guarantees that depend on the sample size $N$, which can outperform stochastic guarantees when $N$ falls into certain regimes (see subsection \ref
{main:subsection:discussion} for details). 

\textbf{Paper organization.}
The remainder of the paper is organized as follows.
Section \ref{main:section:preliminaries} reviews ZO estimators and states the assumptions.
Section \ref{main:section:development} presents the \textbf{ZOMA} framework.
Section \ref{main:section:development} establishes convergence guarantees.
Section \ref{main:simulation} provides simulations.
All proofs are deferred to the appendix.

\textbf{Notation.}
We use normal fonts (e.g., $x$) to denote deterministic quantities, and bold fonts (e.g., $\bx$) to denote stochastic quantities. The symbol 
$\mathcal{O}(\cdot)$ denotes the standard ``big O", while $\tilde{\mathcal{O}}(\cdot)$ hides all hyperparameters except $\varepsilon$. Calligraphic symbol denotes network concatenated quantities, e.g.,
$\mX_i \triangleq \mathrm{col}\{\bx_{1,i}, \ldots, \bx_{K,i}\}
\triangleq {\rm col}\{\bx_{k,i}\}_{k=1}^K \in \mR^{Kd_1}$,
where $\bx_{k,i} \in \mR^{d_1}$ represents the variable of agent $k$ at communication round $i$. We use
$\mathbb{B}^{d}$ to denote the $d$-dimensional {\em unit} ball and 
$\mathbb{S}^{d-1}$ to denote the sphere of the unit ball in $\mathbb{R}^{d}$.
$\mathcal{N}_k$ denotes the neighboring set of agent $k$ (itself included). $\mathbb{I}(\cdot) \in \{0, 1\}$ is an indicator function. $\otimes$ stands for the Kronecker product.
$\mathbf{I}_{K}$ and $\mathds{1}_{K}$ denote the $K$-dimensional identity matrix and vector of all ones, respectively.

\section{Preliminaries}
\label{main:section:preliminaries}
\subsection{Stochastic ZO Gradient Estimator}
\label{main:subsection:ZOestimator}
We review the standard ZO gradient estimator by
considering a loss function $Q(w;\bxi):\mR^d \rightarrow \mR$.

$\bullet$ \textbf{ZO coordinate-wise (ZO-CW) estimator}

The ZO-CW strategy constructs a gradient estimator by approximating each partial derivative via finite differences, using function evaluations at points perturbed along the canonical basis directions. Specifically, it queries $2d$ perturbed 
function values, i.e.,
$
Q(w + \delta e_j;\bxi)$ and $ Q(w - \delta e_j;\bxi), j \in \{1,\ldots,d\}$,
where $\delta > 0$ is a smoothing parameter and
$e_j \in \mathbb{R}^{d}$ denotes the $j$-th canonical basis vector. The  estimator is given by 
\begin{align}
&\sum_{j=1}^{d}
\Big[\Big(Q(w + \delta e_j;\bxi) - Q(w-\delta e_j;\bxi)\Big)/(2\delta)\Big]e_j.
\end{align}
\vspace{-0.5em}

$\bullet$ \textbf{ZO randomized uniform (ZO-RU) estimator}

A more query-efficient alternative is the ZO-RU estimator, which constructs a gradient estimator by perturbing the argument along a single randomly sampled direction. Let
$\bu \sim \mathrm{Unif}(\mathbb{S}^{d-1})$
be uniformly drawn from the unit sphere. Using only two function evaluations, namely
$Q(w+\mu \bu; \bxi)$ and $Q(w;\bxi)$, the ZO-RU estimator is given by
\begin{align}
d\Big[\Big(Q(w +\mu \bu;\bxi) -Q(w;\bxi)\Big)/\mu\Big]\bu,
\end{align}
where $\mu>0$
is a smoothing parameter.
This estimator approximates the gradient of the smoothed objective
$ Q_{\mu}(w,\bxi) \triangleq \mathbb{E}_{\bu \sim \mathrm{Unif}(\mathbb{B}^d)}\left[Q(w + \mu \bu;\bxi)\right]$ \cite{liu2020primer}.
Compared with the ZO-CW estimator, ZO-RU requires only two function queries per construction and is therefore more flexible and query-efficient,  at the expense of reduced gradient estimation accuracy.
\vspace{-0.8em}

\subsection{ZO Gradient Estimator for Minimax Variables}

We now extend the above estimators to the minimax setting. Consider a local loss function $Q_k(x,y;\bxi): \mR^{d_1+d_2} \rightarrow \mR$, where $x\in \mR^{d_1},  y \in \mR^{d_2}$. For compactness, we define the concatenated vector $z \triangleq {\rm col}\{x, y\}$.
By fixing one variable and applying either the ZO-CW or ZO-RU strategy to the other, we construct ZO estimators of the partial gradients with respect to
$x$ and $y$ as follows\footnote{
We use $\bu \not =0$ to denote the ZO-RU estimator, and $\bu =0$ to denote the ZO-CW estimator, since $\bu =0$ cannot be sampled from the unit sphere.}
\begin{align}
&\bq^x_{k}(z;\bxi, \bu) \triangleq 
\label{main:preliminary:ZO_coordinateX}
\\
&\begin{cases}
\displaystyle
\sum_{j=1}^{d_1}
\frac{Q_k(x + \delta_x e_j, y;\bxi) - Q_k(x-\delta_x e_j, y;\bxi)}{2\delta_x}e_j,   \text{ if } \bu =0  \\
\displaystyle
\frac{d_1(Q_k(x +\mu_x \bu,y;\bxi) -Q_k(z;\bxi))}{\mu_x}\bu, \qquad  \qquad  \
\text{ if } \bu \not=0.
\end{cases} \notag
\end{align}
\begin{align}
& \bq^y_{k}(z;\bxi, \bu)  \triangleq \label{main:preliminary:ZO_randomY} \\
&  \begin{cases}
\displaystyle
    \sum_{j=1}^{d_2}
\frac{Q_k(x , y+ \delta_y e_j;\bxi) - Q_k(x, y-\delta_y e_j;\bxi)}{2\delta_y}e_j ,   \text{ if } \bu =0
\label{main:preliminary:ZO_coordinateY} \\ 
\displaystyle\frac{d_2(Q_k(x, y +\mu_y \bu; \bxi) -Q_k(z;\bxi))}{\mu_y}\bu,
\qquad  \quad \quad
\text{ if } \bu \not=0.
\end{cases}
\notag
\end{align}
% \begin{align}
% &(\textbf{ZO-CW}) \notag \\
% &\bq^x_{i,0}(z; \bxi) 
%  \notag \\
% &\triangleq\sum_{j=1}^{d_1}
% \frac{Q_k(x + \delta_x e_j, y;\bxi) - Q_k(x-\delta_x e_j, y;\bxi)}{2\delta_x}e_j,  \label{main:preliminary:ZO_coordinateX} \\
% & \bq^y_{k}(z; \bxi)
%   \notag \\
%  &\triangleq\sum_{j=1}^{d_2}
% \frac{Q_k(x , y+ \delta_y e_j;\bxi) - Q_k(x, y-\delta_y e_j;\bxi)}{2\delta_y}e_j,  
% \label{main:preliminary:ZO_coordinateY}
% \\
% &(\textbf{ZO-RU}) \notag \\
% & \bq^x_{i,0}(z;\bu, \bxi) \triangleq \frac{d_1(Q_k(x +\mu_x \bu,y;\bxi) -Q_k(z;\bxi))}{\mu_x}\bu,  
% \label{main:preliminary:ZO_randomX}
% \\
% & \bq^y_{k}(z;\bu, \bxi)  \triangleq  \frac{d_2(Q_k(x, y +\mu_y \bu; \bxi) -Q_k(z;\bxi))}{\mu_y}\bu,
% \label{main:preliminary:ZO_randomY}
% \end{align}  
where $\delta_x, \delta_y, \mu_x, \mu_y >0$ are smoothing
parameters,
and if $\bu \not =0$, then either $\bu \sim \mathrm{Unif}(\mathbb{S}^{d_1-1})$ or $\bu \sim \mathrm{Unif}(\mathbb{S}^{d_2-1})$,
depending on the variable being perturbed.

% \textbf{Bottleneck of the stochastic ZO Gradient}.
% \label{main:subsection:bottleneck}
% Although the ZO estimators in \eqref{main:preliminary:ZO_coordinateX}–\eqref{main:preliminary:ZO_randomY} 
% are straightforward to use, they
% exhibit higher estimation variance than their FO counterparts. 
% This is further exacerbated by high dimensionality and decentralized network settings, posing a major challenge for efficient minimax optimization. To address this issue, in the next section we introduce an efficient probabilistic loopless variance reduction technique. Before that, we present some standard assumptions used in our work.

\subsection{Assumptions}
We next 
introduce some preliminary assumptions and our optimization criterion. 
Compared with existing ZO minimax works \cite{liu2020min,wang2023zeroth,huang2022accelerated,xu2020gradient,an2024robust}, our assumptions are notably weaker. 
We
begin by defining an $\varepsilon$-game stationary point.
\begin{definition}[\textbf{$\varepsilon$-game stationarity}]
The point $(\bx^\star,\by^\star)$ is called an $\varepsilon$-game stationary point of problem \eqref{main:sec1:problem} if 
$
\mE\|\nabla_x J(\bx^\star, \by^\star)\|^2 \le \varepsilon^2, \mE\|\nabla_y J(\bx^\star, \by^\star)\|^2 \le \varepsilon^2.$
\end{definition}

To establish guarantees for an $\varepsilon$-game stationary point,
we adopt the following assumptions. 

% \cite{nouiehed2019solving,yang2022faster, yang2020global, gao2022decentralized,luo2020stochastic,xian2021faster,huang2022accelerated,huang2025enhanced,cai2025dama1,cai2025dama2}

\begin{Assumption}
\label{main:assumption:costfunction}
The global cost function $J(x, y)$
is nonconvex in $x$ and nonconcave in $y$, satisfying a $\nu$-PL condition, i.e.,
\begin{align}
\|\nabla_y J(x,y)\|^2 \ge 2\nu\Big(\max_{z} J(x,z) - J(x,y)\Big), \ \forall y \in \mR^{d_2}.
\notag
\end{align}
Here, $\nu >0$ is a positive constant. In addition, we assume the value function $P(x) = \max_y J(x,y)$ is lower bounded, i.e., 
$ P^\star= 
\inf_x P(x) > -\infty \notag
$. 
\end{Assumption}
This assumption is consistent with several FO minimax works~\cite{nouiehed2019solving,yang2022faster,cai2025dama1,cai2025dama2}, while being weaker than the assumptions adopted in several existing ZO works~\cite{liu2020min,wang2023zeroth,huang2022accelerated} as we do not need a strong concavity condition on $y$. Furthermore, the PL condition has been shown to hold for over-parameterized neural networks \cite{liu2022loss}, and has recently emerged as a key structural assumption enabling convergence guarantees for nonconvex–nonconcave minimax problems without requiring strong concavity \cite{nouiehed2019solving,yang2020global,yang2022faster}.

\begin{Assumption}
\label{main:assumption:Lipschitz}
The local loss function 
$Q_k(x,y;\bxi_k)$ is continuously differentiable and expected $L_f$-smooth. Specifically, for all $w\in\{x, y\}$, we have 
\begin{align}
&\mE\|\nabla_w Q_k(x,y;\bxi_k) - \nabla_w Q_k(x^\prime,y^\prime;\bxi_k)\|^2
\notag \\
&\le L^2_f(\|x-x^\prime\|^2+\|y-y^\prime\|^2).
\end{align}
The condition number is defined as $\kappa \triangleq L_f/\nu$.
\end{Assumption}
The above assumption aligns with those commonly used to obtain improved rates in FO optimization \cite{gao2022decentralized,luo2020stochastic,xian2021faster,huang2022accelerated,huang2025enhanced,cai2025dama2}, while being weaker than the assumptions imposed in existing ZO methods \cite{huang2022accelerated,xu2023zeroth}, which require smoothness of each component gradient rather than only smoothness in expectation. As a consequence,  bounding the gradient difference is substantially more difficult without assuming convexity on each component
function \cite{xu2020gradient};  see proof of Lemma \ref{appendix:lemma:coorinate_Nxc_Nyc}.

\begin{Assumption}
\label{main:assumption:boundedvariance}
The local loss gradients $\nabla_w Q_k(x,y;\bxi_k)$
are unbiased and have bounded variance, i.e., we have for all $w \in \{x, y\}$
\begin{align}
 &\mE_{\bxi_k}[\nabla_w Q_k(x, y;\bxi_k)] = \nabla_w J_k(x,y),  \\ &\mE_{\bxi_k} \|\nabla_w Q_k(x, y;\bxi_k)- \nabla_w J_k(x,y)\|^2 \le \sigma^2. 
\end{align}
The random samples $\{\bxi_k\}$ 
are assumed to be spatially and temporally independent and identically distributed (i.i.d.).
\end{Assumption}
Assumption \ref{main:assumption:boundedvariance} is standard
in stochastic FO minimax optimization \cite{lin2020gradient,yang2022faster,xian2021faster}.

\section{Development of ZOMA}
\label{main:section:development}
\subsection{Accelerating ZO Gradient Estimator}
\label{main:subsection: accelerated_ZO}
To address the high variance of ZO gradient estimators, we develop a loopless probabilistic variance-reduced ZO  estimator. When the ZO gradient estimator is replaced by its exact FO counterpart, the proposed framework recovers several accelerated methods as special cases, including STORM~\cite{cutkosky2019momentum}, PAGE~\cite{li2021page}, and L2S~\cite{li2020convergence}.

Let $\bg^w_{k,i}$, with $w \in\{ {x,y}\}$, denote the ZO gradient estimator of the variable $w$ at agent $k$ and communication round $i$.
At each round,
each agent starts by generating a Bernoulli random variable $\bpi_i \sim \mathrm{Bernoulli}(p)$ using a shared random seed.
Here, $p \in [0,1]$ is the probability of the event $\bpi_i=1$.
Based on the realization of $\bpi_i$,  each agent computes $\bg^w_{k,i}$ as
\begin{align}
\label{main:zero_GRACE} &\bg^w_{k,i}   = 
\begin{cases}
\frac{1}{B}\sum_{j=1}^{B} \bq^w_{k}(\bz_{k,i};\bxi^j_{k,i},0)  \hspace{5.5em} (\text{if } \bpi_i =1)\\
\\
(1-\beta)
\Big[\bg^w_{k,i-1} 
- 
\frac{1}{b}\sum_{j=1}^b
\bq^w_{k}(\bz_{k,i-1};\bxi^j_{k,i}, \bu^j_{k,i})
\Big] \\
+\frac{1}{b}\sum_{j=1}^b
\bq^w_{k}(\bz_{k,i};\bxi^j_{k,i}, \bu^j_{k,i}) \hspace{3.5em} (\text{if } \bpi_i =0),
\end{cases}
\end{align}
where $B$ and $b$ denote the large- and mini-batch sizes, respectively; $\beta \in [0,1]$ is a momentum smoothing parameter; and $\bu^j_{k,i} = 0$ or $\bu^j_{k,i} \sim \mathrm{Unif}(\mathbb{S}^{d_1-1})$ (resp., $ \mathrm{Unif}(\mathbb{S}^{d_2-1})$) when $\bpi_i=0$.
When the event $\bpi_i = 1$ occurs,
we focus on the ZO-CW strategy, as this event is intended to produce an occasional high-quality gradient estimator.
Otherwise, if $\bpi_i$ = 0, a momentum-style update is performed, using either the ZO-CW or ZO-RU estimator.
The resulting probabilistic ZO gradient estimator is simple and efficient as it avoids nested loops and supports seamless switching between large-batch variance reduction and mini-batch momentum-based updates, enabling us to obtain strong performance in both stochastic and finite-sum settings.
By tuning $(p,\beta,B,b)$, we obtain several special cases of \eqref{main:zero_GRACE}, including the following: 1) ZO STORM:  $p= 0, b = \mathcal{O}(1), \beta \in (0,1)$, 2) ZO PAGE: 
$p\in(0,1), B=\mathcal{O}(N), b \approx \mathcal{O}(\sqrt{B}), \beta = 0$,
3) ZO L2S: 
$p\in(0,1), B = \mathcal{O}(N), b  = \mathcal{O}(1), \beta = 0$, 4) stochastic ZO-CW GDA: $p = 0, b =\mathcal{O}(1), \beta = 1, \bu^j_{k,i} =0$, 5) stochastic ZO-RU GDA: $p = 0, b =\mathcal{O}(1), \beta = 1, \bu^j_{k,i} \not=0$.
% \begin{itemize}
% \item ZO STORM:  $p= 0, b = \mathcal{O}(1), \beta \in (0,1)$,
% \item ZO PAGE: 
% $p\in(0,1), B=\mathcal{O}(N), b \approx \mathcal{O}(\sqrt{B}), \beta = 0$, 
% \item ZO L2S: 
% $p\in(0,1), B = \mathcal{O}(N), b  = \mathcal{O}(1), \beta = 0$,
% \item 
% stochastic ZO-CW GDA: $p = 0, b =\mathcal{O}(1), \beta = 1, \bu^j_{k,i} =0$,
% \item 
% stochastic ZO-RU GDA: $p = 0, b =\mathcal{O}(1), \beta = 1, \bu^j_{k,i} \not=0$.
% \end{itemize}

We note that several of the above variants have not been previously studied in the ZO minimax literature, making our unified estimator novel even in the centralized minimax setting.
For ease of reference, we refer to the unified formulation \eqref{main:zero_GRACE} as \textbf{ZO-GRACE} (\textbf{ZO} \textbf{GR}adient \textbf{AC}celeration \textbf{E}stimator).

\subsection{Description of ZOMA}
\label{main:subsection:unified}
\begin{algorithm}[!t]
\footnotesize
%\tiny
\caption{\hspace{-3.5pt}: \textbf{ZOMA} (Zero-Order decentralized accelerated Minimax Algorithm)}
\label{alg:sample}
\begin{algorithmic}[1]
   \STATE 
   \textbf{Initialize}: $\bx_{k,0}= \bx_{k,-1} \in \mR^{d_1} , \by_{k,0} = \by_{k,-1} \in \mR^{d_2}, \bm^x_{k,-1}= \bg^x_{k,-1}=0 ,\bm^y_{k,-1}=\bg^y_{k,-1}=0 \ \forall k \in [K]$, learning rates $\eta_x, \eta_y$, hyperparameters $p, \mu_{x}, \mu_{y}, \delta_x, \delta_y,\beta,b,B,$ $b_0 (\text{initial batch size}).$
    \FOR{$i = 0, \dots, T-1$}
    \FOR{agent $k$ in parallel} 
    \STATE 
\underline{\texttt{Stage 1: Construct ZO gradients}} \vspace{0.5em}
    \IF{$i=0$}
    \STATE
    Draw a $b_0$-minibatch of i.i.d. samples  and compute
$
\bg^w_{k, 0}
= \frac{1}{b_0}\sum_{j=1}^{b_0} \bq^w_{k}(\bz_{k,0};\bxi^j_{k,0},0), \ \forall w \in \{x,y\}$.
\ELSE 
\STATE
    Draw  $\bpi_{i} \sim \operatorname{Bernoulli}(p)$ using a shared random seed.
    If $\bpi_i=1$ in the \textbf{finite-sum} setting, uniformly sample
 $\{\bxi^j_{k,i}\}$ without replacement; otherwise, draw i.i.d. samples.
   The gradient estimators $\bg^x_{k,i}, \bg^y_{k,i}$ are then updated via \eqref{main:zero_GRACE}.
\ENDIF 
\STATE 
\underline{\texttt{Stage 2: Decentralized learning}}
\vspace{1em}
\STATE 
Select one of the following options\\

(\textbf{Option 1: ED}) 
\begin{align}
     &\begin{cases}
    \bx_{k,i+1} &= \sum_{\ell \in \mathcal{N}_{k}} w_{k\ell}[2\bx_{\ell,i} - \bx_{\ell,i-1} - \eta_x(\bg^x_{\ell,i} - \bg^x_{\ell,i-1})], \notag\\
    \by_{k,i+1} &= \sum_{\ell \in \mathcal{N}_{k}} w_{k\ell}[2\by_{\ell,i} - \by_{\ell,i-1} + \eta_y(\bg^y_{\ell,i} - \bg^y_{\ell,i-1})].  \notag
    \end{cases}
\end{align}
\textbf{(Option 2: GT)} 
\begin{align}
\begin{cases}
    \bm^x_{k,i} &=  \sum_{\ell \in \mathcal{N}_{k}} w_{k\ell} [\bm^x_{\ell,i-1}-\bg^x_{\ell,i-1} +\bg^x_{\ell,i}] , \\
     \bm^y_{k,i} &=  \sum_{\ell \in \mathcal{N}_{k}} w_{k\ell} [\bm^y_{\ell,i-1}-\bg^y_{\ell,i-1} +\bg^y_{\ell,i}],
     \notag\\
     \bx_{k,i+1} &= \sum_{\ell \in \mathcal{N}_{k}} w_{k\ell}[\bx_{\ell,i} -\eta_x \bm^x_{\ell,i} ], \\
     \by_{k,i+1} &= \sum_{\ell \in \mathcal{N}_{k}} w_{k\ell}[\by_{\ell,i} +\eta_y \bm^y_{\ell,i} ].\notag
    \end{cases}
\end{align}
\textbf{(Option 3: EXTRA)} 
\begin{align}
\begin{cases}
    \bx_{k,i+1} &= \sum_{\ell \in \mathcal{N}_{k}} w_{k\ell}[2\bx_{\ell,i} - \bx_{\ell,i-1}] - \eta_x[\bg^x_{\ell,i} - \bg^x_{\ell,i-1}], \\
    \by_{k,i+1} &= \sum_{\ell \in \mathcal{N}_{k}} w_{k\ell}[2\by_{\ell,i} - \by_{\ell,i-1}] + \eta_y[\bg^y_{\ell,i} - \bg^y_{\ell,i-1}]. \notag
\end{cases}
\end{align}
\ENDFOR
\ENDFOR
\end{algorithmic}
\end{algorithm}

We now integrate \textbf{ZO-GRACE} with several decentralized optimization strategies, including GT, ED, and EXTRA, leading to our \textbf{ZOMA} framework.
For compactness, we define network block variables and network block gradient estimators as
$
\mX_i
\triangleq {\rm col}\{\bx_{k,i}\}_{k=1}^K,  \mG_{x,i} \triangleq {\rm col}\{\bg^x_{k,i}\}_{k=1}^K \in \mR^{Kd_1},$$
\mY_i
\triangleq {\rm col}\{\by_{k,i}\}_{k=1}^K, 
\mG_{y,i}
\triangleq {\rm col}\{\bg^y_{k,i}\}_{k=1}^K \in \mR^{Kd_2}.$ Following \cite{alghunaim2022unified}, 
we can write decentralized strategies in a unified form:

\begin{align}
\mX_{i+1} &= \mA_{x}[\mC_x\mX_{i} - \eta_x \mG_{x,i}]- \mathcal{B}_x\mD_{x,i},  \label{main:zodama_recursion1}\\
\mY_{i+1} &= \mA_{y}[\mC_y \mY_{i} + \eta_y \mG_{y,i}]- \mathcal{B}_y\mD_{y,i}, \label{main:zodama_recursion2}\\
\mD_{x,i+1} &= \mD_{x,i} + \mB_x \mX_{i+1}, \label{main:zodama_recursion3}\\
\mD_{y,i+1} &= \mD_{y,i} + \mB_y \mY_{i+1} \label{main:zodama_recursion4}.
\end{align}
Here, $\eta_x,\eta_y>0$ are learning rates;
$\mD_{x,i} \in \mR^{Kd_1}, \mD_{y,i} \in \mR^{Kd_2}$
are auxiliary block vectors; and
$\{\mA_x, \mB_x, \mC_x\} \in \mR^{Kd_1\times Kd_1}, \{\mA_y, \mB_y, \mC_y\} \in \mR^{Kd_2\times Kd_2}$
are design matrices satisfying certain structural properties; see Assumption \ref{main:assumptions:combinationmatrix} ahead.
By appropriately selecting $\{\mA_x, \mB_x, \mC_x, \mA_y, \mB_y, \mC_y\}$, 
we obtain several special cases of \textbf{ZOMA}, which are presented in \textbf{Algorithm}~\ref{alg:sample}.
% The specific choice of these decentralized strategies can be found in Appendix \ref{appendix:subsection:matrixchoice}. 
The algorithm begins with proper initialization and then iteratively executes lines~4–8 to construct the ZO gradients, followed by lines~9–10 to perform updates.
\begin{remark}
We note that all algorithms adopt the ZO-CW strategy to obtain a high-quality gradient at the initial iteration. For this reason, we will refer to cases with $p =0$ and $\bu^j_{k,i} \not =0$ when $\bpi_i =0$ as hybrid variants rather than pure ZO-RU methods.
We further note that $\{\mA_x, \mB_x, \mC_x, \mA_y, \mB_y, \mC_y\}$ can be selected according to Table~\ref{tab:matrix_choices} in Appendix~\ref{appendix:subsection:matrixchoice} to recover the special cases presented in \textbf{Algorithm}~\ref{alg:sample}.
\end{remark}
\section{Main results}
\label{main:section:results}
We present the main results under two ZO design choices: the pure ZO-CW estimator with $\bu^j_{k,i} \equiv 0$ and the hybrid estimator with $\bu^j_{k,i}  \not = 0$ when $\bpi_i=0$.
For convenience, let the network average 
be 
$
\bx_{c,i} \triangleq \frac{1}{K} \sum^{K}_{k=1}\bx_{k,i},   \by_{c,i} \triangleq \frac{1}{K} \sum^K_{k=1}\by_{k,i}$.
We introduce an assumption on gossiping matrices to handle the decentralized setting (see  \cite{alghunaim2022unified, sayed2022inference, cai2025dama2}).
\begin{Assumption}
\label{main:assumptions:combinationmatrix}
Consider a symmetric, doubly stochastic, and primitive combination matrix $W \in \mathbb{R}^{K \times K}$. 
The matrices $\mA_x, \mB_x, \mC_x$ and $\mA_y, \mB_y, \mC_y$ are assumed to be polynomial functions of $W \otimes \mathrm{I}_{d_1}$ and $W \otimes \mathrm{I}_{d_2}$, respectively, such that
i) $\mA_x, \mA_y, \mC_x, \mC_y$ are symmetric and doubly stochastic, and ii) $\mathrm{null}(\mB_x) = \mathrm{span}\{\mathds{1}_{K} \otimes \bar{x}\}, 
\mathrm{null}(\mB_y) = \mathrm{span}\{\mathds{1}_{K}\otimes \bar{y}\}$ for some nonzero vectors $\bar{x}\in \mR^{d_1}, \bar{y} \in \mR^{d_2}$.
\end{Assumption}

We can verify that the above assumption holds for the special cases considered in this work. For instance, a GT-based minimax algorithm admits the following selection:
$
\mA_x= W^2 \otimes \mathrm{I}_{d_1}, \ \mA_y = W^2 \otimes \mathrm{I}_{d_2}, \ \mC_x= \mathrm{I}_{Kd_1}, 
\mC_y=\mathrm{I}_{Kd_2},  \mB_x =(\mathrm{I}_{Kd_1} - W \otimes \mathrm{I}_{d_1}),  \mB_y =(\mathrm{I}_{Kd_2} - W \otimes \mathrm{I}_{d_2}). $
It is not difficult to verify that these matrices satisfy Assumption \ref{main:assumptions:combinationmatrix}.
For convenience,  let $\lambda$ denote the second-largest eigenvalue of $W$ in magnitude. The quantity $1-\lambda$ is widely known as the network spectral gap and characterizes the decaying rate of consensus error.

\subsection{Case: Pure ZO-CW Estimator}
In this subsection, we provide guarantees for the pure ZO-CW estimator.
\label{subsection:case1}
\begin{Theorem}
\label{main:theorem1}
Under Assumptions \ref{main:assumption:costfunction}-\ref{main:assumptions:combinationmatrix}, choosing appropriate hyperparameters,  \textbf{ZOMA} yields the following rate:
\begin{align}
&\frac{1}{T}
\sum_{i=0}^{T-1}
(\mE\|\nabla_x J(\bx_{c,i},\by_{c,i})\|^2 + \mE\|\nabla_y J(\bx_{c,i},\by_{c,i})\|^2)\notag \\
&\le \mathcal{O}
\Big(
\underbrace{\Pi'_0}_{\text{initial gap}} +\underbrace{\frac{\kappa^2\lambda^2_a\eta^2_y(p+\beta^2)\sigma^2}{bb_0K\bar{\beta}^2T(1-\rho)^2\underline{\lambda^2_b}}}_{\text{network noise}} + \underbrace{\Pi'_1 \sigma^2 \mathbb{I}(B<N) }_{\text{large-batch effect}}\notag\\
&\quad 
 \underbrace{+ \frac{\kappa^2\sigma^2}{b_0\bar{\beta}KT}}_{\text{initial noise}}+  \underbrace{\Pi'_2\sigma^2}_{\text{ZO momentum error}}
+\underbrace{\Pi'_3}_{\text{ZO bias}} \Big), 
\end{align}
\end{Theorem}
Here,  $\Pi'_1, \Pi'_2, \Pi'_3$
are terms characterizing different types of error and each depends on some of the hyperparameters $p, \mu_{x}, \mu_{y}, \delta_x, \delta_y,\beta,b,B, T$; we refer the reader to the full version of Theorem~\ref{main:theorem1} in Appendix \ref{appendix:subsection:theorem1} for further details.
The symbol $\mathbb{I}(B<N)$ is an indicator of stochastic or finite-sum settings; we use the convention $N = \infty$ in the stochastic case. Besides, $\bar{\beta}\triangleq p+\beta-p\beta \le 1$ captures the impact of the probability parameter $p$ and the smoothing factor $\beta$;  the quantities $(1-\rho)$ and $\underline{\lambda_b^2}$ are factors depending on $\lambda$, and both approach zero when the network becomes sparse; see Lemma \ref{appendix:lemma:transformed_recursion} for explicit forms.
The large-batch error term
$\Pi'_1$
and momentum error term
$\Pi'_2$
vanish by simultaneously setting full-batch $B = N$ and $\beta = 0$ in the finite-sum setting.
In the stochastic setting, $\Pi'_1$ and $\Pi'_2$ can be controlled by setting $p =0$ and small $\beta$, respectively.
The ZO bias term $\Pi'_3$
can be controlled by choosing sufficiently small smoothing factors $\mu_x, \mu_y,\delta_x, \delta_y$.
We can specialize the unified bound to several representative cases.
For instance, choosing hyperparameters according to
Corollaries~\ref{corollary:online:storm+ed},
\ref{corollary:offline:page+ed},
\ref{corollary:offline:L2S+ed} in Appendix \ref{appendix:corollary:case1}, yields
$\mathcal{O}\!\Big(\frac{\kappa^2}{(T K)^{2/3}} + \frac{\kappa^2}{T(1-\lambda)^2}\Big)$
for ZO-STORM-ED,
$\mathcal{O}\!\Big(\frac{\kappa^2}{(1-\lambda)^{1.5} T}\Big)$
for ZO-PAGE-ED, and
$\mathcal{O}\!\Big(\frac{\kappa^2 \sqrt{N}}{\sqrt{K} T} + \frac{\kappa^2}{T(1-\lambda)^2} +\frac{\kappa^2\lambda^2\sqrt{K}}{\sqrt{N}T(1-\lambda)^3}\Big)$ for ZO-L2S-ED; implying the communication complexity demonstrated in Table \ref{tab:zo-minimax-unified-comparison}.

\subsection{Case: ZO-CW+ZO-RU Estimator}

\label{subsection:case2}
To bound the variance of the stochastic ZO-RU gradient,
we introduce an additional assumption 
that can be seen as an extension of the weak bounded similarity condition \cite{yi2022zeroth, chen2025zeroth}. Specifically, this assumption is weaker than the bounded ZO variance condition used in existing works \cite{huang2022accelerated,an2024robust}.

\begin{Assumption}
\label{main:asssumption:RD-ZO}
Consider the local smoothed risk functions $J^{\mu_x}_{k}(x,y) \triangleq \mE_{\bu \sim {\rm Unif}(\mathbb{B}^{d_1})} J_k(x+\mu_x \bu, y) $ and $J^{\mu_y}_{k}(x,y) \triangleq \mE_{\bu \sim {\rm Unif}(\mathbb{B}^{d_2})} J_k(x, y+\mu_y \bu)$.
We assume $\forall w \in \{x,y\}$ that
\begin{align}
&\mE_{\bxi,\bu}
\|\bq^w_{k}(z;\bu,\bxi) -
\nabla_w J^{\mu_w}_{k}(x,y)\|^2 \le 
\sigma^2_1 \|\nabla_w J(z)\|^2 + \sigma^2_0 .  \notag
\end{align}
\end{Assumption}
Unlike ZO-CW, ZO-RU requires this additional assumption because it implicitly approximates the true gradient of the associated smoothed function, as discussed in subsection \ref{main:subsection:ZOestimator}.
\begin{Theorem}
\label{main:theorem2}
Under Assumptions \ref{main:assumption:costfunction}-\ref{main:asssumption:RD-ZO}, choosing appropriate hyperparameters,  \textbf{ZOMA} yields the following rate
\begin{align}
& \frac{1}{T}
\sum_{i=0}^{T-1}
(\mE\|\nabla_x J(\bx_{c,i},\by_{c,i})\|^2 + \mE\|\nabla_y J(\bx_{c,i},\by_{c,i})\|^2)  \notag \\
&\le\mathcal{O}
\Bigg(
\underbrace{\Pi_0}_{\text{initial gap}} + \underbrace{\frac{d\kappa^2\lambda^2_a\eta^2_y(p+\beta^2)\sigma^2}{bb_0K\bar{\beta}^2T(1-\rho)^2\underline{\lambda^2_b}}}_{\text{network noise}} + \underbrace{\Pi_1 \sigma^2 \mathbb{I}(B<N)}_{\text{large-batch effect}}
\notag\\
&\quad+ \underbrace{\frac{\kappa^2\sigma^2}{b_0K\bar{\beta}T}}_{\text{initial noise}}
+ \underbrace{\Pi_2 \sigma^2_0}_{\text{ZO momentum error}} 
+ \underbrace{\Pi_3}_{\text{ZO bias}}
\Bigg).
\end{align}
\end{Theorem}
Here, $\Pi_1, \Pi_2, \Pi_3$
are terms that capture different sources of error and can be controlled through appropriate hyperparameter choices, in the same spirit as those discussed in  Theorem~\ref{main:theorem1}; we refer the reader to the full version of Theorem~\ref{main:theorem2} in Appendix \ref{appendix:subsection:theorem2} for further details.
Building on this unified  bound,
we can show that the convergence rates for ZO-STORM-ED, ZO-PAGE-ED, and ZO-L2S-ED are given by $\mathcal{O}\Big(\frac{\sqrt{d}\kappa^2}{(TK)^{2/3}} + \frac{\kappa^2}{T(1-\lambda)^2}\Big)$,
$\mathcal{O}\Big(
\frac{\sqrt{d}\kappa^2}{T(1-\lambda)^{1.5}}
\Big)$, $\mathcal{O}\Big(
\frac{d^{1-\frac{c}{2}}\kappa^2\sqrt{N}}{K^{\frac{c+1}{2}}T} + \frac{\kappa^2}{T(1-\lambda)^2}  +\frac{\kappa^2\lambda^2K^{\frac{c+1}{2}}}{\sqrt{N}T(1-\lambda)^3}
\Big)$, respectively,
where $c \in\{0, 1\}$, depending on the size of $N$; see Corollaries \ref{corollary:online:storm+ed:rd}, \ref{corollary:offline:page+ed:rd}, \ref{corollary:offline:L2S+ed:rd} in Appendix \ref{appendix:corollary:case2}.
Several of these rates exhibit explicit $\mathcal{O}(\sqrt{d})$ dependence, leading to a slower convergence rate than the former case.
Nonetheless, these methods achieve improved ZO oracle efficiency 
compared to the former case
as they typically require $\mathcal{O}(1)$ function queries when $\bpi_i =0$.

\subsection{Discussion and Technical Insights}
\label{main:subsection:discussion}

We derive communication and function query complexities for specific algorithmic instances in Corollaries \ref{corollary:online:storm+ed}-\ref{corollary:offline:L2S+atc-gt:rd}; see Appendix \ref{appendix:corollary:case1} and \ref{appendix:corollary:case2} for details and  Table \ref{tab:zo-minimax-unified-comparison} for a summary. From these results, we obtain the following insight.

\textbf{Improvements over the existing centralized works.}
As can be observed from Table \ref{tab:zo-minimax-unified-comparison},
our results improve upon non–variance-reduced ZO methods  \cite{wang2023zeroth, an2024robust} by an order of $\mathcal{O}(\varepsilon^{-1})$. The coordinate-wise ZO-STORM-ED/GT matches the centralized complexity $\mathcal{O}(d\kappa^3\varepsilon^{-3})$ of ZO-VRGDA \cite{xu2020gradient} while achieving linear speedup in the number of agents $K$;  see Corollaries  \ref{corollary:online:storm+ed} and \ref{corollary:online:storm+gt}. Moreover, the hybrid ZO variants further reduce this complexity by a factor of $\mathcal{O}(d^{1/4})$; see Corollaries \ref{corollary:online:storm+ed:rd} and \ref{corollary:online:atc+gt:rd}. Compared with Acc-ZOMDA \cite{huang2022accelerated}, our hybrid variant eliminates the logarithmic factor and improves the condition-number dependence by $\mathcal{O}(\kappa^{1.5})$, owing to an improved two-time-scale condition on the learning rates.

\textbf{No ZO method simultaneously achieves the best function query and communication complexities.}
As shown in Table~\ref{tab:zo-minimax-unified-comparison}, ZO-PAGE-ED achieves the best communication complexity of
$\mathcal{O}(\frac{\kappa^2\varepsilon^{-2}}{(1-\lambda)^{1.5}})$, improving upon ZO-STORM-ED and ZO-L2S-ED by a factor of $\tilde{\mathcal{O}}(\varepsilon^{-1})$. 
This improvement stems from the use of a larger mini-batch. However, it also introduces a stronger dependence on the network spectral gap in the dominant term, since the learning rates need to be chosen explicitly in terms of $1-\lambda$. As a consequence, the resulting function query complexity has a worse dependency on the network spectral gap. 
This result reveals an inherent trade-off between communication and function query efficiency, suggesting that different methods may be preferable depending on whether faster communication or lower query complexity is prioritized. 

\textbf{ED is preferred to the GT strategy for ZO methods.} 
From Table \ref{tab:zo-minimax-unified-comparison},
we can see that ED-based ZO algorithms consistently outperform their GT-based counterparts in terms of the network spectral gap $\mathcal{O}(1-\lambda)$ in both stochastic and finite-sum settings. This matches previous observations, e.g. \cite{alghunaim2022unified, cai2025dama1}, and suggests that ED should be preferred over sparse networks, such as ring and line networks.

% \textbf{The trade}
% Referring to Table \ref{tab:zo-minimax-unified-comparison},
% our results indicate that no algorithm, including ZO-STORM-ED/GT, ZO-L2S-ED/GT, or ZO-PAGE-ED/GT, can simultaneously achieve the best communication and function query complexity. This trade-off arises because, although faster algorithms require fewer communication rounds to reach a prescribed accuracy level of solution, they rely on querying more ZO oracles per iteration to do so. As a result, the benefit of reduced communication rounds is offset by the substantially increased number of function queries per communication round, leading to higher ZO oracle complexity.

\textbf{Leveraging the finite-sum structure can lead to improved convergence guarantees.}
Table~\ref{tab:zo-minimax-unified-comparison} presents two types of results: finite-sum guarantees obtained using hyperparameters that exploit the finite-sum structure, and stochastic guarantees without such tuning. The finite-sum results exhibit an explicit dependence on $\mathcal{O}(\sqrt{N})$ together with a weaker dependence on the solution accuracy parameter $\varepsilon$; see, e.g., Corollaries~\ref{corollary:offline:page+ed}--\ref{corollary:offline:page+atc-gt} in Appendix~\ref{appendix:corollary:case1}. These results suggest that improved guarantees can be achieved compared to the stochastic results when the sample size lies in a suitable regime, e.g.,
$N \le \mathcal{O}(\varepsilon^{-2})$.

\textbf{Different ZO estimators provide trade-offs with respect to communication and oracle complexity.}
Table~\ref{tab:zo-minimax-unified-comparison} shows that ZO-CW methods consistently outperform their hybrid ZO counterparts in terms of communication complexity when the problem dimension $d$ is large. This is expected, as coordinate-wise estimators generally provide more accurate gradient approximations than randomized smoothing estimators, at the cost of increased function queries per communication round.
On the other hand, in the stochastic setting, hybrid ZO methods improve upon ZO-CW methods by a factor of $\mathcal{O}(d^{\frac{1}{4}})$ in function query complexity. 
In the finite-sum setting, the advantage of hybrid ZO methods becomes more apparent when $N \ge \tilde{\mathcal{O}}(\varepsilon^{-2})$, showing an $\mathcal{O}(\sqrt{d})$ improvement.
This improvement arises because the finite-sum error term $\Pi_1$ can even be controlled without requiring a full-batch choice of $B$, which relaxes the choice of $B$ in terms of sample size $N$ and eventually leads to improved complexity bounds on the dimension parameter; see Corollaries
\ref{corollary:offline:page+ed:rd}---\ref{corollary:offline:L2S+atc-gt:rd}.

These theoretical insights further highlight the value of the general \textbf{ZOMA} framework. Since \textbf{ZOMA} encompasses multiple algorithmic variants and ZO estimators, it enables unified convergence results that can be specialized to different problem structures, estimator choices, and resource constraints, thereby allowing systematic comparisons among algorithmic variants. These comparisons reveal trade-offs that are difficult to identify from any single method in isolation and show that the choice of algorithm and estimator should depend on the available computational and communication budgets.
\begin{figure*}[t]
\centering
\begin{minipage}[t]{0.42\textwidth}
    \centering
\includegraphics[width=\textwidth]{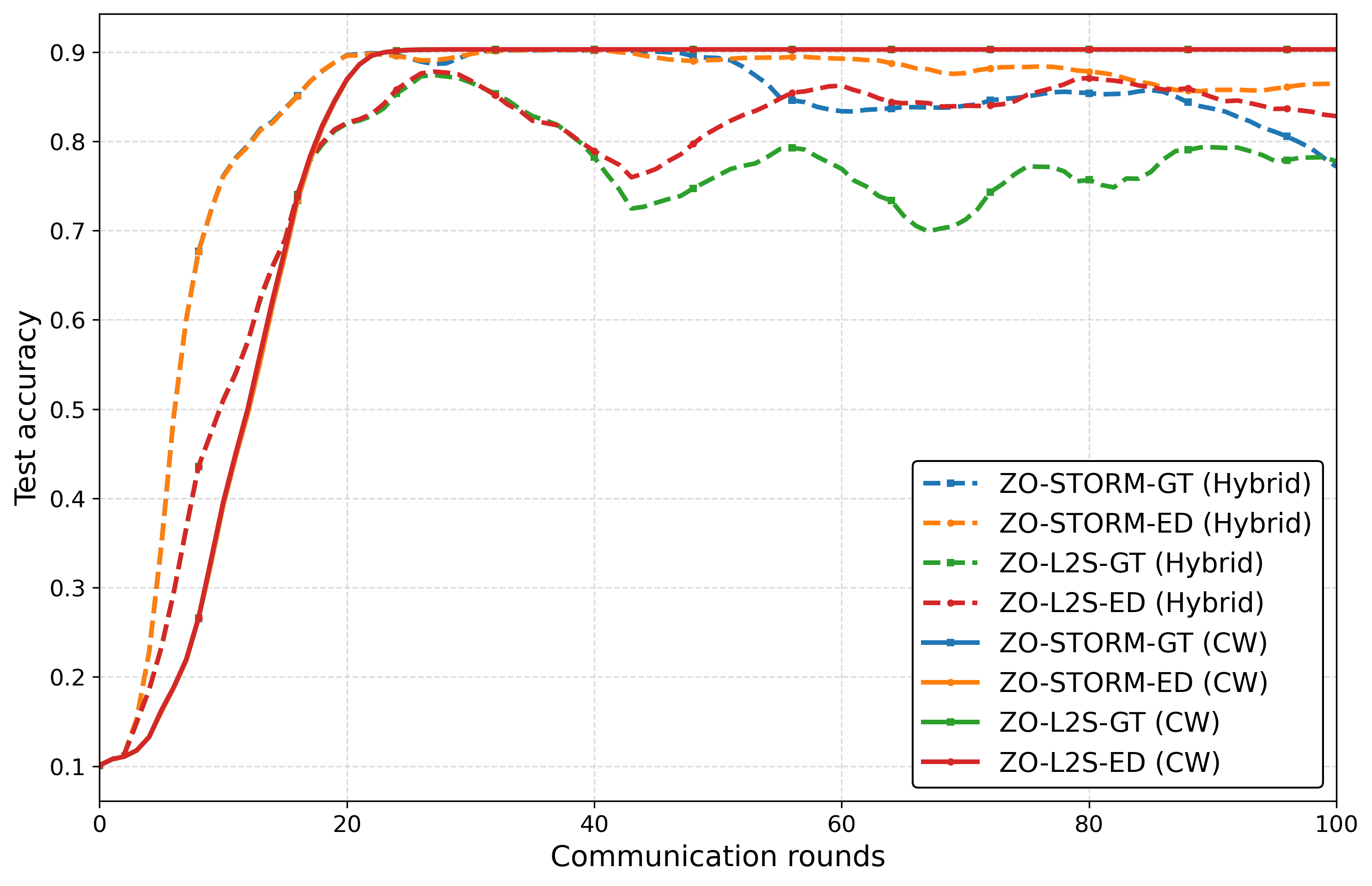}
\end{minipage}
\hfill
\begin{minipage}[t]{0.42\textwidth}
    \centering
    \includegraphics[width=\textwidth]{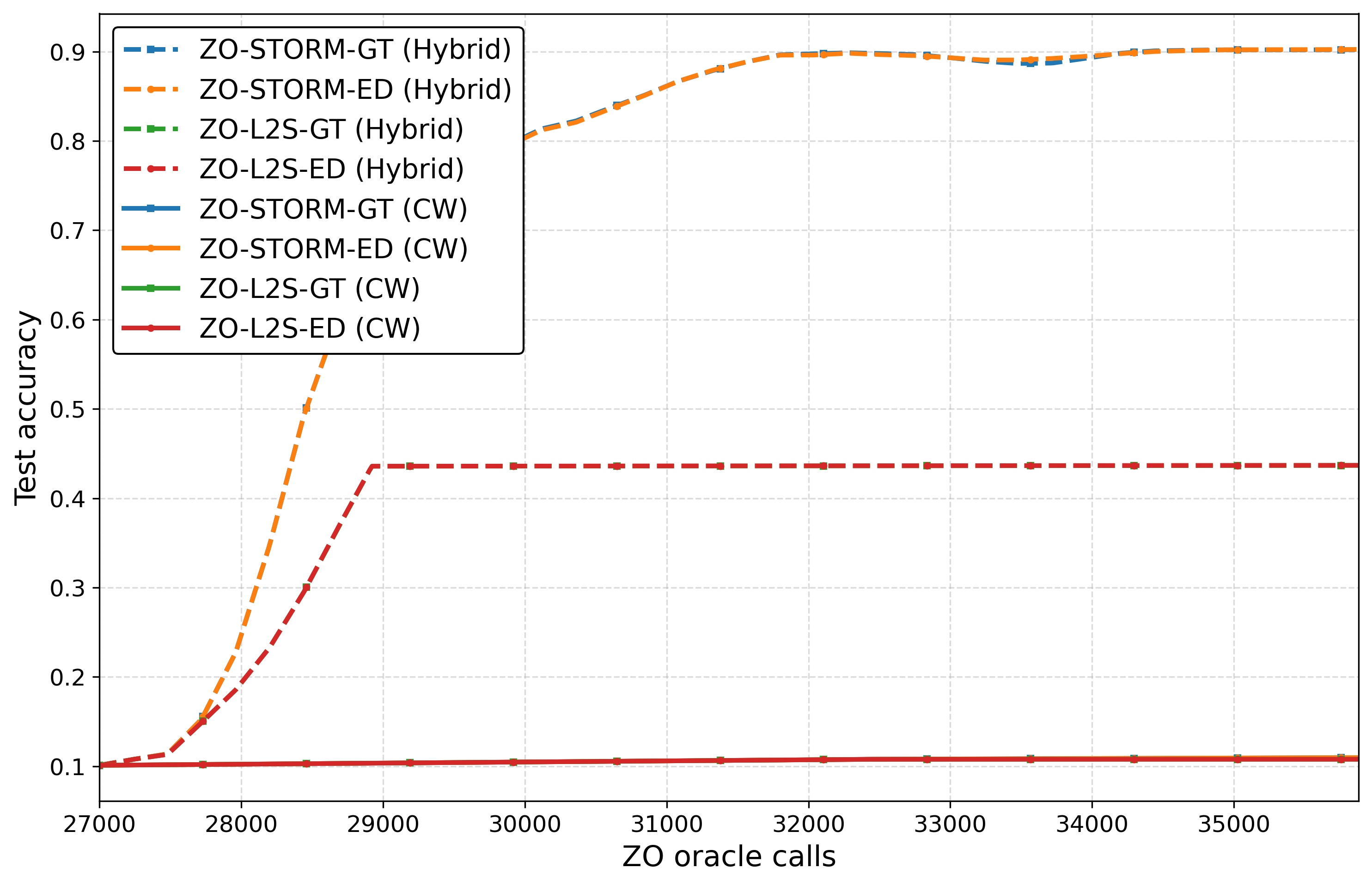}
\end{minipage}
\caption{Simulation result of \textbf{ZOMA} in the nonsmooth case ({\em Setting 1}) with the dataset \texttt{ijcnn1}.
The figures show test accuracy versus communication rounds and ZO oracle calls. In the right figure, the oracle-call axis is truncated once one method reaches its best observed performance.}
\label{fig:compare_nonsmooth}
\end{figure*}

\begin{figure*}[t]
\centering
\begin{minipage}[t]{0.42\textwidth}
    \centering
    \includegraphics[width=\textwidth]{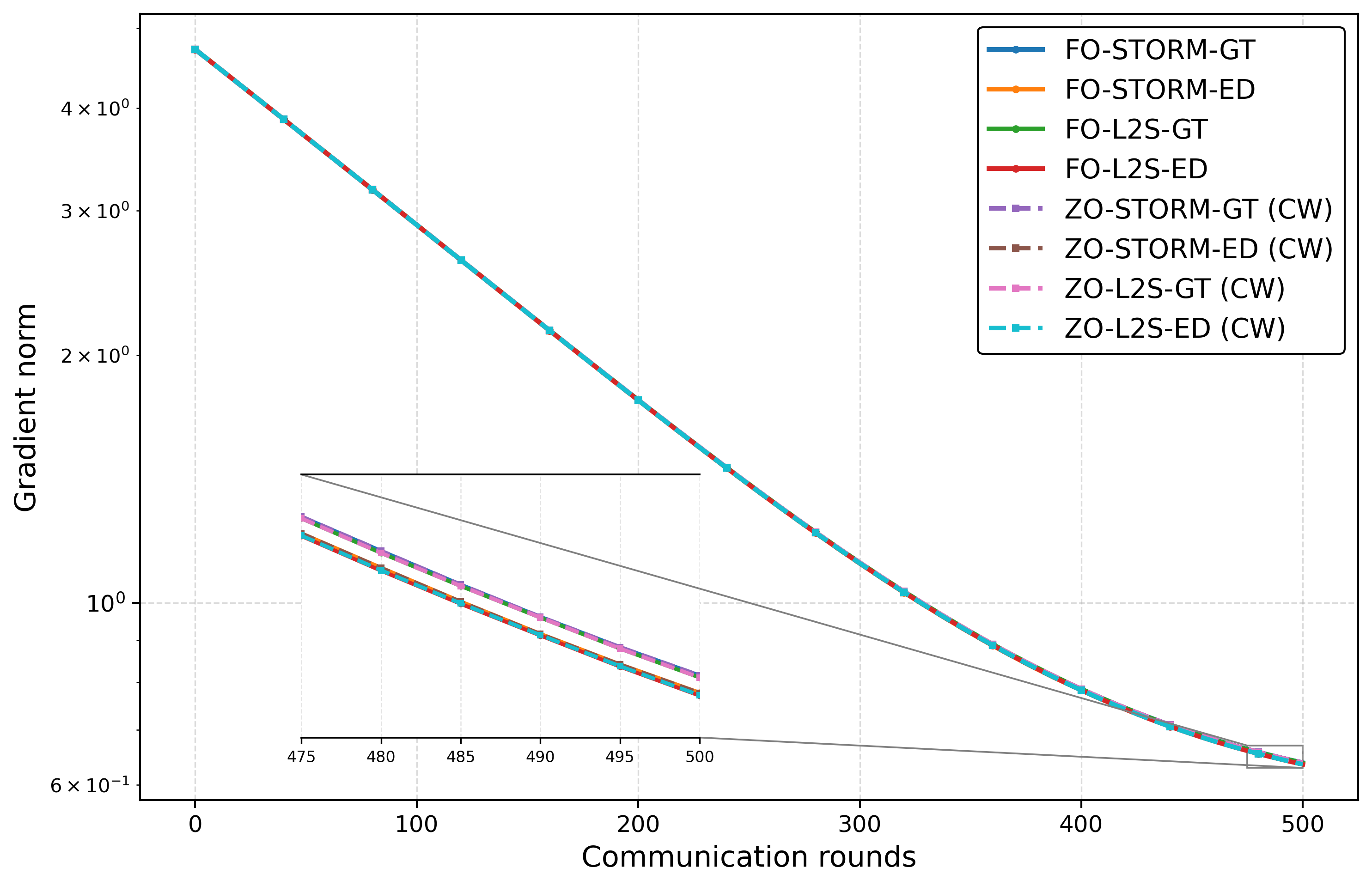}
\end{minipage}
\hfill
\begin{minipage}[t]{0.42\textwidth}
    \centering
    \includegraphics[width=\textwidth]{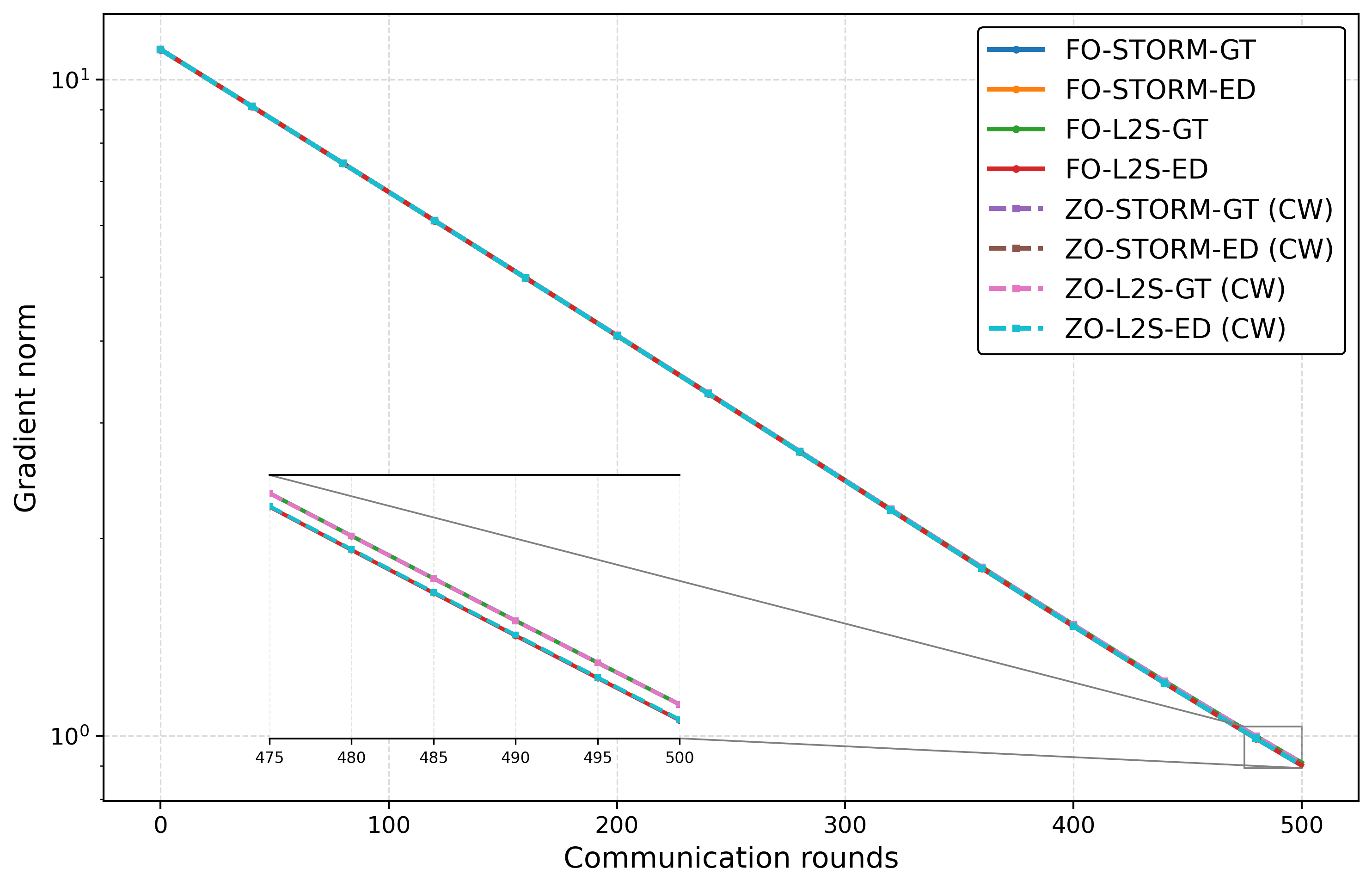}
\end{minipage}
\caption{ Comparison between coordinate-wise (CW) variants of \textbf{ZOMA} and its FO counterparts in the smooth case ({\em Setting 2}). We plot the gradient norm at the network average on two datasets: \texttt{ijcnn1} (left) and \texttt{a9a} (right).}
\label{fig:compare_fo}
\end{figure*}

\begin{figure*}[t]
\centering
\begin{minipage}[t]{0.42\textwidth}
    \centering
    \includegraphics[width=\textwidth]{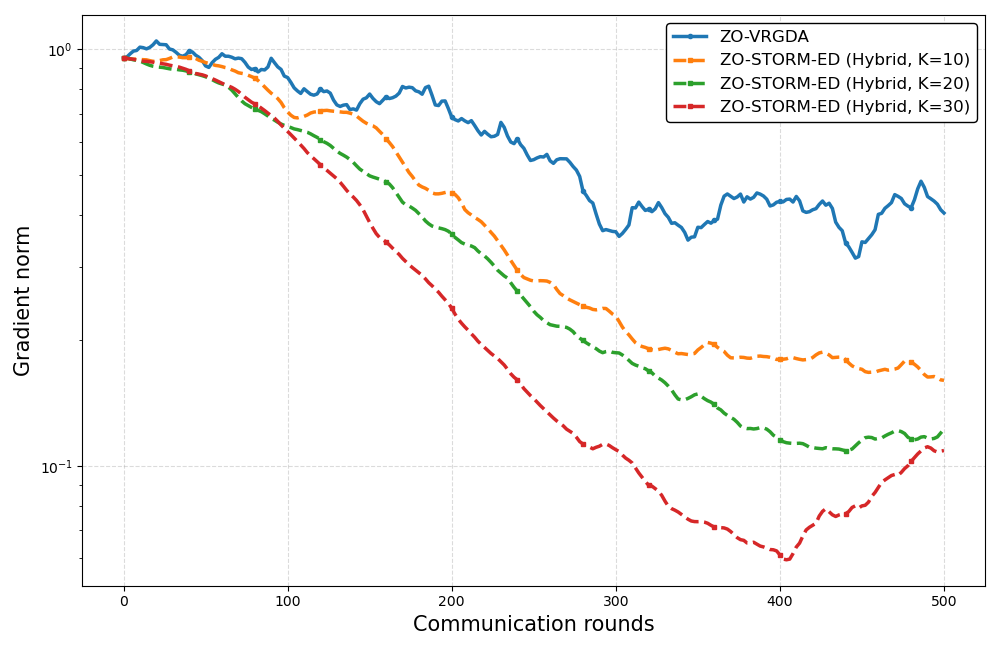}
\end{minipage}
\hfill
\begin{minipage}[t]{0.42\textwidth}
    \centering
    \includegraphics[width=\textwidth]{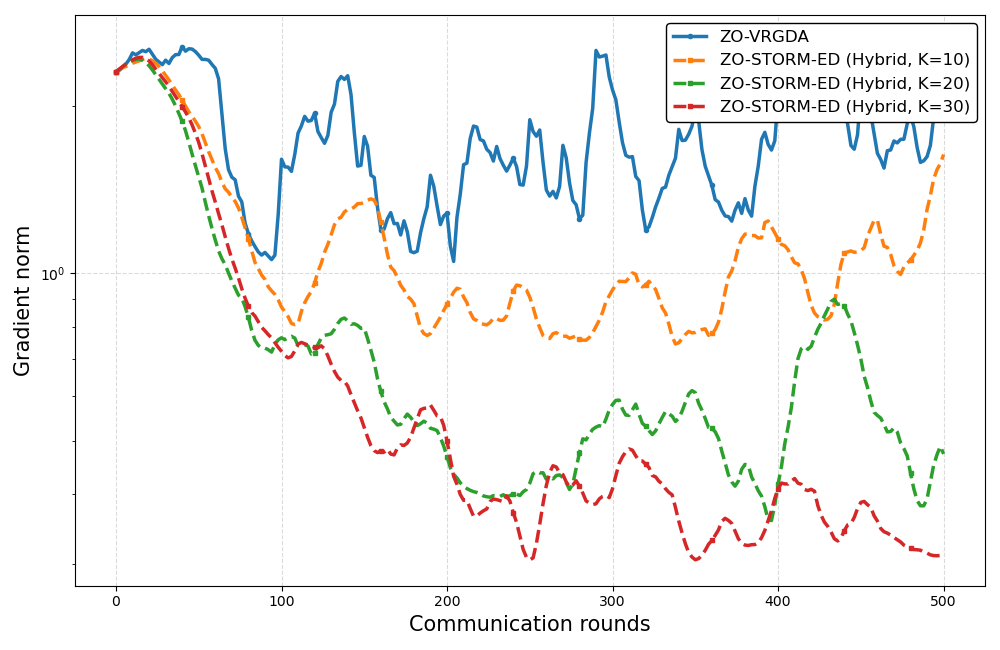}
\end{minipage}
\caption{ Comparison between hybrid variants of \textbf{ZOMA} and centralized method ZO-VRGDA \cite{xu2020gradient} in the smooth case ({\em Setting 3}) with the dataset \texttt{ijcnn1} (left) and \texttt{a9a} (right).}
\label{fig:compare_centralizzed}
\end{figure*}
\section{Numerical simulations}
\label{main:simulation}

We now illustrate the performance of \textbf{ZOMA} on a robust logistic regression task and an image classification task with fairness constraints.

\subsection{Robust logistic regression}
Given the local label–feature pairs $(b_{kj}, a_{kj})$ and a perturbation vector $y$ applied to the feature data,
the robust logistic regression problem is formulated as 
\begin{align}
\min_{x\in \mR^d} \max_{y \in \mR^d} \frac{1}{NK}\sum_{k=1}^K \sum_{j=1}^N
Q_k(x, y; b_{kj}, a_{kj})
- \frac{1}{2}\,\|y\|^2, \notag 
\end{align}
where the local loss function $Q_k(x, y; b_{kj}, a_{kj})$ can be either smooth or nonsmooth, and is defined as
\begin{align}
& Q_k(x, y; b_{kj}, a_{kj}) 
\\
&=  \begin{cases}\ln
\Big(
\frac{\bigl(b_{kj} - x^\top (a_{kj} + y)\bigr)^2}{2} + 1\Big)  &(\textbf{Smooth}) 
\\ 
\max\left\{
0,
1-b_{kj}x^\top(a_{kj}+y)
\right\} &(\textbf{Nonsmooth})
\end{cases}\notag,
\end{align}
A similar function was considered in ~\cite{zhang2024jointly}, however, they only consider the smooth version, while we test the performance on potentially non-smooth functions.

$\bullet$ {\em Setting 1}:
We first evaluate \textbf{ZOMA} in nonsmooth settings where FO methods are not applicable, demonstrating the importance of our ZO methods.
In this setting, we aim to validate the trade-off between communication and function query complexity of the proposed methods.

$\bullet$ {\em Setting 2}:
We further compare our \textbf{ZOMA} against their FO counterparts in smooth settings where gradient information is available, to examine whether ZO methods can match the performance of FO methods. 

$\bullet$ {\em Setting 3}:
Our theory implies a linear speedup benefit of \textbf{ZOMA}  compared with the centralized algorithm; we therefore compare it with the state-of-the-art centralized method ZO-VRGDA \cite{xu2020gradient} to examine whether this speedup is achievable. 

For the first two settings, we generate a lazy ring gossiping matrix $W \in \mR^{K\times K}$ with $K =30$ agents, where $[W]_{ii} = 0.9, [W]_{ij}=0.05 (i \not= j)$. 
For the last setting,  we use Erd\H{o}s--R\'enyi graphs with $K \in \{10,20,30\}$ under a fixed connectivity satisfying
$\|W - \tfrac{1}{K}\mathds{1}_K \mathds{1}_K^\top\| \approx 0.6$.
We consider \texttt{ijcnn1} and \texttt{a9a} datasets from LIBSVM repository\footnote{\url{ https://www.csie.ntu.edu.tw/~cjlin/libsvmtools/datasets/}}. In the smooth case, we can directly evaluate the gradient norm to validate our theoretical results.  In the nonsmooth case, we evaluate the test accuracy of the trained models, where $20\%$ of the data is reserved for testing.

For all algorithms, we tune learning rates $\eta_x, \eta_y $ over the grid $\{0.01, 0.005, 0.001\}$. For ZO-STORM/L2S-ED/GT, we set hyperparameters as $\beta =0.1, p=0.1, B=100$ when applicable. When employing the ZO-CW estimator, we use a minibatch size of $b= b_0 =1$. For the ZO-RU estimators, we might increase the minibatch size over $\{10,20,50\}$ until convergence is observed.
The smoothing factors are set as $\mu_x=\mu_y=\delta_x=\delta_y = 0.001$. 
We adopt the same hyperparameters for the FO counterparts of \textbf{ZOMA}.
For the centralized algorithm ZO-VRGDA, we set the large-batch refreshing period to $q=5$ and tune the number of inner-loop iterations $m$ over $\{1,5,10\}$.
All local iterates are initialized as $0.1 * \mathds{1}_{d'}$, where $ d' =d_1 \text{ or } d_2$.

From Figure \ref{fig:compare_nonsmooth}, we can see that the CW variants of \textbf{ZOMA} require fewer communication rounds to attain the best observed performance compared with their hybrid counterparts. In contrast, when performance is measured against the number of ZO oracle calls, the hybrid variants are more efficient, as they require fewer ZO oracle calls to reach higher performance. We can also observe that among the hybrid variants, the ED-based methods consistently outperform the GT-based variants. Moreover, we can see that our method maintains strong performance even in the non-smooth setting, further highlighting the importance of our ZO framework. All these empirical observations corroborate our theoretical findings.

From Figure~\ref{fig:compare_fo}, we observe that the CW variants of \textbf{ZOMA} can achieve performance comparable to that of its FO counterparts, demonstrating the effectiveness of the proposed ZO methods even in the setting where FO information is available.
From Figure~\ref{fig:compare_centralizzed}, we observe that the proposed decentralized methods exhibit a clear linear speedup with respect to the number of agents $K$ compared to the centralized algorithm ZO-VRGDA. 
In both figures we can see that the CW variants exhibit more stable convergence than the RU variants due to their more accurate gradient estimators.

\subsection{Fair classifier}

We now consider training a neural network classifier over multiple image classes under fairness constraints, considered in ~\cite{mohri2019agnostic}.
The optimization objective is given as follows
\begin{align}
\min_{x} \max_{y} \;
\frac{1}{KC} \sum_{k=1}^{K} \sum_{c=1}^{C}
y_c \, \mathbb{E}_{\boldsymbol{\xi}_{k,c}} \, Q\!\left(x; \boldsymbol{\xi}_{k,c}\right)
- \frac{\rho'}{2} \|y\|^2,
\end{align} 
where $x$ is the model parameter, $\boldsymbol{\xi}_{k,c}$ represents a sample from class $c$ at agent $k$, and $y$ is a weight vector that adjusts the model's attention across different classes and satisfies $\sum^C_{c=1} y_c =1$. The parameter $\rho'= 10^{-3}$ is used for regularization.
We adopt a neural network architecture similar to that in~\cite{wu2024solving}. We use all classes from the FashionMNIST dataset~\cite {xiao2017fashionmnist}, i.e., $C =10$, and each class of data is randomly sharded into $K=20$ agents under a ring network topology.

We use the proposed hybrid variants of \textbf{ZOMA}, including ZO-STORM-ED/EXTRA/GT, to train the model from scratch, as they are more memory-efficient than the CW variants.
The hyperparameters are tuned as follows: $\eta_x = 0.05, \eta_y = 0.1, \beta = 0.95, \mu_x=\mu_y = 0.001$. We report the averaged test accuracy over communication rounds across $10$ independent trials. The results in Figure~\ref{fig:fair_classifier} demonstrate that \textbf{ZOMA} can successfully train a fair neural network from scratch without relying on standard backpropagation. Moreover, ZO-STORM-ED/EXTRA outperforms ZO-STORM-GT by approximately $4\%$ in test accuracy.

% \begin{align}
% &\min_{x} \max_{y} \frac{1}{K}
% \sum_{k=1}^{K}
% \sum_{c=1}^{C} y_c J_{k, c}(x) - \frac{\rho}{2} \|y\|^2, \notag\\
% &\text{where } J_{k,c}(x) = \mE_{\bxi_{k,c} \sim \mathcal{D}_{k,c}}[Q_k(x;\bxi_{k,c})] \notag 
% \\
% &\text{s.t. } y_{i} \ge 0, ~\forall~ c \in \{1,\dots,C\} \text{ and } \sum_{c=1}^C y_{c} = 1,
% \label{simulation:fairclassifier}
% \end{align}
% Here, $\bxi_{k,c}$ is the random sample of category $c$ at agent $k$ and  $Q_k(x;\bxi_{k,c})$
% represents the local loss incurred by category $c$ at the agent $k$. Moreover,  $x$ is the neural network parameter and $\rho$ is the regularization parameter.

% We compare algorithms formed by integrating gradient estimators from the set $\{$"\textbf{GRACE}", "STORM", "Loopless SARAH"$\}$ and decentralized strategies from the set $\{$"ED", "EXTRA", "ATC-GT"$\}$ as well.
% The hyperparameters are tuned via grid search and, unless otherwise specified, are set to the same values across methods, or to smaller values to ensure stability.
% For ED- and ATC-GT-based method, the step sizes are set as $\mu_x = 0.05, \mu_y =0.1$. The $\mu_x$ of the EXTRA-based method is set as $\mu_x = 0.02$ to ensure stability.
% The smoothing factors are set as $\beta_x = \beta_y = 0.95$ for the \textbf{GRACE}- and STORM-based method.
% The warm-up batch size $b_0$ and megabatch size $B$
% is set as a full batch
% while the minibatch size $b$
% is set as $b=50$.
% The Bornoulli parameter $p$
% of \textbf{GRACE}- and Loopless SARAH-based method is set as $p=0.02$.

\begin{figure}
\centering
\includegraphics[width=1.0\linewidth]{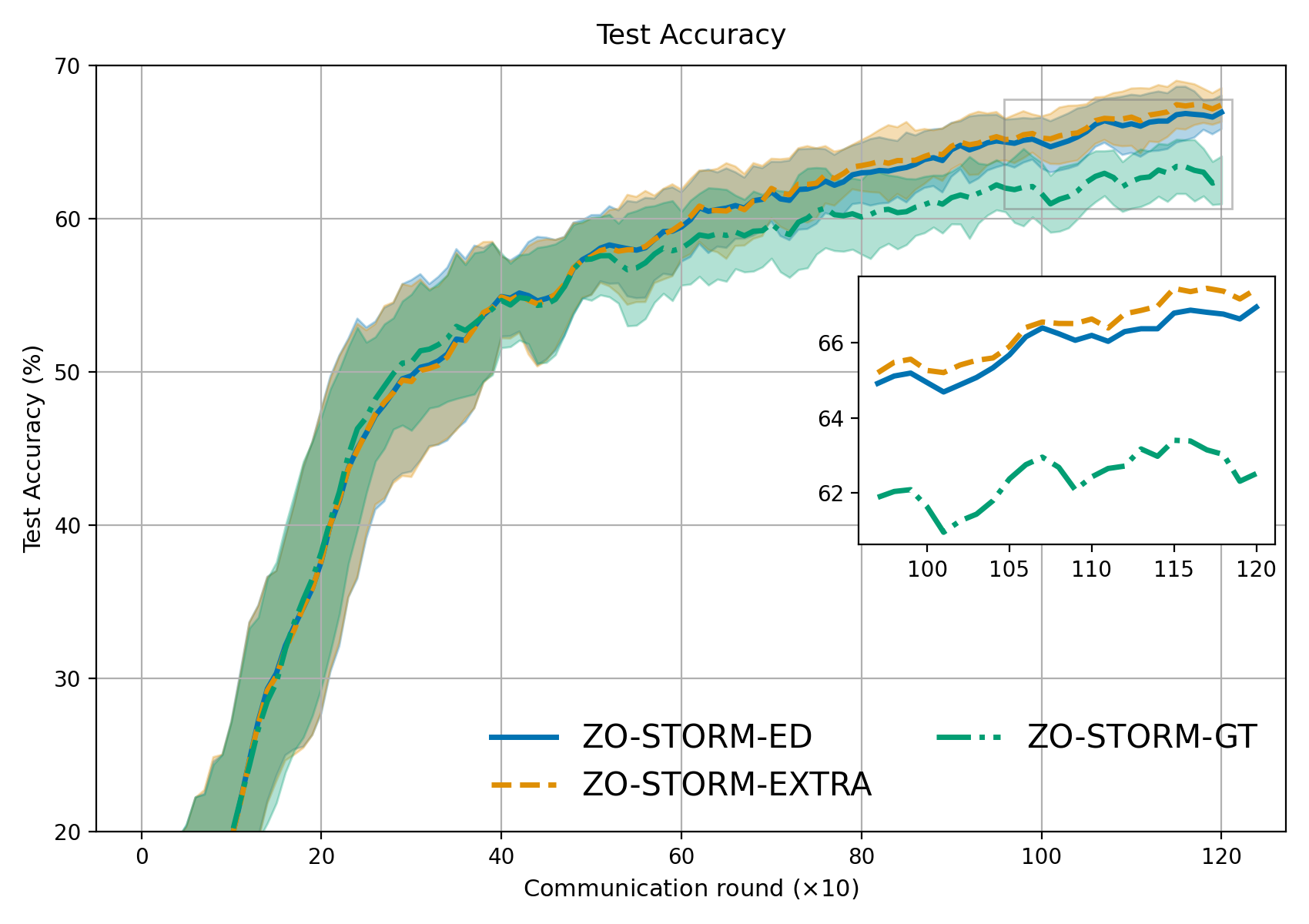}
\caption{
Simulation results on training a fair classifier with the FashionMNIST dataset. We report the
test accuracy of the {\em hybrid} ZO variants of \textbf{ZOMA}, including ZO-STORM-ED/EXTRA/GT. 
}
\label{fig:fair_classifier}
\end{figure}

\section{Conclusion}
In this work, we proposed the first unified ZO framework for decentralized minimax optimization. The proposed framework integrates various decentralized strategies and ZO acceleration techniques, achieving improved guarantees over existing centralized methods. Numerical experiments further validate the theoretical findings and demonstrate the strong empirical performance of our methods. An important future direction is to extend the existing algorithms and analysis to accommodate nondifferentiable and constraint minimax settings.

\bibliographystyle{IEEEtran} 
\bibliography{refs}  

\onecolumn
\appendices

\section{Notation}
\label{sec:notations}
To facilitate analysis, we first introduce the concatenated variables, i.e.,
\begin{align}
z \triangleq  [x; y] , \quad  \bz \triangleq  [\bx;\by] \in \mR^{d_1+d_2}. 
\end{align}
The averaged variable and gradient estimator are defined as
\begin{subequations}
\begin{align}
    \bz_{c,i}  \triangleq[\bx_{c,i}; \by_{c,i}] &\triangleq  \frac{1}{K}
    \sum_{k=1}^{K} [\bx_{k,i};\by_{k,i}] \in \mathbb{R}^{d_1+d_2},   \\ 
   [\bg^x_{c,i};\bg^y_{c,i}] &\triangleq \frac{1}{K}
   \sum_{k=1}^K 
   [\bg^x_{k,i};\bg^y_{k,i}] \in \mathbb{R}^{d_1+d_2}. 
\end{align}
\end{subequations}
The network-extended variables are denoted by
\begin{subequations}
    \begin{align}
&\mZ_i \triangleq [\mX_{i}; \mY_{i}] \triangleq [ {\rm col}\{\bx_{k,i}\}_{k=1}^{K} ;  {\rm col}\{\by_{k,i}\}_{k=1}^{K}] \in \mathbb{R}^{K(d_1+d_2)},  \\
&\mZ_{c,i}\triangleq [\mX_{c,i};\mY_{c,i}] \triangleq  [\mathds{1}_K \otimes  \bx_{c,i}  ; \mathds{1}_K \otimes \by_{c,i}]  \in \mathbb{R}^{K(d_1+d_2)},   \\
&\mG_{z,i} \triangleq[\mG_{x,i}; \mG_{y,i}] \triangleq [
{\rm col}\{\bg^x_{k,i}\}_{k=1}^{K}; {\rm col}\{\bg^y_{k,i}\}_{k=1}^{K}] \in \mathbb{R}^{K(d_1+d_2)}, 
  \\
& [\mD_{x,i}; \mD_{y,i}] \triangleq
[{\rm col}\{\bd^x_{k,i} 
\}_{k=1}^{K} ; {\rm col}\{\bd^y_{k,i} 
\}_{k=1}^{K} ]\in \mathbb{R}^{K(d_1+d_2)}.  
\end{align}
\end{subequations}
We define the risk-value-based ZO-CW gradients $\forall k \in [K]$ as 
\begin{subequations}
    \begin{align}
\hnabla_x J_k(z) &\triangleq \sum_{j=1}^{d_1} \frac{J_k(x+\delta_xe_j,y) - J_k(x-\delta_xe_j,y)}{2\delta_x}e_j \in \mR^{d_1}, \\
\hnabla_y J_k(z) &\triangleq \sum_{j=1}^{d_2} \frac{J_k(x, y+\delta_ye_j) - J_k(x,y-\delta_ye_j)}{2\delta_y}e_j \in \mR^{d_2}. 
\end{align}
\end{subequations}We define {\em smoothed} risk functions  $\forall k \in [K]$ as 
\begin{subequations}
  \begin{align}
J^{\mu_x}_{k}(z) &\triangleq \mE_{\bu \sim {\rm Unif}(\mathbb{B}^{d_1})} J_k(x+\mu_x \bu, y)\in \mR, \\
J^{\mu_y}_{k}(z) &\triangleq \mE_{\bu \sim {\rm Unif}(\mathbb{B}^{d_2})} J_k(x, y+\mu_y \bu) \in \mR. 
\end{align}  
\end{subequations}
When using the hybrid ZO estimator,  we denote the local gradient error 
by
\begin{align}
[\bs^x_{k,i}; \bs^y_{k,i}] &\triangleq [\bg^x_{k,i} - \nabla_x J^{\mu_x}_{k}(\bz_{k,i}) ; \bg^y_{k,i} - \nabla_y J^{\mu_y}_{k}(\bz_{k,i}) ] \in \mR^{d_1+d_2}. 
\end{align}
When using the pure ZO-CW estimator, the local gradient error is instead defined as
\begin{align}
[\bs^x_{k,i}; \bs^y_{k,i}] &\triangleq [\bg^x_{k,i} -\hnabla_x J_k(\bz_{k,i}) ; \bg^y_{k,i} - \hnabla_y J_k(\bz_{k,i}) ] \in \mR^{d_1+d_2}. 
\end{align}
Accordingly, the averaged and networked gradient error can be defined as 
\begin{align}
\mS_{z,i}&\triangleq [\mS_{x,i};\mS_{y,i}] \triangleq  \Big[{\rm col}\{\bs^x_{k,i}\}^{K}_{k=1}; {\rm col}\{\bs^y_{k,i}\}^{K}_{k=1}\Big] \in \mathbb{R}^{K(d_1+d_2)}, 
\\ 
\bs^z_{c,i} &\triangleq  [\bs^x_{c,i};\bs^y_{c,i}] \triangleq  \frac{1}{K}
 \sum_{k=1}^{K} [\bs^x_{k,i}; \bs^y_{k,i}] \in \mR^{d_1+d_2}. 
\end{align}
At communication round $i$, let $\bz = \bz_{k,i}$ or $\bz_{k,i-1}$,
the local $b$-batch ZO gradient estimator is defined as 
\begin{subequations}
    \begin{align}
[\bq^x_{i}(\bz;b);\bq^y_{i}(\bz;b)] &\triangleq  \frac{1}{b}\sum_{j=1}^{b}\Big[\bq^x_k(\bz;\bxi^j_{k,i},\bu^j_{k,i});\bq^y_{k}(\bz;\bxi^j_{k,i},\bu^j_{k,i})\Big]\in \mR^{d_1+d_2},\notag
\\
[\bq^x_{i,0}(\bz;b);\bq^y_{i,0}(\bz;b)] &\triangleq \frac{1}{b}\sum_{j=1}^b\Big[\bq^x_k(\bz;\bxi^j_{k,i},0); \bq^y_{k}(\bz;\bxi^j_{k,i},0)\Big]\in \mR^{d_1+d_2},\notag
\end{align}
\end{subequations}
where the local index \(k\) of the \(b\)-batch ZO gradient estimator is omitted for compactness, since it is clear once we insert the local iterate.

We further define the true gradients as 
\begin{subequations}
\begin{align}
\nabla^z_{k,i} &\triangleq [\nabla^x_{k,i}; \nabla^y_{k,i}] \triangleq [\nabla_x J_k(\bz_{k,i}); \nabla_y J_k(\bz_{k,i})] \in \mR^{d_1+d_2}, \\
\nabla^z_{c,i} &\triangleq [\nabla^x_{c,i}; \nabla^y_{c,i}] \triangleq [\nabla_x J(\bz_{c,i}); \nabla_y J(\bz_{c,i})] \in \mR^{d_1+d_2}. 
\end{align}
\end{subequations}

The function gap $P(x) - J(x,y)$ 
at  $(\bx_{c,i}, \by_{c,i})$
is defined as 
\begin{align}
\Delta^y_{c,i} \triangleq P(\bx_{c,i}) - J(\bx_{c,i}, \by_{c,i}) \in \mR. 
\end{align}

\section{Basic Lemmas}
\label{appendix:sec:basiclemma}
\begin{Lemma}[\textbf{ZO-RU}]
\label{appendix:lemma:randomsmoothing}
Let Assumptions \ref{main:assumption:Lipschitz} and \ref{main:asssumption:RD-ZO} hold,
for the ZO-RU estimator, we have 
$\forall w \in \{x, y\}$ that
\begin{align}
&(1) \quad J^{\mu_x}_{k}(\cdot, y), J^{\mu_y}_{k}(x, \cdot) \text{ are } L^\prime_f (\le L_f)\text{--smooth},\notag\\
&(2) \quad \|\nabla_x J^{\mu_x}_{k}(z) - \nabla_x J_k(z)\| \le \frac{\mu_x L_f d_1}{2}, \quad \|\nabla_y J^{\mu_y}_{k}(z) - \nabla_y J_k(z)\| \le \frac{\mu_y L_f d_2}{2},
\notag
\\
&(3) \quad
\mE_{\bxi, \bu}
[\bq^w_{k}(z;\bxi, \bu)]  = 
\nabla_w J^{\mu_w}_{k}(z),\notag 
\\
&(4) \quad
\mE\Big\|\frac{1}{b}\sum_{j=1}^b
\bq^w_{k}(z;\bxi^j, \bu^j) - \nabla_w J^{\mu_w}_{k}(z) \Big\|^2  \le \frac{\sigma^2_1}{b}\mE\|\nabla_w J(z)\|^2 + \frac{\sigma^2_0}{b}. \notag
\end{align}
\begin{proof}
By fixing either $x$ or $y$, results (1), (2), and (3)
follow directly from  \cite[Lemma 4.1]{gao2018information}.
For (4), we focus on the $x$-variable and obtain
\begin{align}
&\mE\Big\|\frac{1}{b}\sum_{j=1}^b
\bq^x_{k}(z;\bxi^j, \bu^j) - \nabla_x J^{\mu_x}_{k}(z) \Big\|^2 \overset{(a)}{\le} \frac{1}{b^2}
\sum_{j=1}^{b}\mE\Big\|
\bq^x_{k}(z;\bxi^j, \bu^j) - \nabla_x J^{\mu_x}_{k}(z) \Big\|^2\overset{(b)}{\le} \frac{\sigma^2_1}{b}\mE\|\nabla_x J(z)\|^2 + \frac{\sigma^2_0}{b}, \notag
\end{align}
where $(a)$ follows from (3)
and the fact that $\{\bxi^j\}$ and $\{\bu^j\}$
are i.i.d. sequences, while 
$(b)$ follows from Assumption \ref{main:asssumption:RD-ZO}.
\end{proof}
\end{Lemma}
\begin{Lemma}[\textbf{ZO-CW} \cite{ji2019improved}]
\label{appendix:lemma:coordinate_zero_first}

Let Assumption \ref{main:assumption:Lipschitz} hold, 
for any smoothing parameter $\delta_x, \delta_y > 0$,  it holds that
\begin{align}
\|\hnabla_{x} J_k(z) -\nabla_xJ_k(z)\|^2 &\le d_1 L^2_f\delta_x^2, &\quad \|\hnabla_{y} J_k(z) -\nabla_y J_k(z)\|^2 &\le d_2 L^2_f\delta_y^2,
\\
\mE\|\bq^x_{k}(z;\bxi,0) -  \nabla_x Q_k(z;\bxi)\|^2
&\le d_1 L^2_f \delta^2_x, &\quad \mE\|\bq^y_{k}(z;\bxi,0) -  \nabla_y Q_k(z;\bxi)\|^2
&\le d_2 L^2_f \delta^2_y. 
\end{align}
\begin{proof}
The above results can be extended from those in \cite{ji2019improved}. For illustration, we establish the first inequality only.
Using the mean value theorem for  $J_k(\cdot, y)$,
we have 
\begin{align}
&J_k(x +\delta_x e_j ,y) = J_k(x -\delta_x e_j ,y) 
+ (2\delta_xe_j)^\top\nabla_x J_k(x -\delta_x e_j + 2t_j\delta_xe_j,y), \quad \text{for } t_j\in (0,1) . 
\label{appendix:lemma:coordinate_gradient_deviation}
\end{align}
It follows that 
\begin{align}
&\|\hnabla_{x} J_k(z) -\nabla_x J_k(z)\|^2
\notag \\
&= \Big\|\sum_{j=1}^{d_1} \frac{J_k(x+\delta_x e_j, y) -J_k(x-\delta_xe_j, y)}{2\delta_x}e_j -\nabla_x J_k(z)\Big\|^2
\notag \\
&\overset{(a)}{=}
\Big\|
\sum_{j=1}^{d_1}e_je^\top_j \Big(\nabla_x J_k(x + (2t_j-1)\delta_x e_j,y) -\nabla_x J_k(z)\Big)\Big\|^2 \notag \\
&\overset{(b)}{=} \sum^{d_1}_{j=1}
\Big\|e_je^\top_j \Big(\nabla_x J_k(x + (2t_j-1)\delta_x e_j,y) -\nabla_x J_k(z)\Big)\Big\|^2 \notag 
\\
&\overset{(c)}{\le}
\sum^{d_1}_{j=1}
\Big\| \nabla_x J_k(x + (2t_j-1)\delta_x e_j,y) -\nabla_x J_k(z)\Big\|^2 \notag 
\\
&\overset{(d)}{\le}
\sum_{j=1}^{d_1}L^2_f(2t_j-1)^2\|\delta_x e_j\|^2 \le d_1 L^2_f\delta^2_x,  \quad t_j\in(0,1) ,
\end{align}
where $(a)$ follows from \eqref{appendix:lemma:coordinate_gradient_deviation} and the fact that $\sum_{j=1}^{d_1} e_je^\top_j = \mathrm{I}_{d_1}$, $(b)$ follows from the fact that $e_je^\top_j v$
produces a vector with its $j$-th entry being $v[j]$
and others being zero, $(c)$ follows
from the sub-multiplicative property of the matrix norm and $\|e_je^\top_j\|  =1$, and $(d)$ follows from Assumption \ref{main:assumption:Lipschitz} and the relation $0<2t_j-1 < 1$.
\end{proof}
\end{Lemma}
\begin{Lemma}[\textbf{ZO-CW}]
\label{appendix:lemma:coordinate_large_bound}
Let Assumptions \ref{main:assumption:Lipschitz} and  \ref{main:assumption:boundedvariance}
hold,  we have
$\forall w \in \{x, y\}$ and $\forall k \in [K]$ that
\begin{align}
&\mE\Big\|
\frac{1}{B}
\sum_{j=1}^{B}
\bq^w_{k}(z;\bxi^{j},0)
- 
\hnabla_w J_k(z)
\Big\|^2 \le 
\frac{C_0}{B}
\mathbb{I}(B<N),  \notag \\
&\mE\Big\|{\rm col}
\Big\{
\frac{1}{B}
\sum_{j=1}^{B}
\bq^w_{k}(z;\bxi^{j},0)
\Big\}_{k=1}^{K}
- {\rm col}\Big\{
\hnabla_w J_k(z)
\Big\}_{k=1}^{K}
\Big\|^2 \le 
\frac{KC_0}{B}\mathbb{I}(B<N). \notag 
\end{align}
where $\mathbbm{I}(\cdot)$
is an indicator function
and $
C_0 \triangleq \max\Big\{3(2d_1L^2_f\delta^2_x +\sigma^2), 3(2d_2L^2_f\delta^2_y +\sigma^2) \Big\}$.
\end{Lemma}
\begin{proof}
The proof will focus on the $x$-variable.
In an online {\em stochastic} scenario, it holds that
\begin{align}
&\mE\Big\|\frac{1}{B}\Big(
\sum_{j=1}^{B} \bq^x_{k}(z;\bxi^{j},0) - \hnabla_x J_k(z)\Big)\Big\|^2 \notag \\
&\overset{(a)}{=}
\frac{1}{B^2}
\sum_{j=1}^{B}
\mE\| \bq^x_{k}(z;\bxi^{j},0) - \hnabla_x J_k(z)\|^2 \notag \\
&\le 
\frac{1}{B^2}
\sum_{j=1}^{B}
\Big(3\mE\|
\bq^x_{k}(z;\bxi^j,0)
-\nabla_x Q_k(z;\bxi^j)
\|^2+3\mE\|\nabla_x
 Q_k(z;\bxi^j)-
\nabla_x J_k(z)\|^2 +3\mE\|\hnabla_x J_k(z) -
\nabla_x J_k(z)\|^2\Big) \notag \\
&\overset{(b)}{\le} 
\frac{3}{B}(2d_1L^2_f\delta_x^2 +\sigma^2), \notag 
\end{align}
where $(a)$
and $(b)$ follows from Lemma \ref{appendix:lemma:coordinate_zero_first} and Assumption \ref{main:assumption:boundedvariance}.
In {\em finite-sum} scenarios,
$\{\bxi^{j}\}_{j=1}^{B}$
are uniformly sampled without replacement,
we can borrow the result in \cite[Lemma 4]{ji2019improved}
\begin{align}
&\mE\Big\|\frac{1}{B}
\sum_{j=1}^{B} \bq^x_{i,0}(z;\bxi^{j},0) - \hnabla_x J_k(z)\Big\|^2\le 
\frac{3\mathbb{I}(B<N)}{B}(2d_1L^2_f\delta_x^2 +\sigma^2) .
\label{appendix:proof:ZO-CW:keysptep}
\end{align}
We note that relation \eqref{appendix:proof:ZO-CW:keysptep} provides a unified bound for both stochastic and finite-sum scenarios.
Denoting
$C_0 \triangleq \max\Big\{3(2d_1L^2_f\delta^2_x +\sigma^2), 3(2d_2L^2_f\delta^2_y +\sigma^2) \Big\}$, we can complete the proof.
\end{proof}

\begin{Lemma}[\textbf{Dansky-type lemma} \cite{nouiehed2019solving}]
\label{Danskin}
    Under Assumptions \ref{main:assumption:costfunction}--\ref{main:assumption:boundedvariance}, the value function $P(x) = \max_y J(x,y)$ is $L \triangleq (L_f + \frac{\kappa L_f}{2})$-smooth 
    and satisfies $
           \nabla P(x) = \nabla_x J(x , y^o(x))$, 
     where $\kappa \triangleq \frac{L_f}{\nu}$ is the condition number
     and 
$y^o(x) =\operatorname{argmax}_y \  J(x, y)$. Furthermore, the quadratic growth property holds
\begin{align}
 P(x) - J(x,y) \ge  \frac{\nu}{2}\|y -y^o(x)\|^2.
 \notag
\end{align}
\end{Lemma}

\section{Possible Choices for Combination Matrices}
\label{appendix:subsection:matrixchoice}

\begin{table*}[h]
\centering
\caption{\normalfont Possible choices for matrices discussed in subsection \ref{main:subsection:unified}. Below, $W \in \mR^{K\times K}$ is a symmetric, doubly stochastic matrix.}
\label{tab:matrix_choices}
\renewcommand{\arraystretch}{1.2}
\setlength{\tabcolsep}{3pt}
\begin{tabular}{l|cccccc}
\toprule
\diagbox[width=2.5cm]{\textbf{Strategy}}{\textbf{Choices}}
& $\mA_x$ & $\mA_y$ & $\mB_x$ & $\mB_y$ & $\mC_x$ & $\mC_y$ \\ 
\toprule 
ED     & $W \otimes \mathrm{I}_{d_1}$ & $W \otimes \mathrm{I}_{d_2}$ & $(\mathrm{I}_{Kd_1} - W \otimes \mathrm{I}_{d_1})^{1/2}$ & $(\mathrm{I}_{Kd_2} -W \otimes \mathrm{I}_{d_2})^{1/2}$ & $\mathrm{I}_{Kd_1}$ & $\mathrm{I}_{Kd_2}$ \\ 
EXTRA  & $\mathrm{I}_{Kd_1}$ & $\mathrm{I}_{Kd_2}$ & $(\mathrm{I}_{Kd_1} - W \otimes \mathrm{I}_{d_1})^{1/2}$ & $(\mathrm{I}_{Kd_2} - W \otimes \mathrm{I}_{d_2})^{1/2}$ & $W \otimes \mathrm{I}_{d_1}$ & $W \otimes \mathrm{I}_{d_2}$ \\
ATC-GT & $(W \otimes \mathrm{I}_{d_1})^2$ & $(W \otimes \mathrm{I}_{d_2})^2$ & $\mathrm{I}_{Kd_1} - W \otimes \mathrm{I}_{d_1}$ & $\mathrm{I}_{Kd_2} - W \otimes \mathrm{I}_{d_2}$  & $\mathrm{I}_{Kd_1}$ & $\mathrm{I}_{Kd_2}$  
% \\ 
% semi-ATC-GT & $W \otimes \mathrm{I}_{d_1}$ & $W \otimes \mathrm{I}_{d_2}$ & $\mathrm{I}_{Kd_1} - W \otimes \mathrm{I}_{d_1}$ & $\mathrm{I}_{Kd_2} - W \otimes \mathrm{I}_{d_2}$  & $W \otimes \mathrm{I}_{d_1}$ & $W \otimes \mathrm{I}_{d_2}$  \\ 
% non-ATC-GT & $ \mathrm{I}_{Kd_1}$ & $\mathrm{I}_{Kd_2}$ & $\mathrm{I}_{Kd_1} - W \otimes \mathrm{I}_{d_1}$ & $\mathrm{I}_{Kd_2} - W \otimes \mathrm{I}_{d_2}$  & $(W \otimes \mathrm{I}_{d_1})^2$ & $(W \otimes \mathrm{I}_{d_2})^2$  
\\
\bottomrule
\end{tabular}
\end{table*}

The possible choices for design matrices $\{\mA_x, \mB_x, \mC_x, \mA_y, \mB_y, \mC_y\}$ are shown in Table \ref{tab:matrix_choices}. We further refer to \cite{cai2025dama1} for specializing the unified form \eqref{main:zodama_recursion1}–\eqref{main:zodama_recursion4} after determining the design matrices.

\section{Fundamental Transformation}
\label{appendix:sec:transformation}

To facilitate the analysis, 
we apply a transformation to \eqref{main:zodama_recursion1}–\eqref{main:zodama_recursion4} following \cite{cai2025dama1}.  In particular, these results enable us to transform the unified recursions \eqref{main:zodama_recursion1}–\eqref{main:zodama_recursion4} into equivalent forms via the fundamental factorization of the matrices $\{\mA_x, \mB_x, \mC_x, \mA_y, \mB_y, \mC_y\}$.
The combination matrix $W$ is symmetric, doubly-stochastic, and primitive.
Suppose that the eigenvalues of $W$ are given by $\{1, \lambda_2, \ldots, \lambda_K\}$, ordered in descending order.  The matrix $W$ admits the following eigendecomposition
\begin{align}
W = U\Lambda U^\top \triangleq 
\begin{bmatrix}
\frac{1}{\sqrt{K}}\mathds{1}_{K},  \widehat{U}
\end{bmatrix}
\begin{bmatrix}
1 &0\\
0& \widehat{\Lambda}
\end{bmatrix}
\begin{bmatrix}
    \frac{1}{\sqrt{K}} \mathds{1}^\top_K\\
    \widehat{U}^\top
\end{bmatrix}
\in \mathbb{R}^{K\times K},
\end{align}
where $\widehat{\Lambda} ={\rm diag}\{\lambda_i\}^{K}_{i=2}$ is a diagonal matrix and $\widehat{U}\in \mathbb{R}^{K \times (K-1)}$ satisfies $\widehat{U}^\top\widehat{U} = \mathrm{I}_{K-1}$.
From Assumption \ref{main:assumptions:combinationmatrix}, $\mA_x, \mB_x, \mC_x$,
are polynomial functions of $W \otimes \mathrm{I}_{d_1}$, 
it follows that 
\begin{subequations}
\begin{align}
\mA_x &\triangleq \mU_x \Lambda_{a_x} \mU^\top_x
\triangleq \begin{bmatrix}
\frac{1}{\sqrt{K}}
\mathds{1}_K \otimes \mathrm{I}_{d_1}, \widehat{\mU}_x
\end{bmatrix}
\begin{bmatrix}
    \mathrm{I}_{d_1} &0\\
    0& \widehat{\Lambda}_{a_x}
\end{bmatrix}
\begin{bmatrix}
\frac{1}{\sqrt{K}}
\mathds{1}^\top_K \otimes \mathrm{I}_{d_1} \\
\widehat{\mU}^\top_x 
\end{bmatrix} , \label{proof:Ax_eigen}
\\
\mC_x &\triangleq \mU_x \Lambda_{c_x} \mU^\top_x
\triangleq \begin{bmatrix}
\frac{1}{\sqrt{K}}
\mathds{1}_K \otimes \mathrm{I}_{d_1}, \widehat{\mU}_x
\end{bmatrix}
\begin{bmatrix}
    \mathrm{I}_{d_1} &0\\
    0& \widehat{\Lambda}_{c_x}
\end{bmatrix}
\begin{bmatrix}
\frac{1}{\sqrt{K}}
\mathds{1}^\top_K \otimes \mathrm{I}_{d_1}\\
\widehat{\mU}^\top_x
\end{bmatrix},  \label{proof:Cx_eigen}
\\
\mB^2_x &\triangleq \mU_x \Lambda^2_{b_x}\mU^\top_x
\triangleq 
\begin{bmatrix}
\frac{1}{\sqrt{K}}
\mathds{1}_K \otimes \mathrm{I}_{d_1}, \widehat{\mU}_x
\end{bmatrix}
\begin{bmatrix}
    0 &0\\
    0& \widehat{\Lambda}^2_{b_x}
\end{bmatrix}
\begin{bmatrix}
\frac{1}{\sqrt{K}}
\mathds{1}^\top_K \otimes \mathrm{I}_{d_1}\\
\widehat{\mU}^\top_x
\end{bmatrix},  \label{proof:Bx_eigen}
\end{align}
\end{subequations}
where $\widehat{\mU}_x \triangleq \widehat{U}\otimes \mathrm{I}_{d_1}$
and
$\widehat{\Lambda}_{a_x}, \widehat{\Lambda}_{c_x}, \widehat{\Lambda}^2_{b_x}$
denote eigenvalue matrices of $\mA_x, \mC_x, \mB^2_x$, respectively, each defined as follows:
\begin{align}
\widehat{\Lambda}_{a_x} &= {\rm diag}\{ \lambda_{a_x, i}\}_{i=2}^{K} \otimes \mathrm{I}_{d_1}, \ \widehat{\Lambda}_{c_x} = {\rm diag}\{ \lambda_{c_x, i}\}_{i=2}^{K} \otimes \mathrm{I}_{d_1}, 
\widehat{\Lambda}^2_{b_x} = {\rm diag}\{ \lambda^2_{b_x, i}\}_{i=2}^{K} \otimes \mathrm{I}_{d_1}.
\end{align}
We assume a similar eigendecomposition form for $\mA_y, \mC_y, \mB^2_y$, omitted here for brevity.
Introducing the auxiliary variables
\begin{subequations}
\begin{align}
    \mZ_{x,i} &\triangleq \eta_x \mA_x\mG_{x,i} + \mB_x\mD_{x,i}-\mB^2_x\mX_{i} , \label{main:ZxExpression}\\
    \mZ_{y,i} &\triangleq  -\eta_y \mA_y \mG_{y,i} + \mB_y \mD_{y,i}-\mB^2_{y} \mY_{i}.
    \label{main:ZyExpression}
\end{align}
\end{subequations}
 Following the work  \cite{cai2025dama1}, we can derive the following lemma.
\begin{Lemma}[\textbf{Transformed error recursion} \cite{cai2025dama1}]
\label{appendix:lemma:transformed_recursion}
Let Assumption \ref{main:assumptions:combinationmatrix} hold,
recursions \eqref{main:zodama_recursion1}---\eqref{main:zodama_recursion4} can be equivalently transformed into
\begin{subequations}
\begin{align}
\bx_{c,i+1} &= \bx_{c,i}-\frac{\eta_x}{K} \sum_{k=1}^{K} \bg^x_{k,i}\label{main:eq:x_finalrecursion},\\
\by_{c,i+1} &= \by_{c,i}  + \frac{\eta_y}{K}\sum_{k=1}^{K}
\bg^y_{k,i}\label{main:eq:y_finalrecursion} ,\\
\mhE_{x,i+1} &= \mT_x \mhE_{x,i} 
-\frac{\eta_x}{\tau_x}\widehat{\mQ}^{-1}_x 
\begin{bmatrix}
0\\
\widehat{\Lambda}^{-1}_{b_x} \widehat{\Lambda}_{a_x} \widehat{\mU}^\top_x(\mG_{x,i} -\mG_{x,i+1})
\end{bmatrix} \label{main:eq:ex_finalrecursion} ,\\
\mhE_{y,i+1} &= \mT_y \mhE_{y,i}  + \frac{\eta_y}{\tau_y}  \widehat{\mQ}^{-1}_y \begin{bmatrix}
0\\ 
\widehat{\Lambda}^{-1}_{b_y}\widehat{\Lambda}_{a_y} \widehat{\mU}^\top_y(\mG_{y,i} - \mG_{y,i+1})
\end{bmatrix}\label{main:eq:ey_finalrecursion},
\end{align}
\end{subequations}
where  $\tau_x$ and $\tau_y$ are arbitrary positive constants, and
$\mhE_{x,i}, \mhE_{y,i}$ are 
coupled error terms
defined as 
\begin{align}
\mhE_{x,i} &\triangleq
   \frac{1}{\tau_x} \widehat{\mQ}^{-1}_x
\begin{bmatrix}
 \widehat{\mU}^\top_x \mX_i \\
 \widehat{\Lambda}_{b_x}^{-1} \widehat{\mU}^\top_x \mZ_{x,i}
\end{bmatrix} \in  \mathbb{R}^{2(K-1)d_1},  \\
\mhE_{y,i}  &\triangleq
   \frac{1}{\tau_y} \widehat{\mQ}^{-1}_y
\begin{bmatrix}
 \widehat{\mU}^\top_y \mY_i \\
 \widehat{\Lambda}_{b_y}^{-1} \widehat{\mU}^\top_y \mZ_{y,i}
\end{bmatrix}  \in  \mathbb{R}^{2(K-1)d_2},
\label{main:tranform:def_couplederror}
\end{align}
and $\mT_x, \mT_y, \widehat{\mQ}_x, \widehat{\mQ}_y$ result from  a certain transformations of the following transition matrices
\begin{align}\mP_x 
&\triangleq \begin{bmatrix}
\widehat{\Lambda}_{a_{x}}\widehat{\Lambda}_{c_{x}} -\widehat{\Lambda}^2_{b_x} & - \widehat{\Lambda}_{b_x} \\
\widehat{\Lambda}_{b_x}  & \mathrm{I}_{(K-1)d_1} 
\end{bmatrix}\in \mathbb{R}^{2(K-1)d_1\times 2(K-1)d_1}, 
\\
\mP_y &\triangleq \begin{bmatrix}
\widehat{\Lambda}_{a_{y}}\widehat{\Lambda}_{c_{y}} -\widehat{\Lambda}^2_{b_y}& - \widehat{\Lambda}_{b_y} \\
\widehat{\Lambda}_{b_y}  & \mathrm{I}_{(K-1)d_2}
\end{bmatrix} \in \mathbb{R}^{2(K-1)d_2\times 2(K-1)d_2}, \notag
\end{align}
with
$
\mP_x = \widehat{\mQ}_x \mT_x \widehat{\mQ}^{-1}_x$ and $
\mP_y = \widehat{\mQ}_y \mT_y \widehat{\mQ}^{-1}_y$ and $\|\mT_x\| <1, \|\mT_y\|<1$.
Furthermore, 
choosing design matrices $\{\mA_x, \mB_x, \mC_x, \mA_y,\mB_y, \mC_y\}$
according to Table \ref{tab:matrix_choices} and further define the following constants
\begin{align}
    \rho_x&\triangleq \|\mT_x\|, &
    \quad 
    \rho_y & \triangleq
    \|\mT_y\|,  &\quad \rho \triangleq& \max \{\rho_x, \rho_y\}, \notag \\ 
    \lambda_{a_x} &\triangleq \|\widehat{\Lambda}_{a_x}\|, &\quad \lambda_{a_y} &\triangleq \|\widehat{\Lambda}_{a_y}\|, &\quad \lambda_a \triangleq&  \max\{\lambda_{a_x} , \lambda_{a_y} \},\notag\\ 
\underline{\lambda_{b_x}} &\triangleq  \frac{1}{\|\widehat{\Lambda}_{b_x}^{-1}\|}, &\quad 
     \underline{\lambda_{b_y}} &\triangleq  \frac{1}{\|\widehat{\Lambda}_{b_y}^{-1}\|}, &\quad 
     \frac{1}{\underline{\lambda^2_b} } \triangleq& \max \Big\{ \frac{1}{\underline{\lambda^2_{b_x}}},  \frac{1}{\underline{\lambda^2_{b_y}}}  \Big\}\label{appendix:definitionforlambda},
\end{align}
we have \\
\begin{itemize}
    \item \textbf{ED.} If $W \ge 0$, then  
$
\rho_x = \rho_y = \sqrt{\lambda}, \quad \lambda_{a_x} = \lambda_{a_y} = \lambda, \quad \underline{\lambda_{b_x}} = \underline{\lambda_{b_y}} = \sqrt{1-\lambda}.$
\item\textbf{EXTRA.} If $W \ge 0$, then
$
 \rho_x= \rho_y = \sqrt{\lambda}, \quad \lambda_{a_x} = \lambda_{a_y} = 1, \quad \underline{\lambda_{b_x}} = \underline{\lambda_{b_y}} = \sqrt{1-\lambda}
.$
\item \textbf{GT.} If $W \ge 0$, then
$
 \rho_x= \rho_y \le \frac{1+\lambda}{2}, \quad  \lambda_{a_x} = \lambda_{a_y} = \lambda^2, \quad \underline{\lambda_{b_x}} = \underline{\lambda_{b_y}} = 1-\lambda
.$
\end{itemize}
Here, $\lambda = \max_{i=2,\dots, K} |\lambda_i|$ is the second largest absolute  eigenvalue of $W$. 
\end{Lemma}

\bigskip

\section{Useful Inequalities}
\label{appendix:sec:useful}
We list several useful inequalities below.

% $\bullet$ 
% Given a nonnegative sequence  $\{w_i\}$
% and a constant $\alpha \in [0,1)$,
% we have 
% \begin{subequations}
% \begin{align}
% &\sum_{i=1}^{T}\alpha^i \text{ or } 
% \sum_{i=0}^{T-1}\alpha^i \text{ or }
% \sum_{j=0}^{i-1}
% \alpha^{i-j-1}  \le \frac{1}{1-\alpha} ,\label{appendix:sequence:inequality1} 
% \\
% \sum_{i=1}^T\sum_{j=0}^{i-1} \alpha^{i-j-1} w_j
% &=\sum_{i=0}^{T-1}
% \Big(\sum^{T-1-i}_{j=0} \alpha^j\Big) w_i \leq  \frac{1}{1-\alpha}\sum_{i=0}^{T-1} w_i,
% \label{appendix:sequence:inequality3}
% \\
% \sum_{i=1}^T\sum_{j=0}^{i-1} \alpha^{i-j-1} w_j
% &=\sum_{i=0}^{T-1}
% \Big(\sum^{T-1-i}_{j=0} \alpha^j\Big) w_i \leq 
% \sum_{i=0}^{T-1}
% \frac{1- \alpha^{T-i}}{1-\alpha} w_i,
% \label{appendix:sequence:inequality33}
% \\
% \sum_{i=1}^T\sum_{j=1}^{i} \alpha^{i-j} w_j
% &= 
% \sum_{i=1}^{T}
% \Big(\sum^{T-i}_{j=0} \alpha^j\Big) w_i \le \frac{1}{1-\alpha}\sum_{i=1}^{T} w_i,\label{appendix:sequence:inequality4} \\
% \sum_{i=0}^{T-1}
%  \sum_{j=0}^{i} \alpha^{i-j} w_{j} &=
%  \sum_{i=0}^{T-1}
%  \Big(
%  \sum_{j=0}^{T-1-i} \alpha^j
%  \Big)w_i  \le \frac{1}{1-\alpha} \sum_{i=0}^{T-1} w_i
%  ,  \label{appendix:sequence:inequality5}\\
%  \sum_{i=0}^{T-1}
%  \sum_{j=0}^{i} \alpha^{i-j} w_{j+1}
%  &=
%  \sum_{i=0}^{T-1}
%  \Big( \sum_{j=0}^{T-1-i}\alpha^{j}\Big) w_{i+1} \le \frac{1}{1-\alpha}
%  \sum_{i=0}^{T-1} w_{i+1} =\frac{1}{1-\alpha}
%  \sum_{i=1}^{T} w_{i} 
%  \label{appendix:sequence:inequality6}.
% \end{align}
% \end{subequations}

$\bullet$
Given nonnegative sequences $\{a_i\}_{i=0}^{T}, \{b_{i}\}_{i=0}^{T}$ and the recursion $a_i \le \alpha a_{i-1} +b_i+b_{i-1}, \alpha \in [0, 1), \forall i \ge 1$,
we have $a_{i} \le \alpha^{i} a_0+\sum_{j=0}^{i-1}\alpha^{i-j-1} b_{j+1}+\sum_{j=0}^{i-1}\alpha^{i-j-1} b_j $ and
\begin{align}
\sum_{i=1}^{T}a_i &\le
\sum_{i=1}^{T}(\alpha^{i} a_0+\sum_{j=0}^{i-1}\alpha^{i-j-1} b_{j+1}+\sum_{j=0}^{i-1}\alpha^{i-j-1} b_j) \notag\\
&\le \frac{a_0}{1-\alpha}+
\sum_{i=1}^{T}b_i(\sum_{j=0}^{T-i} \alpha^j)
+\sum_{i=0}^{T-1}b_i(\sum_{j=0}^{T-i-1} \alpha^j)
\notag\\
&\le \frac{a_0}{1-\alpha}+\frac{2}{1-\alpha}\sum_{i=1}^{T}b_i
+\frac{b_0}{1-\alpha}.
\label{proof:useful:recursion_first}
\end{align}

$\bullet$
For coupled consensus error $\|\mZ_{i}- \mZ_{c,i}\|^2 $, we can derive that 
\begin{align}
&\|\mZ_{i}- \mZ_{c,i}\|^2  \notag \\
&= 
\Big\|\Big(\Big(\mathrm{I}_{K} - \frac{1}{K} \mathds{1}_{K}\mathds{1}^\top_{K}\Big)\otimes \mathrm{I}_{d_1} \Big) \mX_{i} \Big\|^2
\notag \\
&\quad +\Big\|\Big(\Big(\mathrm{I}_{K} - \frac{1}{K} \mathds{1}_{K}\mathds{1}^\top_{K}\Big)\otimes \mathrm{I}_{d_2}\Big) \mY_{i} \Big\|^2
\notag \\
&=
\Big\|\Big(\widehat{U} \widehat{U}^\top\otimes \mathrm{I}_{d_1}\Big) \mX_i \Big\|^2 + 
\Big\|\Big(\widehat{U} \widehat{U}^\top\otimes I_{d_2}\Big)\mY_i\Big\|^2 \notag \\
&= 
\Big\|\widehat{\mU}_{x} \widehat{\mU}^\top_x \mX_i\Big\|^2
+\Big\|\widehat{\mU}_{y} \widehat{\mU}^\top_y \mY_i\Big\|^2 = 
\Big\|\widehat{\mU}^\top_x \mX_i\|^2 
+\|\widehat{\mU}^\top_y \mY_i\|^2 \notag 
\\
& \overset{(a)}{\le} 
\tau^2_x \|\widehat{\mQ}_x \mhE_{x,i}\|^2
+ 
\tau^2_y \|\widehat{\mQ}_y \mhE_{y,i}\|^2 
\overset{(b)}{\le} 
Kv^2_1v^2_2 \|\mhE_{z,i}\|^2,
\label{proof:useful:consensus}
\end{align}
where $(a)$ follows from 
Lemma \ref{appendix:lemma:transformed_recursion}, while $(b)$ is obtained by  defining
\begin{align}
v^2_1 \triangleq \max\Big\{ \|\widehat{\mQ}_x\|^2, \|\widehat{\mQ}_y\|^2\Big\}, \quad v^2_2 \triangleq \max\Big\{ \|\widehat{\mQ}^{-1}_x\|^2, \|\widehat{\mQ}^{-1}_y\|^2\Big\},   \quad \mhE_{z,i} \triangleq [\mhE_{x,i}; \mhE_{y,i}]
\label{appendix:proof:consensus_error}
\end{align}
and setting 
$\tau^2_x = \tau^2_y = K v^2_2$.

$\bullet$
For the weight drift
$\|\mZ_{i}-\mZ_{i-1}\|^2 $, we can derive that 
\begin{align}
&\|\mZ_{i}-\mZ_{i-1}\|^2\notag \\
&\le 
\|\mZ_{i}-\mZ_{c,i}+\mZ_{c,i} - \mZ_{c,i-1} +\mZ_{c,i-1} - \mZ_{i-1}\|^2
\notag\\
&\le 
3\|\mZ_{i}-\mZ_{c,i}\|^2
+3 \|\mZ_{i-1}-\mZ_{c,i-1}\|^2 +3\|\mZ_{c,i}-\mZ_{c,i-1}\|^2\notag \\
&\le 
3Kv^2_1v^2_2\|\mhE_{z,i}\|^2
+
3Kv^2_1v^2_2\|\mhE_{z,i-1}\|^2 
+3K(\eta^2_x\|\bg^x_{c,i-1}\|^2
+\eta^2_y\|\bg^y_{c,i-1}\|^2),
\label{appendix:proof:incremental}
\end{align}
where the last inequality follows from 
\eqref{proof:useful:consensus}
and Lemma \ref{appendix:lemma:transformed_recursion}.

$\bullet$
For in-expectation gradient norm, we can derive that 
\begin{align}
&\mE\|\nabla^x_{c,i}\|^2
+\mE\|\nabla^y_{c,i}\|^2 \notag \\
&= 
\mE\|\nabla_x J(\bz_{c,i}) - 
\nabla_x J(\bx_{c,i}, \by^o(\bx_{c,i})) + \nabla_x J(\bx_{c,i}, 
\by^o(\bx_{c,i}))\|^2
+\mE\|\nabla_y J(\bz_{c,i}) - \nabla_yJ(\bx_{c,i}, \by^o(\bx_{c,i})) \|^2\notag \\
&\le 
3L^2_f\mE\|\by_{c,i} - \by^o(\bx_{c,i})\|^2
+2\mE\|\nabla P(\bx_{c,i})\|^2 \notag \\
&\le 6\kappa L_f \mE \Delta^y_{c,i} 
+6\mE\|\nabla P(\bx_{c,i})\|^2 ,
\label{proof:useful:gradientnorm}
\end{align}
where the last inequality follows from the 
quadratic growth property of a $\nu$-PL function.
This relation implies  
\begin{align}
\mE\|\nabla^x_{c,i}\|^2
+\mE\|\nabla^y_{c,i}\|^2 = \mathcal{O}(\kappa L_f \mE \Delta^y_{c,i} 
+\mE\|\nabla P(\bx_{c,i})\|^2).
\label{appendix:useful:bound_gradient}
\end{align}

$\bullet$
For $\mE\|\mS_{z,0}\|^2$, using Lemma \ref{appendix:lemma:coordinate_large_bound}, we have 
\begin{align}
\mE\|\mS_{z,0}\|^2 \le  \frac{2KC_0}{b_0}, \quad \mE\|\bs^z_{c,0}\|^2 \le  \frac{2C_0}{Kb_0}.
\end{align}
For smoothing factors satisfying $\delta_x \le \frac{\sigma}{L_f\sqrt{d_1}}, \delta_y\le \frac{\sigma}{L_f\sqrt{d_2}}$,
we can conclude  
\begin{align}
C_0 = \mathcal{O}(\sigma^2), \ \mE\|\mS_{z,0}\|^2 = \mathcal{O} \Big(
\frac{K\sigma^2}{b_0}
\Big), \ 
\mE\|\bs^z_{c,0}\|^2 = \mathcal{O} \Big(
\frac{\sigma^2}{Kb_0}
\Big).
\label{appendix:useful_Sxy}
\end{align}

$\bullet$ From Lemma \ref{appendix:lemma:transformed_recursion},  setting $\tau^2_x = \tau^2_y = Kv^2_2$ and $\bx_{1,0}= \ldots= \bx_{K,0}, \by_{1,0}= \ldots= \by_{K,0}$, we can obtain
\begin{align}
&\mE\|\mhE_{z,0}\|^2
\notag \\
&\overset{(a)}{\le} 
\frac{
\mE\|\widehat{\mU
}^\top_x\mZ_{x,0}\|^2
+\mE\|\widehat{\mU
}^\top_y\mZ_{y,0}\|^2
}{K\underline{\lambda^2_b}}  
\notag 
\\
&= \frac{
\mE\|\widehat{\mU
}_x
\widehat{\mU
}^\top_x\mZ_{x, 0}\|^2
+
\mE\|\widehat{\mU
}_y
\widehat{\mU
}^\top_y\mZ_{y, 0}\|^2}{{K\underline{\lambda^2_b}} }
\notag \\
& =
\frac{1}{{K\underline{\lambda^2_b}} }\mE\Big\|\Big(\mathrm{I}_{K}
-\frac{1}{K} \mathds{1}_K \mathds{1}^\top_K  \Big)\otimes \mathrm{I}_{d_1}\mZ_{x,0}\Big\|^2 +\frac{1}{{K\underline{\lambda^2_b}} }
\mE\Big\|\Big(\mathrm{I}_{K}
-\frac{1}{K} \mathds{1}_{K} \mathds{1}^\top_{K}  \Big)\otimes \mathrm{I}_{d_2}\mZ_{y, 0}\Big\|^2
\notag
\\
&\overset{(b)}{=}
\frac{\eta^2_x}{K\underline{\lambda^2_b}} 
\mE\Big\|
\Big(\mA_{x}  - 
\frac{1}{K}
\mathds{1}_{K}
\mathds{1}^\top_{K}
\otimes \mathrm{I}_{d_1}\Big)\mG_{x,0}
\Big\|^2 +
\frac{\eta^2_y}{K\underline{\lambda^2_b}}  
\mE\Big\|
\Big(\mA_{y}  - 
\frac{1}{K}
\mathds{1}_{K}
\mathds{1}^\top_{K}
\otimes \mathrm{I}_{d_2}\Big)\mG_{y,0}
\Big\|^2
\overset{(c)}{\le} 
\frac{\eta^2_y \zeta^2_{0}}{\underline{\lambda^2_b}},
\label{appendix:useful_Exy}
\end{align}
where $(a)$ follows by defining 
$\frac{1}{\underline{\lambda^2_b} } \triangleq \max \Big\{ \frac{1}{\underline{\lambda^2_{b_x}}},  \frac{1}{\underline{\lambda^2_{b_y}}}  \Big\}$ (see \eqref{appendix:definitionforlambda}) and $\widehat{\mU}^\top_x\mX_0 =0, \widehat{\mU}^\top_y\mY_0 =0$,
$(b)$ follows from Assumption \ref{main:assumptions:combinationmatrix} together with the identical initialization of $\bd^x_{k,0}, \bd^y_{k,0}$ such that 
$\mB_x \mD_{x,0} =0, \mB_y \mD_{y,0} =0$, and $(c)$ follows by choosing  $\eta_x \le \eta_y$ and defining
\begin{align}
\zeta^2_{0} 
&\triangleq \frac{1}{K}
\mE\Big\|
\Big(\mA_{x}  - 
\frac{1}{K}
\mathds{1}_{K}
\mathds{1}^\top_{K}
\otimes \mathrm{I}_{d_1}\Big)\mG_{x,0}
\Big\|^2 +\frac{1}{K}
\mE\Big\|
\Big(\mA_{y}  - 
\frac{1}{K}
\mathds{1}_{K}
\mathds{1}^\top_{K}
\otimes \mathrm{I}_{d_2}\Big)\mG_{y,0}
\Big\|^2.
\end{align}

\section{Proof of pure coordinate ZO estimator}
We prove the convergence guarantee for the hybrid ZO estimator presented in Theorem \ref{main:theorem1}.
We start by establishing some key lemmas.
\subsection{Key Lemmas}
\begin{Lemma}[\textbf{Network gradient error}]
\label{appendix:lemma:coorinate_variance}
Under Assumptions \ref{main:assumption:Lipschitz}–\ref{main:assumption:boundedvariance}, we consider $\bu^j_{k,i} \equiv 0$ and
$\bar{\beta} \triangleq p +\beta - p \beta \le 1$, it holds that
\begin{align}
&\frac{1}{T}
\sum_{i=0}^{T-1}
\mE\|\mS_{z,i}\|^2\le 
\frac{2\mS_0}{T\bar{\beta}}
+ 
\frac{72KL^2_fv^2_1v^2_2}{bT\bar{\beta}}
\sum_{i=1}^{T}
\mE\|\mhE_{z,i}\|^2 +\frac{36KL^2_fv^2_1v^2_2 \mhE_0}{bT\bar{\beta}}+\frac{36KL^2_f}{bT\bar{\beta}}
\sum_{i=0}^{T-1}
(\eta^2_x\mE\|\bg^x_{c,i}\|^2 + \eta^2_y\mE\|\bg^y_{c,i}\|^2)
\notag \\
& 
\quad + \frac{K(C'_1+C'_2)}{\bar{\beta}},
\end{align}
where  $\mS_0\triangleq \mE\|\mS_{z,0}\|^2$ and $\mhE_0 \triangleq \mE\|\mhE_{z,0}\|^2$ and 
\begin{align}
C'_1 \triangleq \frac{12(d_1L^2_f\delta^2_x+d_2L^2_f\delta^2_y)}{b} , \quad
C'_2 \triangleq \frac{2pC_0}{B}\mathbb{I}(B<N)
+ \frac{4\beta^2C_0}{b} 
\end{align}
\end{Lemma}
\begin{proof}
Note that $\forall k \in [K]$,
we have 
\begin{align}
&\mE\|\bg^x_{k,i} - \hnabla_x J_{k}(\bz_{k,i})\|^2\notag\\
&\overset{(a)}{=}
p\mE\Big\|
 \bq^x_{i,0}(\bz_{k,i};B)
- \hnabla_x J_k (\bz_{k,i})\Big\|^2 
+(1-p)\mE\Big\|(1-\beta)\Big(\bg^x_{k,i-1} -  \bq^x_{i,0}(\bz_{k,i-1};b)\Big)
+\bq^x_{i,0}(\bz_{k,i};b)
   - \hnabla_x J_k(\bz_{k,i})\Big\|^2,
\label{appendix:proof:CW_ZO:grad_error1}
\end{align}
where $(a)$ follows from the probabilistic gradient estimator.
 The first term of 
\eqref{appendix:proof:CW_ZO:grad_error1}
corresponds to the event $\bpi_i=1$. 
Invoking 
Lemma \ref{appendix:lemma:coordinate_large_bound}, we can bound it as
\begin{align}
\mE\Big\|
 \bq^x_{i,0}(\bz_{k,i};B)
- \hnabla_x J_k (\bz_{k,i})\Big\|^2 \le \frac{C_0}{B} \mathbb{I}(B < N).
\end{align}
The second term corresponds to the event 
$\bpi_i=0$, we have 
\begin{align}
&\mE\Big\|(1-\beta)(
\bg^x_{k,i-1} -  \bq^x_{i,0}(\bz_{k,i-1};b)
)
+  \bq^x_{i,0}(\bz_{k,i};b) -\hnabla_x J_k(\bz_{k,i})\Big\|^2
\notag \\
&=
\mE\Big\|(1-\beta)(
\bg^x_{k,i-1} - \hnabla_x J_k(\bz_{k,i-1})
)+(1-\beta)
\Big( \bq^x_{i,0}(\bz_{k,i};b)
 - \bq^x_{i,0}(\bz_{k,i-1};b)+ \hnabla_x J_k(\bz_{k,i-1})
-\hnabla_x J_k(\bz_{k,i})
\Big) 
\notag\\
&\quad+ \beta \Big(\bq^x_{i,0}(\bz_{k,i};b) -\hnabla_x J_k(\bz_{k,i})\Big)\Big\|^2
\notag 
\\
&\overset{(a)}{\le}   
(1-\beta)^2
\mE\|\bg^x_{k,i-1} - \hnabla_x J_k(\bz_{k,i-1})\|^2
+\mE\Big\|
(1-\beta)
\Big( \bq^x_{i,0}(\bz_{k,i};b)
-  \bq^x_{i,0}(\bz_{k,i-1};b) + \hnabla_x J_k(\bz_{k,i-1})-\hnabla_x J_k(\bz_{k,i})
\Big) 
\notag \\
&\quad 
+ \beta \Big(\bq^x_{i,0}(\bz_{k,i};b) -\hnabla_x J_k(\bz_{k,i})\Big)\Big\|^2 \notag 
\\
&\overset{(b)}{\le} 
(1-\beta)
\mE\|\bg^x_{k,i-1} - \hnabla_x J_k(\bz_{k,i-1})\|^2
+
2\mE\Big\|\Big( \bq^x_{i,0}(\bz_{k,i};b)
 - \bq^x_{i,0}(\bz_{k,i-1};b)+ \hnabla_x J_k(\bz_{k,i-1})
-\hnabla_x J_k(\bz_{k,i})
\Big) \Big\|^2
\notag\\
&\quad +2\beta^2
\mE\Big\|\bq^x_{i,0}(\bz_{k,i};b) -\hnabla_x J_k(\bz_{k,i})\Big\|^2,
\end{align}
where $(a)$ follows from \begin{align}
   \mE_{\bxi^{j}_{k,i}}[\bq^x_{i,0}(\bz_{k,i};b) ] = \hnabla_x J_k(\bz_{k,i}), \notag\\ 
   \mE_{\bxi^{j}_{k,i}}[ \bq^x_{i,0}(\bz_{k,i-1};b)] = 
\hnabla_x J_k(\bz_{k,i-1}),
\end{align}
and the fact that  $\{\bxi^{j}_{k,i}\}$ is i.i.d. sequence;
$(b)$ follows from Jensen's inequality and $(1-\beta)^2 \le (1-\beta) \le 1$.
Note that 
\begin{align}
&\mE\Big\|\bq^x_{i,0}(\bz_{k,i};b) -\hnabla_x J_k(\bz_{k,i})\Big\|^2 \notag \\
&\le 
\frac{3}{b^2}\sum_{j=1}^b
\Big(\mE\|\bq^x_{k}(\bz_{k,i};\bxi^j_{k,i},0) - \nabla_x Q_k(\bz_{k,i};\bxi^j_{k,i})\|^2
+\mE\|\nabla_x Q_k(\bz_{k,i};\bxi^j_{k,i}) - \nabla_x J_k(\bz_{k,i})\|^2\notag 
+ \mE\| \hnabla_x J_k(\bz_{k,i}) - \nabla_x J_k(\bz_{k,i})\|^2\Big)\notag\\
&
\overset{(a)}{\le} 
\frac{C_0}{b}, 
\label{appendix:proof:minibbtach_variance}
\end{align}
where  $(a)$ follows from an argument similar to that used in Lemma \ref{appendix:lemma:coordinate_large_bound}.
On the other hand, we have 
\begin{align}
&\mE\Big\|\bq^x_{i,0}(\bz_{k,i};b)
-  \bq^x_{i,0}(\bz_{k,i-1};b) + \hnabla_x J_k(\bz_{k,i-1})
-\hnabla_x J_k(\bz_{k,i})
 \Big\|^2 \notag \\
&\overset{(a)}{=} \frac{1}{b^2}
\sum_{j=1}^b
\mE\|\bq^x_{k}(\bz_{k,i};\bxi^{j}_{k,i},0) - \bq^x_{k}(\bz_{k,i-1};\bxi^{j}_{k,i},0)  + \hnabla_x J_k(\bz_{k,i-1})
-\hnabla_x J_k(\bz_{k,i})\|^2
 \notag \\
&\overset{(b)}{\le} \frac{1}{b^2}
\sum_{j=1}^b
\mE\|\bq^x_{k}(\bz_{k,i};\bxi^{j}_{k,i},0) - \bq^x_{k}(\bz_{k,i-1};\bxi^{j}_{k,i},0)\|^2,
\end{align}
where $(a)$ follows from 
\begin{align}
&\mE_{\bxi^{j}_{k,i}}[\bq^x_{k}(\bz_{k,i};\bxi^{j}_{k,i},0) - \bq^x_{k}(\bz_{k,i-1};\bxi^{j}_{k,i},0)] = \hnabla_x J_k(\bz_{k,i})-
\hnabla_x J_k(\bz_{k,i-1}).
\end{align}
and the fact that the samples $\{\bxi^{j}_{k,i}\}$ are i.i.d. across $j$, $(b)$ follows from $\mE\|\bxi - \mE \bxi\|^2 \le 
\mE\|\bxi\|^2$.
Note that
\begin{align}
&\mE\|\bq^x_{k}(\bz_{k,i};\bxi^j_{k,i},0) - \bq^x_{k}(\bz_{k,i-1};\bxi^j_{k,i},0)\|^2 \notag \\
&\overset{(a)}{\le}
3\mE\|\bq^x_{k}(\bz_{k,i};\bxi^j_{k,i},0) - \nabla_x Q_k(\bz_{k,i};\bxi^{j}_{k,i})\|^2
 +3\mE\|\nabla_x Q_k(\bz_{k,i};\bxi^{j}_{k,i}) - 
\nabla_x Q_k(\bz_{k,i-1};\bxi^{j}_{k,i})\|^2
\notag \\
&\quad 
+3\mE\|\bq^x_{k}(\bz_{k,i-1};\bxi^j_{k,i},0)- 
\nabla_x Q_k(\bz_{k,i-1};\bxi^{j}_{k,i})
\|^2 \notag \\
&\overset{(b)}{\le} 
6d_1L^2_f\delta^2_x
+3L^2_f\|\bz_{k,i}-\bz_{k,i-1}\|^2,
\end{align}
where $(a)$
follows from Jensen's inequality 
and $(b)$
we invoked 
Lemma \ref{appendix:lemma:coordinate_zero_first}
and used Assumption \ref{main:assumption:Lipschitz}.
Putting results together,
we get 
\begin{align}
&\mE\|\bg^x_{k,i} - \hnabla_x J_{k}(\bz_{k,i})\|^2 \le \frac{pC_0}{B} \mathbb{I}(B < N) +
(1-p)(1-\beta) \mE\|\bg^x_{k,i-1} - \hnabla_x J_k(\bz_{k,i-1})\|^2+
\frac{12(1-p)d_1L^2_f\delta^2_x}{b}
\notag \\
&\quad 
+\frac{6(1-p)L^2_f\|\bz_{k,i} - \bz_{k,i-1}\|^2}{b} 
+\frac{2\beta^2C_0(1-p)}{b}.
\end{align}
We then use 
a symmetric argument to bound $\mE\|\bg^y_{k,i} - \hnabla_y J_{k}(\bz_{k,i})\|^2$.
% we get 
% \begin{align}
% &\mE\|\bg^y_{k,i} - \hnabla_y J_{k}(\bz_{k,i})\|^2\le \frac{pC_0}{B} \mathbb{I}(B < N) +
% (1-p)(1-\beta) \notag \\
% &\quad \times \mE\|\bg^y_{k,i-1} - \hnabla_y J_k(\bz_{k,i-1})\|^2
% +
% \frac{12(1-p)d_2L^2_f\delta^2_y}{b}
% \notag \\
% &\quad 
% +\frac{6(1-p)L^2_f\|\bz_{k,i} - \bz_{k,i-1}\|^2}{b}
% +\frac{2\beta^2C_0(1-p)}{b}.
% \end{align}
Combining the results over $k \in [K]$
and denoting $\bar{\beta} \triangleq p+\beta -p\beta$, we get
\begin{align}
&\mE\|\mS_{z,i}\|^2 
\notag \\
&\le 
(1-\bar{\beta})
\mE\|\mS_{z,i-1}\|^2
+ \frac{12L^2_f\mE\|\mZ_{i}-\mZ_{i-1}\|^2}{b} + \frac{4K\beta^2C_0}{b} 
+\frac{12K(d_1L^2_f\delta^2_x+d_2L^2_f\delta^2_y)}{b}
+\frac{2pKC_0}{B}\mathbb{I}(B<N)
\notag\\
&\overset{(a)}{\le} 
(1-\bar{\beta})
\mE\|\mS_{z,i-1}\|^2
+\frac{36KL^2_f}{b}
(\eta^2_x\mE\|\bg^x_{c,i-1}\|^2
+\eta^2_y\mE\|\bg^y_{c,i-1}\|^2)+\frac{36KL^2_fv^2_1v^2_2(\mE\|\mhE_{z,i}\|^2
+\mE\|\mhE_{z,i-1}\|^2 ) }{b}
\notag\\
&\quad+\frac{12K(d_1L^2_f\delta^2_x+d_2L^2_f\delta^2_y)}{b} +\frac{2pKC_0}{B}\mathbb{I}(B<N)
+ \frac{4K\beta^2C_0}{b},
\label{appendix:lemma_coordinate:step1}
\end{align}
where $(a)$ follows from the inequality 
\eqref{appendix:proof:incremental}.
Applying the above inequality recursively, it holds  for $i= 1, \ldots, T$
that
\begin{align}
&\mE\|\mS_{z,i}\|^2 \notag \\
&\le (1-\bar{\beta})^{i}
\mE\|\mS_{z,0}\|^2 
+\frac{36KL^2_fv^2_1v^2_2}{b}
\sum_{j=1}^{i} (1-\bar{\beta})^{i-j}\mE\|\mhE_{z,j}\|^2
+\frac{36KL^2_fv^2_1v^2_2}{b}
\sum_{j=0}^{i-1} (1-\bar{\beta})^{i-j-1}\mE\|\mhE_{z,j}\|^2  
\notag\\
&\quad 
+\frac{36KL^2_f}{b}\sum_{j=0}^{i-1}(1-\bar{\beta})^{i-j-1}
(\eta^2_x\mE\|\bg^x_{c,j}\|^2+\eta^2_y\mE\|\bg^y_{c,j}\|^2)+\frac{K}{\bar{\beta}}\underbrace{\frac{12(d_1L^2_f\delta^2_x+d_2L^2_f\delta^2_y)}{b}}_{\triangleq C'_1}
+\frac{K}{\bar{\beta}}\underbrace{\Big(\frac{2pC_0}{B}\mathbb{I}(B<N)+ \frac{4\beta^2C_0}{b}\Big)}_{\triangleq C'_2} , 
\label{appendix:lemma_coordinate:step2}
\end{align}
Averaging the above inequality over $i =1, \ldots, T$, we have 
\begin{align}
&\frac{1}{T}
\sum_{i=1}^{T}
\mE\|\mS_{z,i}\|^2  \label{appendix:lemma_coordinate:step_final}\\
% &\le 
% \frac{1}{T}\sum_{i=1}^{T}(1-\bar{\beta})^{i}
% \mE\|\mS_{z,0}\|^2+\frac{36KL^2_fv^2_1v^2_2}{bT}
% \sum_{i=1}^{T}\sum_{j=1}^{i}(1-\bar{\beta})^{i-j}\notag \\
% &\quad
% \mE\|\mhE_{z,j}\|^2 +\frac{36KL^2_fv^2_1v^2_2}{bT}\sum_{i=1}^{T}
% \sum_{j=0}^{i-1} (1-\bar{\beta})^{i-j-1}\mE\|\mhE_{z,j}\|^2\notag \\
% &\quad 
% +\frac{36KL^2_f}{bT}\sum_{i=1}^{T}\sum_{j=0}^{i-1}(1-\bar{\beta})^{i-j-1}
% (\eta^2_x\mE\|\bg^x_{c,j}\|^2+\eta^2_y\mE\|\bg^y_{c,j}\|^2)
% \notag \\
% &\quad 
% +\frac{K(C'_1+C'_2)}{\bar{\beta}}\notag \\
&\le
\frac{\mE\|\mS_{z,0}\|^2}{T\bar{\beta}}
+ 
\frac{36KL^2_fv^2_1v^2_2}{bT\bar{\beta}}
\Big(\sum_{i=1}^{T}
\mE\|\mhE_{z,i}\|^2+
\sum_{i=0}^{T-1}
\mE\|\mhE_{z,i}\|^2\Big) +\frac{36KL^2_f}{bT\bar{\beta}}
\sum_{i=0}^{T-1}
(\eta^2_x\mE\|\bg^x_{c,i}\|^2 + \eta^2_y\mE\|\bg^y_{c,i}\|^2)
+ \frac{K(C'_1+C'_2)}{\bar{\beta}}\notag \\
&\overset{(a)}{\le}  
\frac{\mE\|\mS_{z,0}\|^2}{T\bar{\beta}}
+ 
\frac{72KL^2_fv^2_1v^2_2}{bT\bar{\beta}}
\sum_{i=1}^{T}
\mE\|\mhE_{z,i}\|^2
+\frac{36KL^2_fv^2_1v^2_2}{bT\bar{\beta}}
\mE\|\mhE_{z,0}\|^2+\frac{36KL^2_f}{bT\bar{\beta}}
\sum_{i=0}^{T-1}
(\eta^2_x\mE\|\bg^x_{c,i}\|^2 + \eta^2_y\mE\|\bg^y_{c,i}\|^2)
+ \frac{K(C'_1+C'_2)}{\bar{\beta}}, \notag
\end{align}
where $(a)$ follows from \eqref{proof:useful:recursion_first}.
Denoting $\mS_0 \triangleq \mE\|\mS_{z,0}\|^2$,  
$\mhE_0 \triangleq \mE\|\mhE_{z,0}\|^2$,
and using
$\frac{1}{T\bar{\beta}} \ge \frac{1}{T}$, we can deduce that
$\frac{1}{T}
\sum_{i=0}^{T-1}
\mE\|\mS_{z,i}\|^2 
 \le \frac{1}{T}
\sum_{i=0}^{T}
\mE\|\mS_{z,i}\|^2   \le \frac{1}{T}
\sum_{i=1}^{T}
\mE\|\mS_{z,i}\|^2 + \mS_0/(T\bar{\beta})$. 
The proof is completed by plugging the results
\eqref{appendix:lemma_coordinate:step_final} into this relation.
 \end{proof}
\begin{Lemma}[\textbf{Averaged gradient error}]
\label{appendix:lemma:CW_Nxc_Nyc}
Under Assumptions \ref{main:assumption:Lipschitz}
--\ref{main:assumption:boundedvariance}, 
we consider $\bu^j_{k,i} \equiv 0$ and 
$\bar{\beta} \triangleq p +\beta - p \beta \le 1$,  it follows that
\begin{align}
&\frac{1}{T}
\sum_{i=0}^{T-1}
\mE\|\bs^z_{c,i}\|^2 \le 
\frac{2\bs_{c,0}}{T\bar{\beta}}
+ 
\frac{72L^2_fv^2_1v^2_2}{bKT\bar{\beta}}
\sum_{i=1}^{T}
\mE\|\mhE_{z,i}\|^2
+\frac{36L^2_fv^2_1v^2_2 \mhE_0}{bKT\bar{\beta}}+
\frac{36L^2_f}{bK\bar{\beta}T}
\sum_{i=0}^{T-1}
\eta^2_x\mE\|\bg^x_{c,i}\|^2
\notag \\
& 
\quad 
+\frac{36L^2_f\eta^2_y}{bK\bar{\beta}T}\sum_{i=0}^{T-1}
(1-(1-\bar{\beta})^{T-i}) \mE\|\bg^y_{c,i}\|^2
+\frac{C'_1+C'_2}{K\bar{\beta}}.
\label{appendix:lemma_coordinate:bmc:statement}
\end{align}
where $\bs_{c,0} \triangleq \mE\|\bs^z_{c,0}\|^2$ and $\mhE_0 \triangleq \mE\|\mhE_{z,0}\|^2$ and $C'_1, C'_2$ are defined in Lemma \ref{appendix:lemma:coorinate_variance}.
\end{Lemma}
\begin{proof}
We can complete the proof by following an argument similar to that used in Lemma \ref{appendix:lemma:coorinate_variance}.
\end{proof}
\begin{Lemma}[\textbf{Coupled consensus error}]
\label{appendix:lemma:cw_coupled}
Under Assumptions \ref{main:assumption:Lipschitz}, \ref{main:assumption:boundedvariance},
we consider $\bu^j_{k,i} \equiv 0$ and
\begin{align}
 \eta_x \le \eta_y, \ \eta_y \le 
\frac{(1-\rho)\underline{\lambda_b}}{54\lambda_aL_fv_1v_2}, \ b\ge1,   \ \beta+bp \le b.
\end{align}
For the coupled consensus error $\mhE_{z,i} \triangleq [\mhE_{x,i}; \mhE_{y,i}]$,
we can bound it as follows
\begin{align}
&\frac{1}{T}\sum_{i=0}^{T-1}
\mE\|\mhE_{z,i}\|^2 \le\frac{1}{T}\sum_{i=1}^{T}
\mE\|\mhE_{z,i}\|^2 + \frac{\mhE_0}{T(1-\rho)}  \notag\\
&\le 
\frac{5\mhE_0}{T(1-\rho)}
+
L'_0\sum_{i=0}^{T-1}(\eta^2_x\mE\|\bg^x_{c,i}\|^2 +\eta^2_y\mE\|\bg^y_{c,i}\|^2) 
 I'_0+\frac{20\eta^2_y\lambda^2_a(p+\beta^2)(C'_1+C'_2)}{(1-\rho)^2\underline{\lambda^2_b}\bar{\beta}} 
+ \frac{4\eta^2_y\lambda^2_aC'_3}{(1-\rho)^2\underline{\lambda^2_b}},
\end{align}
where 
\begin{align}
C'_3 &\triangleq \frac{10pC_0}{B}\mathbb{I}(B<N)
+18L^2_f(d_1\delta^2_x+d_2\delta^2_y)+\frac{6\beta^2C_0}{b}, \\
L'_0 &\triangleq \frac{936\eta^2_y\lambda^2_aL^2_f}{(1-\rho)^2\underline{\lambda^2_b}T},
\quad 
I'_0 \triangleq \frac{20\eta^2_y\lambda^2_a(p+\beta^2)}{(1-\rho)^2\underline{\lambda^2_b}K}  \Big(
\frac{2\mS_0}{T\bar{\beta}} +\frac{36KL^2_fv^2_1v^2_2\mhE_0}{bT\bar{\beta}}\Big).
\end{align}
\end{Lemma}
\begin{proof}
Invoking Lemma \ref{appendix:lemma:transformed_recursion} and choosing $\eta_x \le \eta_y$, we can establish 
\begin{align}
&\|\mhE_{z,i+1}\|^2  \le  
\rho\|\mhE_{z,i}\|^2 + \frac{\eta^2_y\lambda^2_{a}\|\mG_{z,i+1} -\mG_{z,i}\|^2}{(1-\rho)\underline{\lambda^2_{b}}K}. 
\label{appendix:proof:lemma_exey_coordinate_step1}
\end{align}
% Recall the transformed recursion 
% \begin{align}
% \mhE_{x,i+1}
% = \mT_x \mhE_{x,i} - \frac{\eta_x}{\tau_x}
% \mQ^{-1}_x
% \begin{bmatrix}
% 0\\
% \widehat{\Lambda}^{-1}_{b_x}\widehat{\Lambda}_{a_x} \widehat{\mU}^\top_x(\mG_{x,i} -\mG_{x,i+1})
% \end{bmatrix}.
% \label{appendix:proof:lemma_exey_coordinate_step1}
% \end{align}
% Applying Jensen's inequality
% $\|a+b\| \le \frac{1}{t}\|a\|^2 + \frac{1}{1-t}\|b\|^2 $ and selecting $ t  = \|\mT_x\| \in(0,1)$ for the above expression, we get 
% \begin{align}
% \|\mhE_{x,i+1}\|^2 \le 
% \|\mT_x\|\|\mhE_{x,i}\|^2 + \frac{\eta^2_x\|\mQ^{-1}_x\|^2\|\widehat{\Lambda}^{-1}_{b_x}\|^2\|\widehat{\Lambda}_{a_x}\|^2\|\widehat{\mU}^\top_x\|^2\|\mG_{x,i+1} -\mG_{x,i}\|^2}{(1-\|\mT_x\|)\tau^2_x},
% \label{appendix:proof:lemma_exey_coordinate_step2}
% \end{align}
% where $\|\widehat{\mU}^\top_x\|^2 \le 1$.
% Denoting 
% \begin{align}
% \rho \triangleq \max\Big\{\|\mT_x\|, \|\mT_y\|\Big\}, \lambda^2_a \triangleq \max\Big\{\|\widehat{\Lambda}_{a_x}\|^2, \|\widehat{\Lambda}_{a_y}\|^2\Big\}, \underline{\frac{1}{\lambda^2_b}}
% \triangleq\max\Big\{\|\widehat{\Lambda}^{-1}_{b_x}\|^2,\|\widehat{\Lambda}^{-1}_{b_y}\|^2\Big\},
% \label{appendix:proof:lemma_exey_coordinate_step3}
% \end{align}
% and setting $\tau^2_x = Kv^2_2 $, where $ v^2_2 = \max\{\|\mQ^{-1}_x\|^2, \|\mQ^{-1}_y\|^2\}$, it follows that 
% \begin{align}
% \|\mhE_{x,i+1}\|^2 \le 
% \rho\|\mhE_{x,i}\|^2 + \frac{\eta^2_x\lambda^2_{a}\|\mG_{x,i+1} -\mG_{x,i}\|^2}{(1-\rho)\underline{\lambda^2_{b}}K} . \label{appendix:proof:lemma_exey_coordinate_step4}
% \end{align}
We use the probabilistic gradient estimator $\forall k \in [K]$, it follows that
\begin{align}
&\mE\|\bg^x_{k,i+1} - \bg^x_{k,i}\|^2 =
p\mE\Big\|\bq^x_{i+1,0}(\bz_{k,i+1};B) - \bg^x_{k,i}\Big\|^2 + (1-p)\mE\Big\|
(1-\beta)\Big(\bg^x_{k,i} - \bq^x_{i+1,0}(\bz_{k,i};b)\Big)
 + \bq^x_{i+1,0}(\bz_{k,i+1};b) - \bg^x_{k,i}
\Big\|^2.
\end{align}
The first term corresponds to the event $\bpi_{i+1} =1$, we have
\begin{align}
&\mE\Big\|\bq^x_{i+1,0}(\bz_{k,i+1};B) - \bg^x_{k,i}\Big\|^2  \notag \\
&=
\mE\Big\|\bq^x_{i+1,0}(\bz_{k,i+1};B) - \hnabla_x J_k(\bz_{k,i+1})
+
\hnabla_x J_k(\bz_{k,i+1}) 
 - \nabla_x J_k(\bz_{k,i+1}) +
\nabla_x J_k(\bz_{k,i+1})-\nabla_x J_k(\bz_{k,i})
\notag\\&\quad+
\nabla_x J_k(\bz_{k,i})-\hnabla_x J_k(\bz_{k,i}) 
+\hnabla_x J_k(\bz_{k,i}) 
- \bg^x_{k,i}\Big\|^2 \notag \\
&\le 
5\mE\Big\|\bq^x_{i+1,0}(\bz_{k,i+1};B) - \hnabla_x J_k(\bz_{k,i+1})\Big\|^2
+5\mE\|
\hnabla_x J_k (\bz_{k,i+1})- \nabla_x J_k(\bz_{k,i+1})\|^2+5\mE\|\nabla_x J_k(\bz_{k,i+1})-\nabla_x J_k(\bz_{k,i})\|^2\notag \\
&\quad 
+5\mE\|\nabla_x J_k(\bz_{k,i})-\hnabla_x J_k(\bz_{k,i}) \|^2
+5\mE\|\hnabla_x J_k(\bz_{k,i}) 
- \bg^x_{k,i}\|^2
\notag \\
&\overset{(a)}{\le}
\frac{5C_0}{B}\mathbb{I}(B<N) + 10d_1L^2_f\delta^2_x
+5L^2_f\mE\|\bz_{k,i+1}-\bz_{k,i}\|^2 +5\mE\|\hnabla_x J_k(\bz_{k,i}) 
- \bg^x_{k,i}\|^2,
\end{align}
where $(a)$ follows from Lemmas \ref{appendix:lemma:coordinate_zero_first} and  \ref{appendix:lemma:coordinate_large_bound}
and Assumption \ref{main:assumption:Lipschitz}.
The second term can be bounded as follows 
\begin{align}
&\mE\Big\|
(1-\beta)\Big(\bg^x_{k,i} - \bq^x_{i+1,0}(\bz_{k,i};b)\Big)+ \bq^x_{i+1,0}(\bz_{k,i+1};b) - \bg^x_{k,i}
\Big\|^2 \notag \\
&=
\mE\Big\|
-\beta(\bg^x_{k,i} - \hnabla_x J_k(\bz_{k,i}))
-
\Big(\bq^x_{i+1,0}(\bz_{k,i};b) 
 - \bq^x_{i+1,0}(\bz_{k,i+1};b)\Big)
+\beta \Big(
\bq^x_{i+1,0}(\bz_{k,i};b) - \hnabla_x J_k(\bz_{k,i})
\Big)
\Big\|^2 \notag \\
&\le 3\beta^2
\mE\|\bg^x_{k,i} - \hnabla_x J_k(\bz_{k,i})\|^2
+3 \mE\Big\|
\bq^x_{i+1,0}(\bz_{k,i+1};b) 
- \bq^x_{i+1,0}(\bz_{k,i};b) 
\Big\|^2+3\beta^2 \mE\Big\|
\bq^x_{i+1,0}(\bz_{k,i};b)   - \hnabla_x J_k(\bz_{k,i})
\Big\|^2 \notag \\
&\overset{(a)}{\le} 
3\beta^2
\mE\|\bg^x_{k,i} - \hnabla_x J_k(\bz_{k,i})\|^2 +3\mE\Big\|
\frac{1}{b}\sum_{j=1}^{b} (\bq^x_{k}(\bz_{k,i+1};\bxi^j_{k,i+1},0)- 
\nabla_x Q_k(\bz_{k,i+1};\bxi^j_{k,i+1}))\notag\\
&\quad 
 +
\frac{1}{b}\sum_{j=1}^b\Big(\nabla_x Q_k(\bz_{k,i+1};\bxi^j_{k,i+1}) -
\nabla_x Q_k(\bz_{k,i};\bxi^j_{k,i+1})\Big)
 + \frac{1}{b}\sum_{j=1}^{b}\Big(\nabla_x Q_k(\bz_{k,i};\bxi^j_{k,i+1})-\bq^x_{k}(\bz_{k,i};\bxi^j_{k,i+1},0)\Big)
\Big\|^2
\notag\\
&\quad + \frac{3\beta^2C_0}{b} \notag \\
&\overset{(b)}{\le} 
3\beta^2\mE\|\bs^x_{k,i}\|^2
+9L^2_f\mE\|\bz_{k,i+1}-\bz_{k,i}\|^2
+ 18d_1L^2_f\delta^2_x
+\frac{3\beta^2C_0}{b},
\end{align}
where $(a)$ follows from Lemma 
\ref{appendix:lemma:coordinate_large_bound} and \eqref{appendix:proof:minibbtach_variance};
$(b)$ follows from Jensen's inequality,  Lemma \ref{appendix:lemma:coordinate_zero_first}
and Assumption \ref{main:assumption:Lipschitz}.
Collecting the results  $\forall k \in [K]$,
we can derive 
\begin{align}
&\mE\|\mG_{x,i+1} - \mG_{x,i}\|^2 \notag \\
% &=
% 5p\mE\|\mS_{x,i}\|^2 + 3\beta^2(1-p)
% \mE\|\mS_{x,i}\|^2+5pL^2_f\mE\|\mZ_{i+1} -\mZ_{i}\|^2 
% \notag\\
% &\quad +9L^2_f(1-p)\mE\|\mZ_{i+1} - \mZ_{i}\|^2
% + p\Big(
% \frac{5KC_0}{B}\mathbb{I}(B<N)
% \notag\\
% &\quad +10Kd_1L^2_f\delta^2_x
% \Big) 
% +(1-p)\Big(18d_1KL^2_f\delta^2_x
% +\frac{3\beta^2KC_0}{b}
% \Big) \notag\\
&\overset{(a)}{\le} 
5(p+\beta^2)
\mE\|\mS_{x,i}\|^2
+27KL^2_fv^2_1v^2_2(\mE\|\mhE_{z,i+1}\|^2
 +\mE\|\mhE_{z,i}\|^2)+27KL^2_f(\eta^2_x\mE\|\bg^x_{c,i}\|^2 +\eta^2_y\mE\|\bg^y_{c,i}\|^2)\notag\\
&\quad 
+\Big(
\frac{5pC_0}{B}\mathbb{I}(B<N)
+18d_1L^2_f\delta^2_x+\frac{3\beta^2C_0}{b}
\Big)K,
\end{align}
where $(a)$ follows from inequality \eqref{appendix:proof:incremental} and $p\le 1$.
% Plugging the above results into 
% \eqref{appendix:proof:lemma_exey_coordinate_step4}, we get 
% \begin{align}
% &\mE\|\mhE_{x,i+1}\|^2
% \notag \\
% &\le \rho 
% \mE\|\mhE_{x,i}\|^2
% + \frac{\eta^2_x\lambda^2_a}{(1-\rho)\underline{\lambda^2_b}K}\Bigg(
% 5(p+\beta^2)
% \mE\|\mS_{x,i}\|^2
% +27KL^2_fv^2_1v^2_2(\mE\|\mhE_{x,i+1}\|^2
% \notag \\
% &
% \quad +\mE\|\mhE_{y,i+1}\|^2)
% +27KL^2_fv^2_1v^2_2(\mE\|\mhE_{x,i}\|^2+\mE\|\mhE_{y, i}\|^2)+27KL^2_f(\eta^2_x\mE\|\bg^x_{c,i}\|^2 +\eta^2_y\mE\|\bg^y_{c,i}\|^2)
% \notag\\
% &\quad 
% +\Big(
% \frac{5pKC_0}{B}\mathbb{I}(B<N)
% +18d_1KL^2_f\delta^2_x+\frac{3\beta^2KC_0}{b}
% \Big)
% \Bigg).
% \end{align}
We then use a symmetric argument to bound 
$\mE\|\mhE_{y,i+1}\|^2$. Adding up two inequalities into \eqref{appendix:proof:lemma_exey_coordinate_step1}, we obtain
\begin{align}
&\mE\|\mhE_{z,i+1}\|^2\notag \\
&\le 
\rho\mE\|\mhE_{z,i}\|^2
+\frac{5\eta^2_y\lambda^2_a(p+\beta^2)}{(1-\rho)\underline{\lambda^2_b}K}
\mE\|\mS_{z,i}\|^2 + \frac{54\eta^2_y\lambda^2_aL^2_fv^2_1v^2_2}{(1-\rho)\underline{\lambda^2_b}}
 \Big(\mE\|\mhE_{z,i+1}\|^2+\mE\|\mhE_{z,i}\|^2
\Big) + \frac{54\eta^2_y\lambda^2_aL^2_f}{(1-\rho)\underline{\lambda^2_b}}(\eta^2_x\mE\|\bg^x_{c,i}\|^2
\notag\\
&\quad +\eta^2_y\mE\|\bg^y_{c,i}\|^2)+ \frac{\eta^2_y\lambda^2_a}{(1-\rho)\underline{\lambda^2_b}}  \underbrace{\Big(
\frac{10pC_0}{B}\mathbb{I}(B<N)
+18L^2_f(d_1\delta^2_x+d_2\delta^2_y)+\frac{6\beta^2C_0}{b}
\Big)}_{\triangleq C'_3}.
\end{align}
Let 
\begin{align}
\frac{54\eta^2_y\lambda^2_aL^2_fv^2_1v^2_2}{(1-\rho)\underline{\lambda^2_b}} \le \frac{1-\rho}{8} \Longrightarrow
\eta_y \le \frac{(1-\rho)\underline{\lambda_b}}{21\lambda_aL_fv_1v_2}. 
\end{align}
Also note that $\frac{1-\rho}{8} \le \frac{1-\rho}{2}$, we have 
\begin{align}
&\mE\|\mhE_{z,i+1}\|^2 \le 
\frac{1+\rho}{2}\mE\|\mhE_{z,i}\|^2+\frac{5\eta^2_y\lambda^2_a(p+\beta^2)}{(1-\rho)\underline{\lambda^2_b}K}
\mE\|\mS_{z,i}\|^2  + \frac{1-\rho}{8}\mE\|\mhE_{z,i+1}\|^2
+ \frac{54\eta^2_y\lambda^2_aL^2_f}{(1-\rho)\underline{\lambda^2_b}}(\eta^2_x\mE\|\bg^x_{c,i}\|^2+\eta^2_y\mE\|\bg^y_{c,i}\|^2)
\notag \\
&\quad 
 + \frac{\eta^2_y\lambda^2_aC'_3}{(1-\rho)\underline{\lambda^2_b}}.
\end{align}
Applying the above inequality recursively for $i=0, \ldots, T-1$
and averaging the results, we obtain 
\begin{align}
& \frac{1}{T}\sum_{i=0}^{T-1}
\mE\|\mhE_{z,i+1}\|^2  =\frac{1}{T}\sum_{i=1}^{T}
\mE\|\mhE_{z,i}\|^2 \notag \\
% &\le \frac{1}{T}
% \sum_{i=0}^{T-1}
% \Big(\frac{1+\rho}{2}\Big)^{i+1}
% \mE\|\mhE_{z,0}\|^2
% +\frac{5\eta^2_y\lambda^2_a(p+\beta^2)}{(1-\rho)\underline{\lambda^2_b}K} \frac{1}{T}
% \sum_{i=0}^{T-1}
% \sum_{j=0}^{i} \notag \\
% &\quad 
% \Big(\frac{1+\rho}{2}\Big)^{i-j}
% \mE\|\mS_{z,j}\|^2+\frac{1-\rho}{8}
% \frac{1}{T}
% \sum_{i=0}^{T-1}
% \sum_{j=0}^{i}\Big(
% \frac{1+\rho}{2}
% \Big)^{i-j}\mE\|\mhE_{z,j+1}\|^2 \notag \\
% &\quad  
% +\frac{54\eta^2_y\lambda^2_aL^2_f}{(1-\rho)\underline{\lambda^2_b}} \frac{1}{T}
% \sum_{i=0}^{T-1}\sum_{j=0}^{i}
% \Big(
% \frac{1+\rho}{2}
% \Big)^{i-j}(\eta^2_x\mE\|\bg^x_{c,j}\|^2+\notag\\
% &\quad \eta^2_y 
% \mE\|\bg^y_{c,j}\|^2)
% + \frac{1}{T}\sum_{i=0}^{T-1}\sum_{j=0}^{i}\Big(\frac{1+\rho}{2}\Big)^{i-j}\frac{\eta^2_y\lambda^2_aC'_3}{(1-\rho)\underline{\lambda^2_b}} \notag \\
&\overset{(a)}{\le} 
\frac{2\mhE_0}{T(1-\rho)}
+ \frac{10\eta^2_y\lambda^2_a(p+\beta^2)}{(1-\rho)^2\underline{\lambda^2_b}KT}
\sum_{i=0}^{T-1}
\mE\|\mS_{z,i}\|^2+\frac{1}{4T}\sum_{i=1}^{T} 
\mE\|\mhE_{z,i}\|^2
+\frac{108\eta^2_y\lambda^2_aL^2_f}{(1-\rho)^2\underline{\lambda^2_b}T}
\sum_{i=0}^{T-1}(\eta^2_x\mE\|\bg^x_{c,i}\|^2+\eta^2_y\mE\|\bg^y_{c,i}\|^2)
\notag\\
&\quad + \frac{2\eta^2_y\lambda^2_aC'_3}{(1-\rho)^2\underline{\lambda^2_b}} \notag \\
&\overset{(b)}{\le} 
\frac{2\mhE_0}{T(1-\rho)}
+ \frac{10\eta^2_y\lambda^2_a(p+\beta^2)}{(1-\rho)^2\underline{\lambda^2_b}K} \Bigg(
\frac{2\mS_0}{T\bar{\beta}}
+ 
\frac{72KL^2_fv^2_1v^2_2}{bT\bar{\beta}}
\sum_{i=1}^{T}
\mE\|\mhE_{z,i}\|^2+\frac{36KL^2_fv^2_1v^2_2\mhE_0}{bT\bar{\beta}}+\frac{36KL^2_f}{bT\bar{\beta}}
\sum_{i=0}^{T-1}
(\eta^2_x\mE\|\bg^x_{c,i}\|^2
\notag \\
&\quad 
 + \eta^2_y\mE\|\bg^y_{c,i}\|^2)
+\frac{K(C'_1+C'_2)}{\bar{\beta}}\Bigg)+\frac{1}{4T}
\sum_{i=1}^{T}\mE\|\mhE_{z,i}\|^2 
 +\frac{108\eta^2_y\lambda^2_aL^2_f}{(1-\rho)^2\underline{\lambda^2_b}T}
\sum_{i=0}^{T-1}(\eta^2_x\mE\|\bg^x_{c,i}\|^2+\eta^2_y\mE\|\bg^y_{c,i}\|^2) 
+ \frac{2\eta^2_y\lambda^2_aC'_3}{(1-\rho)^2\underline{\lambda^2_b}},
\label{appendix:proof:lemma_exey_coordinate_step_recursion}
\end{align}
where $(a)$ follows from useful inequality \eqref{proof:useful:recursion_first},
$(b)$ follows from Lemma \ref{appendix:lemma:coorinate_variance}.
Choosing $\eta_x \le \eta_y$ and
\begin{subequations}
\begin{align}
&\frac{10\eta^2_y\lambda^2_a(p+\beta^2)}{(1-\rho)^2\underline{\lambda^2_b}K} \frac{36KL^2_f}{b\bar{\beta}} \le \frac{360\eta^2_y\lambda^2_aL^2_f}{(1-\rho)^2\underline{\lambda^2_b}} \Longrightarrow
p+\beta^2 \le b \bar{\beta}
\Longrightarrow
0 \le (b-1)p +\beta(b -\beta - bp)  \Longrightarrow
b \ge 1, \beta +bp \le b, \\ \notag\\
&\frac{10\eta^2_y\lambda^2_a(p+\beta^2)}{(1-\rho)^2\underline{\lambda^2_b}K} \frac{72KL^2_fv^2_1v^2_2}{b\bar{\beta}} \le \frac{1}{4} \Longrightarrow\eta_y \le \frac{(1-\rho)\underline{\lambda_b}}{54L_f\lambda_av_1v_2}\sqrt{\frac{b\bar{\beta}}{p+\beta^2}}  \label{proof:stepcondition1} ,
% &\frac{10\eta^2_y\lambda^2_a(p+\beta^2)}{(1-\rho)^2\underline{\lambda^2_b}K} \times \frac{36KL^2_f}{b\bar{\beta}} \le \frac{360\eta^2_y\lambda^2_aL^2_f}{(1-\rho)^2\underline{\lambda^2_b}}\Longrightarrow \notag 
% \\
% &
% b \ge 1, \beta +bp \le b, 
\end{align}
\end{subequations}
where the second condition can be simplified into 
$\eta_y  \le \frac{(1-\rho)\underline{\lambda_b}}{54L_f\lambda_av_1v_2}$ using $p+\beta^2 \le b \bar{\beta}$.
% \begin{align}
% \frac{10\eta^2_y\lambda^2_a(p+\beta^2)}{(1-\rho)^2\underline{\lambda^2_b}K} \times \frac{2pKC_0}{B\bar{\beta}} \le \frac{2\eta^2_y \lambda^2_a}{(1-\rho)^2\underline{\lambda^2_b}} \times \frac{10pC_0}{B} & \Longrightarrow
% p+\beta^2 \le \bar{\beta}  \notag 
% \\
% &\Longrightarrow 0 \le \beta(1-\beta-p) 
% \notag \\
% &\Longrightarrow p+\beta \le 1 ,\\
% \frac{10\eta^2_y\lambda^2_a(p+\beta^2)}{(1-\rho)^2\underline{\lambda^2_b}K} \times \frac{4K\beta^2C_0}{b\bar{\beta}}
% \le \frac{2\eta^2_y\lambda^2_a}{(1-\rho)^2\underline{\lambda^2_b}} \times \frac{20\beta^2C_0}{b} &\Longrightarrow
% p +\beta^2 \le \bar{\beta} \Longrightarrow p+\beta \le 1, \end{align}
% \begin{align}
% \frac{10\eta^2_y\lambda^2_a(p+\beta^2)}{(1-\rho)^2\underline{\lambda^2_b}K} \times 
% \frac{12K(d_1\delta^2_x+d_2\delta^2_y)L^2_f}{b\bar{\beta}} \le \frac{2\eta^2_y\lambda^2_a}{(1-\rho)^2\underline{\lambda^2_b}} \times 
% 60(d_1\delta^2_x+d_2\delta^2_y)L^2_f
% &\Longrightarrow b \ge 1, \beta +bp \le b.
% \end{align}
We can then conclude that 
\begin{align}
&\frac{1}{T}\sum_{i=1}^{T}
\mE\|\mhE_{z,i}\|^2  \le \frac{1}{2T}\sum_{i=1}^{T}
\mE\|\mhE_{z,i}\|^2 +
\frac{2\mhE_0}{T(1-\rho)}
+\frac{468\eta^2_y\lambda^2_aL^2_f}{(1-\rho)^2\underline{\lambda^2_b}T}\sum_{i=0}^{T-1}(\eta^2_x\mE\|\bg^x_{c,i}\|^2 +\eta^2_y\mE\|\bg^y_{c,i}\|^2) 
\notag\\
&\quad +\frac{10\eta^2_y\lambda^2_a(p+\beta^2)}{(1-\rho)^2\underline{\lambda^2_b}K}
  \Big(
\frac{2\mS_0}{T\bar{\beta}} +\frac{36KL^2_fv^2_1v^2_2\mhE_0}{bT\bar{\beta}} + \frac{K(C'_1+C'_2)}{\bar{\beta}}\Big) + \frac{2\eta^2_y\lambda^2_aC'_3}{(1-\rho)^2\underline{\lambda^2_b}}.
\end{align}
Moving $\frac{1}{2T}\sum_{i=1}^{T}
(\mE\|\mhE_{x,i}\|^2+\mE\|\mhE_{y,i}\|^2) $ to the left hand side, it follows that
\begin{align}
&\frac{1}{T}\sum_{i=1}^{T}
\mE\|\mhE_{z,i}\|^2\le 
\frac{4\mhE_0}{T(1-\rho)}
+\frac{936\eta^2_y\lambda^2_aL^2_f}{(1-\rho)^2\underline{\lambda^2_b}T}
\sum_{i=0}^{T-1}(\eta^2_x\mE\|\bg^x_{c,i}\|^2  +\eta^2_y\mE\|\bg^y_{c,i}\|^2) 
+\frac{20\eta^2_y\lambda^2_a(p+\beta^2)}{(1-\rho)^2\underline{\lambda^2_b}K}  \Big(
\frac{2\mS_0}{T\bar{\beta}} +\frac{36KL^2_fv^2_1v^2_2\mhE_0}{bT\bar{\beta}}\notag \\
&\quad 
+ \frac{K(C'_1+C'_2)}{\bar{\beta}}
\Big)+ \frac{4\eta^2_y\lambda^2_aC'_3}{(1-\rho)^2\underline{\lambda^2_b}}.
\end{align}
Adding $\frac{\mhE_0}{T(1-\rho)}$ to both sides of the above relation and using $\frac{\mhE_0}{T(1-\rho)} > \frac{\mhE_0}{T}(\rho <1)$, we can further conclude that 
\begin{align}
&\frac{1}{T}\sum_{i=0}^{T-1}
\mE\|\mhE_{z,i}\|^2 \notag\\
&\le\frac{1}{T}\sum_{i=1}^{T}
\mE\|\mhE_{z,i}\|^2 + \frac{\mhE_0}{T(1-\rho)}  \notag\\
&\le 
\frac{5\mhE_0}{T(1-\rho)}
+\frac{936\eta^2_y\lambda^2_aL^2_f}{(1-\rho)^2\underline{\lambda^2_b}T}
\sum_{i=0}^{T-1}(\eta^2_x\mE\|\bg^x_{c,i}\|^2 +\eta^2_y\mE\|\bg^y_{c,i}\|^2) 
+\frac{20\eta^2_y\lambda^2_a(p+\beta^2)}{(1-\rho)^2\underline{\lambda^2_b}K}  \Big(
\frac{2\mS_0}{T\bar{\beta}} +\frac{36KL^2_fv^2_1v^2_2\mhE_0}{bT\bar{\beta}} + \frac{K(C'_1+C'_2)}{\bar{\beta}}
\Big)\notag \\
&\quad 
+ \frac{4\eta^2_y\lambda^2_aC'_3}{(1-\rho)^2\underline{\lambda^2_b}}.
\end{align}
\end{proof}
\begin{Lemma}[\textbf{Descent relation}]
\label{appendix:lemma:costfunction}
    Under Assumptions \ref{main:assumption:costfunction}---\ref{main:assumption:Lipschitz},  
    we consider $\bu^j_{k,i} \equiv 0$ and then bound the  value function $P(\bx_{c,i+1})$
    as follows 
    \begin{align}
      & P(\bx_{c,i+1})\le   P(\bx_{c,i})
-\frac{\eta_x}{2}
\|\nabla P(\bx_{c,i})\|^2
-\frac{\eta_x}{2}
(1-L\eta_x)
\|\bg^x_{c,i}\|^2
+4\eta_x \kappa L_f \Delta^y_{c,i} + 2\eta_x L^2_fv^2_1 v^2_2 
\|\mhE_{z,i}\|^2 
+ 2\eta_x\|\bs^x_{c,i}\|^2 \notag\\
&+ 2\eta_x\delta^2_xL^2_fd_1.
    \end{align}
\end{Lemma}
\begin{proof}
 Since \(P(\cdot)\) is \(L\)-smooth, we have
\begin{align}
&P(\bx_{c,i+1}) \notag \\
&\le 
P(\bx_{c,i})
- \eta_x \langle 
\nabla P(\bx_{c,i}), \bg^x_{c,i}
\rangle + \frac{L\eta^2_x}{2}\|\bg^x_{c,i}\|^2 \notag \\
&\overset{(a)}{\le} 
P(\bx_{c,i})
-\frac{\eta_x}{2}
\|\nabla P(\bx_{c,i})\|^2
-\frac{\eta_x}{2} (1-L\eta_x)
\|\bg^x_{c,i}\|^2  +
\eta_x
\Big\|\nabla P(\bx_{c,i})- \frac{1}{K}\sum_{k=1}^K \nabla^x_{k,i}\Big\|^2 +\eta_x \Big\|\bg^x_{c,i}- \frac{1}{K}\sum_{k=1}^{K}\nabla^x_{k,i}\Big\|^2
\notag\\
&\overset{(b)}{\le} 
P(\bx_{c,i})
-\frac{\eta_x}{2}
\|\nabla P(\bx_{c,i})\|^2
-\frac{\eta_x}{2}
(1-L\eta_x)
\|\bg^x_{c,i}\|^2
+
2\eta_x\Big\|\nabla P(\bx_{c,i}) - \frac{1}{K}
\sum_{k=1}^K\nabla_x J_k(\bz_{c,i})\Big\|^2 +2\eta_x\Big\| \frac{1}{K}\sum_{k=1}^{K}
\notag \\
&\quad 
 (\nabla_x J_k(\bz_{c,i}) - \nabla^x_{k,i})\Big\|^2 +2\eta_x\|\bs^x_{c,i}\|^2 + 2\eta_x\Big\| \frac{1}{K}\sum_{k=1}^{K}(\hnabla_x J_k(\bz_{k,i}) -\nabla^x_{k,i})\Big\|^2\notag \\
&\overset{(c)}{\le} 
P(\bx_{c,i})
-\frac{\eta_x}{2}
\|\nabla P(\bx_{c,i})\|^2
-\frac{\eta_x}{2}
(1-L\eta_x)
\|\bg^x_{c,i}\|^2
 +2\eta_xL^2_f\|\by_{c,i}-\by^o(\bx_{c,i})\|^2 + \frac{2\eta_xL^2_f}{K}\|\mZ_i-\mZ_{c,i}\|^2 
\notag \\
&\quad + 2\eta_x\|\bs^x_{c,i}\|^2+ 2\eta_x\delta^2_xL^2_fd_1 \notag \\
&\overset{(d)}{\le} P(\bx_{c,i})
-\frac{\eta_x}{2}
\|\nabla P(\bx_{c,i})\|^2
-\frac{\eta_x}{2}
(1-L\eta_x)
\|\bg^x_{c,i}\|^2
+2\eta_xL^2_f\|\by_{c,i} - \by^o(\bx_{c,i})\|^2+ 2\eta_xL^2_f v^2_1 v^2_2 
\|\mhE_{z,i}\|^2   \notag \\
&\quad + 2\eta_x\|\bs^x_{c,i}\|^2
 + 2\eta_x\delta^2_xL^2_fd_1\notag \\
&\overset{(e)}{\le} P(\bx_{c,i})
-\frac{\eta_x}{2}
\|\nabla P(\bx_{c,i})\|^2
-\frac{\eta_x}{2}
(1-L\eta_x)
\|\bg^x_{c,i}\|^2
 +4\eta_x \kappa L_f \Delta^y_{c,i}
+ 2 \eta_xL^2_f v^2_1 v^2_2 
\|\mhE_{z,i}\|^2  + 2\eta_x\|\bs^x_{c,i}\|^2 \notag \\&\quad+ 2\eta_x\delta^2_xL^2_fd_1,
\end{align}
where $(a)$ and $(b)$ follow from Jensen's inequality, 
$(c)$ follows from  Assumption \ref{main:assumption:Lipschitz} and Lemma \ref{appendix:lemma:coordinate_zero_first};$(d)$ follows from the relation \eqref{proof:useful:consensus}; $(e)$ follows from Lemma \ref{Danskin}.
\end{proof}
\begin{Lemma}[\textbf{Duality gap}]
\label{lemma:optimality_gap:coordinate}
Under Assumptions \ref{main:assumption:costfunction}---\ref{main:assumption:Lipschitz}, choosing step sizes
$\eta_x \le \min\{\frac{\eta_y}{16\kappa^2}, \frac{1}{32L}\}, \eta_y \le \min\{\frac{1}{\nu}, \frac{1}{2L_f}\}$, 
we can bound the optimality gap $\Delta^y_{c,i}$  as follows
\begin{align}
&\frac{1}{T}\sum_{i=0}^{T-1}\Delta^y_{c,i} \le 
\frac{3}{T\nu\eta_y}
\Delta^y_{c,0}
+
\frac{\eta_x}{4\nu\eta_yT} 
\sum_{i=0}^{T-1}
\|\bg^x_{c,i}\|^2-\frac{\eta_y}{4T}\sum_{i=1}^{T}
\sum_{j=0}^{i-1}
\Big(
1- \frac{\nu\eta_y}{2}\Big)^{i-j-1}
\|\bg^y_{c,j}\|^2+
\frac{4\kappa L_fv^2_1v^2_2}{T}
\sum_{i=0}^{T-1}
\|\mhE_{z,i}\|^2 \notag \\
&\quad 
+\frac{8}{T\nu}
\sum_{i=0}^{T-1}
\|\bs^y_{c,i}\|^2 +8\kappa L_fd_2\delta^2_y. 
\end{align}
\end{Lemma}
\begin{proof}
The proof can be completed by using an argument similar to that used in \cite[Lemma 5]{cai2025dama2}.
\end{proof}

\subsection{Proof of Theorem \ref{main:theorem1}}
\label{appendix:subsection:theorem1}
\begin{Theorem}[Restatement of Theorem \ref{main:theorem1}]
\label{appendix:restatement:theorem1}
Under Assumptions
\ref{main:assumption:costfunction}---\ref{main:assumptions:combinationmatrix}, we consider $\bu^j_{k,i} \equiv 0$ and choose hyperparameters satisfying
\begin{align}
 \eta_x &\le \min\Bigg\{\frac{1}{32L}, \frac{\eta_y}{16\kappa^2}, \frac{\sqrt{A_1}}{72\sqrt{3} \kappa L_f} \Bigg\}, p+\beta\le1,\
 b\ge1,  \notag \\
   b\bar{\beta} &\le \frac{129}{K},\ \beta+bp \le b, \ \delta_x \le \frac{\sigma}{L_f\sqrt{d_1}}, 
 \delta_y \le \frac{\sigma}{L_f\sqrt{d_2}},
 \\
 \eta_y &\le \min 
\Bigg\{ \frac{1}{\nu}, \frac{1}{2L_f}, \frac{\sqrt{A_1}}{72\sqrt{2}L_f} ,\frac{(1-\rho)\underline{\lambda_b}}{54A_3},\frac{(1-\rho)^{\frac{1}{2}}\underline{\lambda^{\frac{1}{2}}_b}(A_1)^{\frac{1}{4}}}{74L_f (\kappa A_2)^{\frac{1}{2}}},\frac{ (1-\rho)^{\frac{2}{3}}
\underline{\lambda^{\frac{2}{3}}_b}
(A_1)^{\frac{1}{3}}}{ 203L_f \kappa^{\frac{1}{3}} (A_2)^\frac{2}{3}}\Bigg\}, \bar{\beta} \le \frac{\nu\eta_y}{2},
\end{align}
where
\begin{align}
A_1 \triangleq bK\bar{\beta}, A_2 \triangleq v_1v_2\lambda_a, A_3 \triangleq L_fv_1v_2\lambda_a.
\end{align}
The performance bound is given by 
\begin{align}
&\frac{1}{T}\sum_{i=0}^{T-1}\Big(\mE\|\nabla_x J(\bx_{c,i}, \by_{c,i})\|^2 + \mE\|\nabla_y J(\bx_{c,i}, \by_{c,i})\|^2\Big)\notag \\
&\le \mathcal{O}
\Big(
\underbrace{\Pi'_0}_{\text{initial gap}} +\underbrace{\frac{\kappa^2\lambda^2_a\eta^2_y(p+\beta^2)\sigma^2}{bb_0K\bar{\beta}^2T(1-\rho)^2\underline{\lambda^2_b}}}_{\text{network noise}} + \underbrace{\Pi'_1 \sigma^2 \mathbb{I}(B<N) }_{\text{large-batch effect}}+
 \underbrace{ \frac{\kappa^2\sigma^2}{b_0\bar{\beta}KT}}_{\text{initial noise}}+  \underbrace{\Pi'_2\sigma^2}_{\text{ZO-CW noise}}
+\underbrace{\Pi'_3}_{\text{ZO-CW bias}} \Big),
\end{align}
where 
\begin{subequations}
\begin{align}
\Pi'_0 &\triangleq \frac{\mE \Delta'_{c,0}}{\eta_xT} 
+  \frac{\kappa^2\mE \Delta^y_{c,0}}{\eta_y T} + \frac{\kappa^2\eta^2_y\zeta^2_0}{bK\bar{\beta}T(1-\rho)\underline{\lambda^2_b}}, \\
\Pi'_1 &\triangleq \frac{\kappa^2\eta^2_y\lambda^2_a p}{bK\bar{\beta}B(1-\rho)^2\underline{\lambda^2_b}}    +\frac{\kappa^2p}{BK\bar{\beta}}
, \quad \Pi'_2 \triangleq \frac{\kappa^2\beta^2}{bK\bar{\beta}},\\
\Pi'_3 &\triangleq
\frac{\kappa^2\eta^2_y\lambda^2_a(d_1\delta^2_x+d_2\delta^2_y)}{bK\bar{\beta}(1-\rho)^2\underline{\lambda^2_b}}
+\frac{\kappa^2(d_1\delta^2_x+d_2\delta^2_y)}{bK\bar{\beta}} 
+ d_1\delta^2_x+d_2\kappa^2\delta^2_y.
\end{align}
\end{subequations}
\end{Theorem}
\begin{proof}
Invoking Lemma \ref{appendix:lemma:costfunction},
we can derive that 
\begin{align}
&\frac{1}{T}\sum_{i=0}^{T-1}\Big(\mE\|\nabla P(\bx_{c,i})\|^2 + \kappa L_f\mE\Delta^y_{c,i}\Big)\notag \\
&\le  \sum_{i=0}^{T-1}\frac{2\mE(P(\bx_{c,i}) - P(\bx_{c,i+1}))}{\eta_xT}
-\frac{1-L\eta_x}{T}
\sum_{i=0}^{T-1}\mE\|\bg^x_{c,i}\|^2
 +\frac{9\kappa L_f}{T}
\sum_{i=0}^{T-1}
\mE \Delta^y_{c,i}+\frac{4L^2_fv^2_1v^2_2}{T}
\sum_{i=0}^{T-1}\mE\|\mhE_{z,i}\|^2 \notag\\
&\quad + \frac{4}{T}
\sum_{i=0}^{T-1}\mE\|\bs^x_{c,i}\|^2 +4\delta^2_xL^2_fd_1.
\end{align}
Invoking Lemma 
\ref{lemma:optimality_gap:coordinate},
we get 
\begin{align}
&\frac{1}{T}\sum_{i=0}^{T-1}\Big(\mE\|\nabla P(\bx_{c,i})\|^2 + \kappa L_f\mE\Delta^y_{c,i}\Big)\notag \\
% &\le  \sum_{i=0}^{T-1}\frac{2\mE(P(\bx_{c,i}) - P(\bx_{c,i+1}))}{\eta_xT}
% -\frac{1-L\eta_x}{T}
% \sum_{i=0}^{T-1}\mE\|\bg^x_{c,i}\|^2
% +
% \Big(
% \frac{27\kappa^2\mE\Delta_{c,0}}{\eta_y T}
% +
% \frac{9\eta_x\kappa^2}{4\eta_yT} 
% \sum_{i=0}^{T-1}
% \mE\|\bg^x_{c,i}\|^2
% \notag\\
% &\quad 
% -\frac{9\eta_y\kappa L_f}{4T}\sum_{i=1}^{T}
% \sum_{j=0}^{i-1}\Big(
% 1- \frac{\nu\eta_y}{2}\Big)^{i-j-1}
% \mE\|\bg^y_{c,j}\|^2 +\frac{36\kappa^2L^2_fv^2_1v^2_2}{T} \sum_{i=0}^{T-1} (\mE\|\mhE_{x,i}\|^2
% +\mE\|\mhE_{y,i}\|^2)
% +\frac{36\kappa^2}{T}
% \sum_{i=0}^{T-1}\mE\|\bs^y_{c,i}\|^2\Big)
% \notag \\
% &\quad +\frac{4L^2_fv^2_1v^2_2}{T}
% \sum_{i=0}^{T-1}(\mE\|\mhE_{x,i}\|^2
% +\mE\|\mhE_{y,i}\|^2) + \frac{4}{T}
% \sum_{i=0}^{T-1}\mE\|\bs^x_{c,i}\|^2  +4\delta^2_xL^2_fd_1\\
&\overset{(a)}{\le} \frac{2\mE\Delta'_{c,0}}{\eta_xT} 
-\frac{1-L\eta_x}{T}
\sum_{i=0}^{T-1}\mE\|\bg^x_{c,i}\|^2
+
\Big(
\frac{27\kappa^2\mE\Delta^y_{c,0}}{\eta_y T}
+
\frac{9\eta_x\kappa^2}{4\eta_yT}  
\sum_{i=0}^{T-1}
\mE\|\bg^x_{c,i}\|^2-\frac{9\eta_y\kappa L_f}{4T}\sum_{i=1}^{T}
\sum_{j=0}^{i-1}\Big(
1- \frac{\nu\eta_y}{2}\Big)^{i-j-1}
\mE\|\bg^y_{c,j}\|^2 \Big)\notag\\
&\quad+\frac{40\kappa^2 L^2_fv^2_1v^2_2}{T}
\sum_{i=0}^{T-1}\mE\|\mhE_{z,i}\|^2+ \frac{72\kappa^2}{T}
\sum_{i=0}^{T-1}\mE\|\bs^z_{c,i}\|^2 
+\underbrace{4d_1\delta^2_xL^2_f+72d_2\kappa^2L^2_f\delta^2_y}_{\triangleq C'_4}
\notag \\
&\overset{(b)}{\le} 
\frac{2\mE\Delta'_{c,0}}{\eta_xT} 
-\frac{2}{3T}
\sum_{i=0}^{T-1}\mE\|\bg^x_{c,i}\|^2
+
\frac{27\kappa^2\mE\Delta^y_{c,0}}{\eta_y T}-\frac{9\eta_y\kappa L_f}{4T}
 \sum_{i=1}^{T}
\sum_{j=0}^{i-1}\Big(
1- \frac{\nu\eta_y}{2}\Big)^{i-j-1}
\mE\|\bg^y_{c,j}\|^2 
+\frac{40\kappa^2 L^2_fv^2_1v^2_2}{T}\sum_{i=0}^{T-1}\mE\|\mhE_{z,i}\|^2
\notag\\
&\quad  + \frac{72\kappa^2}{T}
\sum_{i=0}^{T-1}\mE\|\bs^z_{c,i}\|^2 +C'_4,
\end{align}
where $(a)$ follows by introducing $\mE\Delta'_{c,0} \triangleq \mE(P(\bx_{c,0}) - P^\star)$ and  using the relations $4<72\kappa^2$ for $\mE\|\bs^x_{c,i}\|^2$  and $4L^2_fv^2_1v^2_2<4\kappa^2L^2_fv^2_1v^2_2$ for $\mE\|\mhE_{z,i}\|^2$, $(b)$ 
follows by choosing 
\begin{align}
L\eta_x \le \frac{1}{6} &\Longrightarrow \eta_x \le \frac{1}{6L},\quad 
\frac{9\eta_x \kappa^2}{4\eta_y} \le \frac{1}{6} \Longrightarrow  \frac{\eta_x}{\eta_y} \le \frac{1}{14\kappa^2}.
\end{align}
Invoking Lemma \ref{appendix:lemma:CW_Nxc_Nyc},
we get 
\begin{align}
&\frac{1}{T}\sum_{i=0}^{T-1}\Big(\mE\|\nabla P(\bx_{c,i})\|^2 + \kappa L_f\mE\Delta^y_{c,i}\Big)\notag \\
&\le  
\frac{2\mE\Delta'_{c,0}}{\eta_xT} 
-\frac{2}{3T}
\sum_{i=0}^{T-1}\mE\|\bg^x_{c,i}\|^2
+
\frac{27\kappa^2\mE\Delta^y_{c,0}}{\eta_y T}-\frac{9\eta_y\kappa L_f}{4T}
\sum_{i=1}^{T}
\sum_{j=0}^{i-1}\Big(
1- \frac{\nu\eta_y}{2}\Big)^{i-j-1}
\mE\|\bg^y_{c,j}\|^2 
+\frac{40\kappa^2 L^2_fv^2_1v^2_2}{T}
\sum_{i=0}^{T-1}\mE\|\mhE_{z,i}\|^2
\notag\\
&\quad 
+ 
\Bigg(\frac{144\kappa^2\bs_{c,0}}{T\bar{\beta}}
+ 
\frac{72^2 \kappa^2L^2_fv^2_1v^2_2}{bKT\bar{\beta}}
\sum_{i=1}^{T}
\mE\|\mhE_{z,i}\|^2
+\frac{72^2 \kappa^2 L^2_fv^2_1v^2_2 \mhE_0}{bKT\bar{\beta}}+
\frac{72^2\kappa^2L^2_f}{bK\bar{\beta}T}
\sum_{i=0}^{T-1}
\eta^2_x\mE\|\bg^x_{c,i}\|^2
\notag\\
&\quad+\frac{72^2\kappa^2L^2_f\eta^2_y}{bK\bar{\beta}T}\sum_{i=0}^{T-1}
(1-(1-\bar{\beta})^{T-i})\mE\|\bg^y_{c,i}\|^2
 + \frac{72\kappa^2(C'_1+C'_2)}{K\bar{\beta}}\Bigg) +C'_4 \notag \\
&\overset{(a)}{\le} \frac{2\mE\Delta'_{c,0}}{\eta_xT} 
-\frac{1}{3T}
\sum_{i=0}^{T-1}\mE\|\bg^x_{c,i}\|^2
+
\frac{27\kappa^2\mE\Delta^y_{c,0}}{\eta_y T}-\frac{9\eta_y\kappa L_f}{4T}\sum_{i=1}^{T}
\sum_{j=0}^{i-1}\Big(
1- \frac{\nu\eta_y}{2}\Big)^{i-j-1}
\mE\|\bg^y_{c,j}\|^2 + 
\frac{ 72^2 \kappa^2L^2_fv^2_1v^2_2}{bKT\bar{\beta}}
\notag\\
&\quad 
\times
\Big(\sum_{i=1}^{T}
\mE\|\mhE_{z,i}\|^2
+
\sum_{i=0}^{T-1}\mE\|\mhE_{z,i}\|^2
\Big)
+ 
\Bigg(\frac{144\kappa^2\bs_{c,0}}{T\bar{\beta}}
+\frac{72^2 \kappa^2 L^2_fv^2_1v^2_2 \mhE_0}{bKT\bar{\beta}}
+\frac{72^2\kappa^2L^2_f\eta^2_y}{bK\bar{\beta}T}\sum_{i=0}^{T-1}
(1-(1-\bar{\beta})^{T-i})\mE\|\bg^y_{c,i}\|^2
\notag \\
& 
\quad + \frac{72\kappa^2(C'_1+C'_2)}{K\bar{\beta}}\Bigg) + C'_4\notag \\
&\overset{(b)}{\le} 
\frac{2\mE\Delta'_{c,0}}{\eta_xT} 
-\frac{1}{3T}
\sum_{i=0}^{T-1}\mE\|\bg^x_{c,i}\|^2
+
\frac{27\kappa^2\mE\Delta^y_{c,0}}{\eta_y T} -\frac{5\eta_y\kappa L_f}{4T} \sum_{i=1}^{T}
\sum_{j=0}^{i-1}\Big(
1- \frac{\nu\eta_y}{2}\Big)^{i-j-1}
\mE\|\bg^y_{c,j}\|^2\notag\\
&\quad  + 
\frac{ 72^2 \kappa^2L^2_fv^2_1v^2_2}{bKT\bar{\beta}}
\Big(\sum_{i=1}^{T}
\mE\|\mhE_{z,i}\|^2+
\sum_{i=0}^{T-1}\mE\|\mhE_{z,i}\|^2
 \Big)+ 
\Bigg(\frac{144\kappa^2\bs_{c,0}}{T\bar{\beta}}
+\frac{72^2 \kappa^2 L^2_fv^2_1v^2_2 \mhE_0}{bKT\bar{\beta}} +\frac{72\kappa^2(C'_1+C'_2)}{K\bar{\beta}}\Bigg)  + C'_4,
\end{align}
where $(a)$ follows by choosing 
\begin{align}
\frac{72^2\kappa^2L^2_f \eta^2_x}{bK\bar{\beta}} \le \frac{1}{3} &\Longrightarrow 
\eta_x \le \frac{\sqrt{bK\bar{\beta}}}{72\sqrt{3}\kappa L_f}, \\
40\kappa^2L^2_fv^2_1v^2_2  \le \frac{72^2 \kappa^2L^2_fv^2_1v^2_2}
{bK\bar{\beta}} & \Longrightarrow 
b\bar{\beta} \le \frac{129}{K},
\end{align}
and $(b)$ follows by choosing 
\begin{align}
&\frac{72^2\kappa^2L^2_f\eta^2_y}{bK\bar{\beta}}
\le 2\kappa^2 \Longrightarrow \eta_y \le \frac{\sqrt{bK\bar{\beta}}}{72\sqrt{2}L_f}, \\
&(1-(1-\bar{\beta})^{T-i}) \le \Big(
1-\Big(
1-\frac{\nu\eta_y}{2}
\Big)^{T-i}
\Big) \Longrightarrow 
\bar{\beta} \le \frac{\nu\eta_y}{2}, \bar{\beta} \le 1, \eta_y \le \frac{2}{\nu},
\end{align}
such that
\begin{align}
&\frac{\eta_y\kappa L_f}{T}
\sum_{i=1}^{T}\sum_{j=0}^{i-1}\Big(1 - \frac{\nu \eta_y}{2}\Big)^{i-j-1}
\mE\|\bg^y_{c,j}\|^2 \notag \\
&=
\frac{2\kappa^2}{T}\sum_{i=0}^{T-1} \Big(1- \Big(1-\frac{\nu\eta_y}{2}\Big)^{T-i}\Big)
\mE\|\bg^y_{c,i}\|^2 \notag \\
&\ge 
\frac{72^2\kappa^2L^2_f\eta^2_y}{bK\bar{\beta}T}
\sum_{i=0}^{T-1}(1-(1-\bar{\beta})^{T-i})\mE\|\bg^y_{c,i}\|^2.
\label{proof:appendix:eliminate_gyc}
\end{align}
Invoking Lemma \ref{appendix:lemma:cw_coupled}, we have 
\begin{align}
&\frac{1}{T}\sum_{i=0}^{T-1}\Big(\mE\Delta'_{c,0}\Big)\notag \\
&\le 
\frac{2\mE\Delta'_{c,0}}{\eta_xT} 
-\frac{1}{3T}
\sum_{i=0}^{T-1}\mE\|\bg^x_{c,i}\|^2
+
\frac{27\kappa^2\mE\Delta^y_{c,0}}{\eta_y T}-\frac{5\eta_y\kappa L_f}{4T}\sum_{i=1}^{T}
\sum_{j=0}^{i-1}\Big(
1- \frac{\nu\eta_y}{2}\Big)^{i-j-1}
\mE\|\bg^y_{c,j}\|^2 
\notag\\
&\quad 
+ 
\frac{ 2\times72^2 \kappa^2L^2_fv^2_1v^2_2}{bK\bar{\beta}}
\Bigg( \frac{5\mhE_0}{T(1-\rho)}
+L'_0
\sum_{i=0}^{T-1}(\eta^2_x\mE\|\bg^x_{c,i}\|^2  
+\eta^2_y\mE\|\bg^y_{c,i}\|^2)
+I'_0+\frac{20\eta^2_y\lambda^2_a(p+\beta^2)(C'_1+C'_2)}{(1-\rho)^2\underline{\lambda^2_b}\bar{\beta}} 
\notag \\
&\quad + \frac{4\eta^2_y\lambda^2_aC'_3}{(1-\rho)^2\underline{\lambda^2_b}}\Bigg)
+ 
\Bigg(\frac{144\kappa^2\bs_{c,0}}{T\bar{\beta}}
+\frac{72^2 \kappa^2 L^2_fv^2_1v^2_2 \mhE_0}{bKT\bar{\beta}} +\frac{72\kappa^2(C'_1+C'_2)}{K\bar{\beta}}\Bigg) + C'_4.
\end{align}
We then choose 
\begin{align}
&\frac{2\times 72^2\kappa^2L^2_fv^2_1v^2_2}{bK\bar{\beta}}
\times \underbrace{\frac{936 \eta^2_y\lambda^2_a L^2_f}{(1-\rho)^2\underline{\lambda^2_b}T}}_{L'_0} \times\eta^2_x\le \frac{1}{3T} \Longrightarrow
\eta_y \le  \frac{(1-\rho)^{\frac{1}{2}}\underline{\lambda^{\frac{1}{2}}_b}(bK\bar{\beta})^{\frac{1}{4}}}{74L_f (\kappa \lambda_a v_1v_2)^{\frac{1}{2}}} ,\\
&\frac{2\times 72^2\kappa^2L^2_fv^2_1v^2_2}{bK\bar{\beta}}
\times \underbrace{\frac{936 \eta^2_y\lambda^2_a L^2_f}{(1-\rho)^2\underline{\lambda^2_b}T}}_{L'_0} \times\eta^2_y \le \frac{\kappa^2 \nu \eta_y}{T} \Longrightarrow \eta_y \le 
\frac{ (1-\rho)^{\frac{2}{3}}
\underline{\lambda^{\frac{2}{3}}_b}
(bK\bar{\beta})^{\frac{1}{3}}}{ 203\kappa^{\frac{1}{3}} L_f (\lambda_a v_1v_2)^\frac{2}{3}},
\end{align}

such that 
\begin{align}
&\frac{\eta_y\kappa L_f}{T}
\sum_{i=1}^{T}\sum_{j=0}^{i-1}\Big(1 - \frac{\nu \eta_y}{2}\Big)^{i-j-1}
\mE\|\bg^y_{c,j}\|^2 \notag \\
&=
\frac{2\kappa^2}{T}\sum_{i=0}^{T-1} \Big(1- \Big(1-\frac{\nu\eta_y}{2}\Big)^{T-i}\Big)
\mE\|\bg^y_{c,i}\|^2 \notag \\
&\ge \frac{2\kappa^2}{T}\sum_{i=0}^{T-1} \Big(1- \Big(1-\frac{\nu\eta_y}{2}\Big)\Big)
\mE\|\bg^y_{c,i}\|^2 \notag \\
& = \frac{\kappa^2 \nu \eta_y}{T}\sum_{i=0}^{T-1}\mE\|\bg^y_{c,i}\|^2 \notag\\
&\ge \frac{2\times 72^2\kappa^2L^2_fv^2_1v^2_2}{bK\bar{\beta}}
\times \frac{936 \eta^4_y\lambda^2_a L^2_f}{(1-\rho)^2\underline{\lambda^2_b}T}\sum_{i=0}^{T-1}\mE\|\bg^y_{c,i}\|^2 .
\label{proof:appendix:eliminate_gyc_cw}
\end{align}
Finally, we can conclude that 
\begin{align}
&\frac{1}{T}\sum_{i=0}^{T-1}\Big(\mE\|\nabla P(\bx_{c,i})\|^2 + \kappa L_f\mE\Delta_{c,i}\Big)\notag \\
&\le 
\frac{2\mE\Delta'_{c,0}}{\eta_xT} 
+
\frac{27\kappa^2\mE\Delta^y_{c,0}}{\eta_y T}+ 
\frac{ 2\times 72^2 \kappa^2L^2_fv^2_1v^2_2}{bK\bar{\beta}}
\Bigg( \frac{5\mhE_0}{T(1-\rho)} 
+I'_0+\frac{20\eta^2_y\lambda^2_a(p+\beta^2)(C'_1+C'_2)}{(1-\rho)^2\underline{\lambda^2_b}\bar{\beta}} 
+ \frac{4\eta^2_y\lambda^2_aC'_3}{(1-\rho)^2\underline{\lambda^2_b}}\Bigg)+ 
\notag \\
&\quad \Bigg(
\frac{144\kappa^2\bs_{c,0}}{T\bar{\beta}}
+\frac{72^2 \kappa^2 L^2_fv^2_1v^2_2 \mhE_0}{bKT\bar{\beta}} +\frac{72\kappa^2(C'_1+C'_2)}{K\bar{\beta}}\Bigg) + C'_4 \notag\\
% \notag\\
% &\le 
% \frac{2\mE G_{p,0}}{\eta_xT} 
% +
% \frac{27\kappa^2\mE\Delta_{c,0}}{\eta_y T} 
% + 
% \frac{2\times 36^2 \kappa^2L^2_fv^2_1v^2_2}{bK\bar{\beta}}
% \Bigg( \frac{5\mhE_0}{T(1-\rho)}+\frac{20\eta^2_y\lambda^2_a(p+\beta^2)}{(1-\rho)^2\underline{\lambda^2_b}K}  \Big(
% \frac{2\mS_0}{T\bar{\beta}}+\frac{36KL^2_fv^2_1v^2_2\mhE_0}{bT\bar{\beta}}\Big)
% \notag \\
% &\quad +\frac{4\eta^2_y\lambda^2_a}{(1-\rho)^2\underline{\lambda^2_b}}\Big(
% \frac{20pKc^2_0}{B}\mathbb{I}(B<N) +78(d_1\delta^2_x
% +d_2\delta^2_y)L^2_f
% +\frac{26\beta^2c^2_0}{b}
% \Big)\Bigg)+ 
% \Bigg(\frac{72\kappa^2\bs_{c,0}}{T\bar{\beta}}+\frac{36^2 \kappa^2 L^2_fv^2_1v^2_2 \mhE_0}{bKT\bar{\beta}}\notag \\
% &\quad 
% +\frac{72\kappa^2pc^2_0}{BK\bar{\beta}}\mathbb{I}(B<N)
% + \frac{144\kappa^2\beta^2c^2_0}{bK\bar{\beta}} \notag+\frac{432\kappa^2(d_1L^2_f\delta^2_x+d_2L^2_f\delta^2_y)}{bK\bar{\beta}} +4\delta^2_xL^2_fd_1\Bigg) \notag \\
&\overset{(a)}{\le} \mathcal{O}
\Bigg(
\frac{\mE \Delta'_{c,0}}{\eta_xT} 
+  \frac{\kappa^2\mE \Delta^y_{c,0}}{\eta_y T} + \frac{\kappa^2\eta^2_y\zeta^2_0}{bK\bar{\beta}T(1-\rho)\underline{\lambda^2_b}} 
+ \frac{\kappa^2\lambda^2_a\eta^2_y(p+\beta^2)\sigma^2}{bb_0K\bar{\beta}^2T(1-\rho)^2\underline{\lambda^2_b}} 
+\frac{\kappa^2\lambda^2_a\eta^4_y(p+\beta^2)\zeta^2_0}{b^2K\bar{\beta}^2T(1-\rho)^2\underline{\lambda^4_b}} \notag\\
&\quad +\frac{\kappa^2\eta^2_y\lambda^2_a(p+\beta^2)}{bK\bar{\beta}^2(1-\rho)^2\underline{\lambda^2_b}} \Big( \frac{d_1\delta^2_x +d_2\delta^2_y}{b} +
\frac{p\sigma^2}{B} \mathbb{I}(B<N) + \frac{\beta^2\sigma^2}{b}\Big)+
\frac{\kappa^2\lambda^2_a \eta^2_y }{bK\bar{\beta}(1-\rho)^2\underline{\lambda^2_b}}\Big(
\frac{p\sigma^2}{B}\mathbb{I}(B<N)+ d_1\delta^2_x + d_2\delta^2_y+\frac{\beta^2\sigma^2}{b}\Big)
\notag\\
&\quad 
+ \frac{\kappa^2\sigma^2}{b_0\bar{\beta}KT}
+\frac{\kappa^2\zeta^2_0\eta^2_y}{bKT\bar{\beta}\underline{\lambda^2_b}}
+ \Big(\frac{\kappa^2 p \sigma^2}{BK\bar{\beta}}\mathbb{I}(B<N)
+\frac{\kappa^2\beta^2\sigma^2}{bK\bar{\beta}} + \frac{\kappa^2(d_1\delta^2_x+d_2\delta^2_y)}{bK\bar{\beta}}\Big)
 + d_1\delta^2_x+\kappa^2d_2\delta^2_y\Bigg), \notag\\
&\overset{(b)}{\le} \mathcal{O}
\Bigg(
\frac{\mE \Delta'_{c,0}}{\eta_xT} 
+  \frac{\kappa^2\mE \Delta^y_{c,0}}{\eta_y T} + \frac{\kappa^2\eta^2_y\zeta^2_0}{bK\bar{\beta}T(1-\rho)\underline{\lambda^2_b}} 
+ \frac{\kappa^2\lambda^2_a\eta^2_y(p+\beta^2)\sigma^2}{bb_0K\bar{\beta}^2T(1-\rho)^2\underline{\lambda^2_b}} 
+\frac{\kappa^2\lambda^2_a\eta^4_y(p+\beta^2)\zeta^2_0}{b^2K\bar{\beta}^2T(1-\rho)^2\underline{\lambda^4_b}} \notag\\
&\quad +
\frac{\kappa^2\lambda^2_a \eta^2_y }{bK\bar{\beta}(1-\rho)^2\underline{\lambda^2_b}}\Big(
\frac{p\sigma^2}{B}\mathbb{I}(B<N)+ d_1\delta^2_x + d_2\delta^2_y
+\frac{\beta^2\sigma^2}{b}\Big)+ \frac{\kappa^2\sigma^2}{b_0\bar{\beta}KT}
+\frac{\kappa^2\zeta^2_0\eta^2_y}{bKT\bar{\beta}\underline{\lambda^2_b}}\notag \\
&\quad 
+ \Big(\frac{\kappa^2 p \sigma^2}{BK\bar{\beta}}\mathbb{I}(B<N)
+\frac{\kappa^2\beta^2\sigma^2}{bK\bar{\beta}} + \frac{\kappa^2(d_1\delta^2_x+d_2\delta^2_y)}{bK\bar{\beta}}\Big)
 + d_1\delta^2_x+\kappa^2d_2\delta^2_y\Bigg)
\end{align}
where in $(a)$ we  suppressed the constants $v^2_1, v^2_2, L^2_f$, used the expression of $I'_0$, and
\begin{align}
C_0 &= \mathcal{O}(\sigma^2), \quad C'_1 = \mathcal{O}\Big( \frac{d_1\delta^2_x +d_2\delta^2_y}{b}\Big), \notag\\
C'_2 &= \mathcal{O}\Big(\frac{pC_0}{B} \mathbb{I}(B<N) + \frac{\beta^2C_0}{b} \Big), \notag\\
C'_3 & = \mathcal{O}\Big(\frac{pC_0}{B}\mathbb{I}(B<N)
+d_1\delta^2_x+d_2\delta^2_y+\frac{\beta^2C_0}{b}\Big), \notag\\
C'_4  &\triangleq \mathcal{O}(d_1\delta^2_x+\kappa^2d_2\delta^2_y),
\end{align}
$(b)$ is due to
\begin{align}
\frac{p+\beta^2}{\bar{\beta} }\le 1\Longrightarrow \frac{\kappa^2\eta^2_y\lambda^2_a(p+\beta^2)}{bK\bar{\beta}^2(1-\rho)^2\underline{\lambda^2_b}}  \le \frac{\kappa^2\eta^2_y\lambda^2_a}{bK\bar{\beta}(1-\rho)^2\underline{\lambda^2_b}}.
\end{align}
We will choose sufficiently small $\eta_y$, which gives 
\begin{subequations}
\begin{align}
&\frac{\kappa^2\lambda^2_a\eta^2_y(p+\beta^2)\sigma^2}{bb_0K\bar{\beta}^2T(1-\rho)^2\underline{\lambda^2_b}} 
+\frac{\kappa^2\lambda^2_a\eta^2_y(p+\beta^2)\zeta^2_0}{b^2K\bar{\beta}^2T(1-\rho)^2\underline{\lambda^2_b}} \times\underbrace{\frac{\eta^2_y}{\underline{\lambda^2_b}}}_{\text{small}}\le \mathcal{O}
\Big(\frac{\kappa^2\lambda^2_a\eta^2_y(p+\beta^2)\sigma^2}{bb_0K\bar{\beta}^2T(1-\rho)^2\underline{\lambda^2_b}} \Big),  \\
% &\Big(\frac{\kappa^2\eta^2_y\lambda^2_a(p+\beta^2)}{bK\bar{\beta}^2(1-\rho)^2\underline{\lambda^2_b}}
% \times \frac{p\sigma^2}{B}
% +\frac{\kappa^2\eta^2_y\lambda^2_a}{bK\bar{\beta}(1-\rho)^2\underline{\lambda^2_b}} \times \frac{p\sigma^2}{B} \notag\\
% &\quad +\frac{\kappa^2p\sigma^2}{BK\bar{\beta}}\Big)\mathbb{I}(B<N) \le \mathcal{O}\Big(\frac{\kappa^2\eta^2_y\lambda^2_a p\sigma^2}{bK\bar{\beta}B(1-\rho)^2\underline{\lambda^2_b}}   
% \notag\\
% &\quad +\frac{\kappa^2p\sigma^2}{BK\bar{\beta}}
% \Big)\mathbb{I}(B<N),\\
% &\frac{\kappa^2\eta^2_y\lambda^2_a(p+\beta^2)}{bK\bar{\beta}^2(1-\rho)^2\underline{\lambda^2_b}}\times \frac{d_1\delta^2_x+d_2\delta^2_y}{b} +\frac{\kappa^2\eta^2_y\lambda^2_a}{bK\bar{\beta}(1-\rho)^2\underline{\lambda^2_b}} \times d_1\delta^2_x \notag\\
% &\le \mathcal{O}\Big(
% \frac{\kappa^2\eta^2_y\lambda^2_a(d_1\delta^2_x+d_2\delta^2_y)}{bK\bar{\beta}(1-\rho)^2\underline{\lambda^2_b}}
% \Big), \\
&\frac{\kappa^2\zeta^2_0\eta^2_y}{bKT\bar{\beta}\underline{\lambda^2_b}} \le \mathcal{O}\Big(\frac{\kappa^2\zeta^2_0\eta^2_y}{bKT\bar{\beta}(1-\rho)\underline{\lambda^2_b}}\Big), \notag\\
& \frac{\kappa^2\lambda^2_a\eta^2_y}{bK\bar{\beta}(1-\rho)^2\underline{\lambda^2_b}} \times \frac{\beta^2\sigma^2}{b} =\mathcal{O}\Big(
\frac{\kappa^2\beta^2\sigma^2}{bK\bar{\beta}}\Big)
\end{align}
\end{subequations}
In fact, when the hyperparameters are specified, these omitted terms can also be absorbed into other existing terms in the bound.
The performance bound is given by 
\begin{align}
&\frac{1}{T}\sum_{i=0}^{T-1}\Big(\mE\|\nabla P(\bx_{c,i})\|^2 + \kappa L_f\mE\Delta^y_{c,i}\Big)\notag \\
&\le \mathcal{O}
\Big(
\underbrace{\Pi'_0}_{\text{initial gap}} +\underbrace{\frac{\kappa^2\lambda^2_a\eta^2_y(p+\beta^2)\sigma^2}{bb_0K\bar{\beta}^2T(1-\rho)^2\underline{\lambda^2_b}}}_{\text{network noise}} + \underbrace{\Pi'_1 \sigma^2 \mathbb{I}(B<N) }_{\text{large-batch effect}}
 \underbrace{+ \frac{\kappa^2\sigma^2}{b_0\bar{\beta}KT}}_{\text{initial noise}}+  \underbrace{\Pi'_2\sigma^2}_{\text{ZO momentum error}}
+\underbrace{\Pi'_3}_{\text{ZO-CW bias}} \Big), 
\end{align}
where 
\begin{subequations}
\begin{align}
\Pi'_0 &\triangleq \frac{\mE \Delta'_{c,0}}{\eta_xT} 
+  \frac{\kappa^2\mE \Delta^y_{c,0}}{\eta_y T} + \frac{\kappa^2\eta^2_y\zeta^2_0}{bK\bar{\beta}T(1-\rho)\underline{\lambda^2_b}}, \\
\Pi'_1 &\triangleq \frac{\kappa^2\eta^2_y\lambda^2_a p}{bK\bar{\beta}B(1-\rho)^2\underline{\lambda^2_b}}    +\frac{\kappa^2p}{BK\bar{\beta}}
,\Pi'_2 \triangleq \frac{\kappa^2\beta^2}{bK\bar{\beta}}\\
\Pi'_3 &\triangleq
\frac{\kappa^2\eta^2_y\lambda^2_a(d_1\delta^2_x+d_2\delta^2_y)}{bK\bar{\beta}(1-\rho)^2\underline{\lambda^2_b}}
+\frac{\kappa^2(d_1\delta^2_x+d_2\delta^2_y)}{bK\bar{\beta}} 
 + d_1\delta^2_x+d_2\kappa^2\delta^2_y.
\end{align}
\end{subequations}
The proof is completed by using \eqref{appendix:useful:bound_gradient}.
\end{proof}
\subsection{Corollaries of Theorem \ref{main:theorem1}}
\label{appendix:corollary:case1}
We now specialize Theorem 1 by considering two settings: the online stochastic setting, where $N$ may be infinitely large, and the finite-sum setting, where $N<+\infty$.

$\bullet$ \textbf{Online stochastic scenario}

\begin{Corollary}[\textbf{ZO-STORM-ED}]
\label{corollary:online:storm+ed}
Under Assumptions \ref{main:assumption:costfunction}-\ref{main:assumptions:combinationmatrix} and the matrix condition in Lemma \ref{appendix:lemma:transformed_recursion}, 
we consider $\bu^j_{k,i} \equiv 0$ and the  ED strategy shown in Table \ref{tab:matrix_choices}. We set the hyperparameters as follows 
\begin{align}
\eta_x &= \mathcal{O}\Bigg(
\frac{K^{\frac{2}{3}}}{\kappa^2 T^{\frac{1}{3}}}
\Bigg), \eta_y = \mathcal{O}\Bigg(
\frac{K^{\frac{2}{3}}}{T^{\frac{1}{3}}}
\Bigg), p = 0, \beta =   \mathcal{O}\Bigg(
\frac{K^{\frac{1}{3}}}{T^{\frac{2}{3}}}
\Bigg), \notag\\
b&= \mathcal{O}(1),
\delta_x = \mathcal{O}\Bigg(
\frac{K^{\frac{1}{3}}}{d^{\frac{1}{2}}_1T^{\frac{2}{3}}}
\Bigg),\delta_y = \mathcal{O}\Bigg(
\frac{K^{\frac{1}{3}}}{d^{\frac{1}{2}}_2T^{\frac{2}{3}}} 
\Bigg), 
b_0 = \mathcal{O}\Bigg(
\frac{T^{\frac{1}{3}}}{K^{\frac{2}{3}}}
\Bigg).
\end{align}
They satisfy the following condition for sufficiently large $T$
\begin{align}
\eta_x &\le \min \Bigg\{ 
\frac{1}{32L}, \mathcal{O}\Bigg(
\frac{K^{\frac{2}{3}}}{\kappa^2T^{\frac{1}{3}}}\Bigg), \mathcal{O}\Bigg(
\frac{K^{\frac{2}{3}}}{\kappa T^{\frac{1}{3}}}
\Bigg)
\Bigg\},  \beta \le \mathcal{O}\Bigg( \frac{1}{K}\Bigg) ,
\delta_x \le \frac{\sigma^2}{L_f\sqrt{d_1}},
\quad \delta_y \le \frac{\sigma^2}{L_f\sqrt{d_2}}, \quad  \beta \le \frac{\nu\eta_y}{2} \notag \\
\eta_y 
&\le \min\Bigg\{
\frac{1}{\nu}, \frac{1}{2L_f},\mathcal{O}\Bigg(
\frac{K^{\frac{2}{3}}}{T^{\frac{1}{3}}}\Bigg),
\frac{(1-\sqrt{\lambda})\sqrt{1-\lambda}}{54A_3},
\frac{(1-\sqrt{\lambda})^{\frac{1}{2}}(1-\lambda)^{\frac{1}{4}}K^{\frac{1}{3}}}{74L_f (\kappa A_2)^{\frac{1}{2}} T^{\frac{1}{6}}},
 \frac{ (1-\sqrt{\lambda})^{\frac{2}{3}}
(1-\lambda)^{\frac{1}{3}}
K^\frac{4}{9}}{ 203L_f\kappa^{\frac{1}{3}}  (A_2)^\frac{2}{3}T^{\frac{2}{9}}} 
\Bigg\}. \notag 
\end{align}
The performance bound is given by 
\begin{align}
&\frac{1}{T}\sum_{i=0}^{T-1}\Big(\mE\|\nabla_x J(\bx_{c,i}, \by_{c,i})\|^2 + \mE\|\nabla_y J(\bx_{c,i}, \by_{c,i})\|^2\Big) \notag \\
&\le \mathcal{O}
\Bigg(
\frac{\kappa^2\Delta'_0}{(TK)^{\frac{2}{3}}}
+\frac{\kappa^2\zeta^2_0}{T(1-\lambda)^2}+ \frac{\kappa^2K\lambda^2\sigma^2}{T^2(1-\lambda)^3}+ \frac{\kappa^2\sigma^2}{(TK)^{\frac{2}{3}}}  
+ \frac{\kappa^2\lambda^2K^{\frac{2}{3}}}{(1-\lambda)^3T^{\frac{4}{3}}} 
\Bigg), 
\end{align}
where $\Delta'_0 \triangleq \mE(\Delta'_{c,0} + \Delta^y_{c,0})$.
The dominant communication complexity (CC)  and function query complexity (FC) are given by 
\begin{align}
&CC = \mathcal{O}\Bigg(\frac{\kappa^3\varepsilon^{-3}}{K} + \frac{\kappa^2\varepsilon^{-2}}{(1-\lambda)^2} \Bigg), \quad  FC = 
CC \times 4d  + b_0 \times 2d \approx \mathcal{O}\Bigg(\frac{d\kappa^3\varepsilon^{-3}}{K} + \frac{d\kappa^2\varepsilon^{-2}}{(1-\lambda)^2} + \frac{d \kappa \varepsilon^{-1}}{K}  \Bigg).
\end{align}
Furthermore, the transient time in achieving linear speedup in the number of agent $K$ is given by 
\begin{align}
\max \Bigg\{ 
\mathcal{O} \Bigg(\frac{K^2}{(1-\lambda)^6} \Bigg), \mathcal{O} \Bigg( 
\frac{K^{2}\lambda^{3}}{(1-\lambda)^{\frac{9}{2}}}
\Bigg),
\mathcal{O} \Bigg( 
\frac{K^{\frac{5}{4}}\lambda^{\frac{3}{2}}}{(1-\lambda)^{\frac{9}{4}}}
\Bigg)
\Bigg\}.
\end{align}
\end{Corollary}
    We note that ZO-STORM-EXTRA achieves performance comparable to ZO-STORM-ED.
\begin{Corollary}[\textbf{ZO-STORM-GT}]
\label{corollary:online:storm+gt}
Under Assumptions \ref{main:assumption:costfunction}-\ref{main:assumptions:combinationmatrix} and the matrix condition in Lemma \ref{appendix:lemma:transformed_recursion},
we consider $\bu^j_{k,i} \equiv 0$ 
and the ATC-GT strategy shown in in Table \ref{tab:matrix_choices}, and then choose the hyperparameters similar to Corollary \ref{corollary:online:storm+ed}.
The performance bound is given by 
\begin{align}
&\frac{1}{T}\sum_{i=0}^{T-1}\Big(\mE\|\nabla_x J(\bx_{c,i}, \by_{c,i})\|^2 + \mE\|\nabla_y J(\bx_{c,i}, \by_{c,i})\|^2\Big) \notag \\
&\le \mathcal{O}
\Bigg(
\frac{\kappa^2\Delta'_0}{(TK)^{\frac{2}{3}}}
+\frac{\kappa^2\zeta^2_0}{T(1-\lambda)^3}
+ \frac{\kappa^2K\lambda^4\sigma^2}{T^2(1-\lambda)^4} + \frac{\kappa^2\sigma^2}{(TK)^{\frac{2}{3}}} 
 +\frac{\kappa^2\lambda^4K^{\frac{2}{3}}}{(1-\lambda)^4T^{\frac{4}{3}}}
\Bigg).
\end{align}
The CC and FC are given by 
\begin{align}
&CC = \mathcal{O}\Bigg(\frac{\kappa^3\varepsilon^{-3}}{K} + \frac{\kappa^2\varepsilon^{-2}}{(1-\lambda)^3} \Bigg), \quad  FC = 
CC \times 4d 
+ b_0 \times 2d \approx \mathcal{O}\Bigg(\frac{d\kappa^3\varepsilon^{-3}}{K} + \frac{d\kappa^2\varepsilon^{-2}}{(1-\lambda)^3} + \frac{d\kappa \varepsilon^{-1}}{K}\Bigg).
\end{align}
Furthermore, the transient time in achieving linear speedup in the number of agents $K$ is given by 
\begin{align}
\max \Bigg\{ 
\mathcal{O} \Bigg(\frac{K^2}{(1-\lambda)^9} \Bigg), \mathcal{O} \Bigg( 
\frac{K^{2}\lambda^{6}}{(1-\lambda)^{6}} \Bigg),
\mathcal{O} \Bigg( 
\frac{K^{\frac{5}{4}}\lambda^{3}}{(1-\lambda)^{3}} \Bigg)
\Bigg\}.
\end{align}
\end{Corollary}

$\bullet$ \textbf{Finite-sum scenario}

Below, we focus on full-batch setting, 
i.e., 
$B=N$.
\begin{Corollary}[\textbf{ZO-PAGE-ED}]
\label{corollary:offline:page+ed}
Under Assumptions \ref{main:assumption:costfunction}-\ref{main:assumptions:combinationmatrix} and the matrix condition in Lemma \ref{appendix:lemma:transformed_recursion},
we consider $\bu^j_{k,i} \equiv 0$ and  the ED strategy shown in Table \ref{tab:matrix_choices}.
We set the hyperparameters as follows
\begin{align}
 B &= N, b =b_0 = \sqrt{\frac{N}{K}}, p = \frac{1}{\sqrt{NK}}, \beta = 0,  \notag\\ \eta_x &= \mathcal{O}\Big(\frac{(1-\lambda)^{1.5}}{\kappa^2}\Big),
 \eta_y = \mathcal{O}((1-\lambda)^{1.5}),   \notag\\ \delta_x &= \mathcal{O}
 \Big(
 \frac{1}{d^{\frac{1}{2}}_1 (1-\lambda)^{\frac{3}{4}}T^{\frac{1}{2}}}
 \Big),
 \delta_y = \mathcal{O}
 \Big(
 \frac{1}{d^{\frac{1}{2}}_2 (1-\lambda)^{\frac{3}{4}}T^{\frac{1}{2}}}
 \Big).
\end{align}
The performance bound is then given by 
\begin{align}
&\frac{1}{T}\sum_{i=0}^{T-1}\Big(\mE\|\nabla_x J(\bx_{c,i}, \by_{c,i})\|^2 + \mE\|\nabla_y J(\bx_{c,i}, \by_{c,i})\|^2\Big) \notag \\
&\le \mathcal{O}
\Bigg(
\frac{\kappa^2(\Delta'_0+1)}{(1-\lambda)^{1.5}T}
+\frac{\kappa^2(1-\lambda)\zeta^2_0}{T}
+\frac{\kappa^2\lambda^2\sqrt{K}\sigma^2}{\sqrt{N} T}
+\frac{\kappa^2\sigma^2}{T}
\Bigg).
\end{align}
The dominant CC and FC are given by 
\begin{align}
CC &\approx \mathcal{O}\Bigg(\frac{\kappa^2\varepsilon^{-2}}{(1-\lambda)^{1.5}}
+ \kappa^2\varepsilon^{-2} \Bigg), 
\\
FC &\approx 
\Big(CC \times p \times B + b_0\Big) \times 2d + CC \times (1-p) \times b  \times 4d \approx \mathcal{O}\Bigg(
\frac{d\kappa^2\sqrt{N}\varepsilon^{-2}}{\sqrt{K}(1-\lambda)^{1.5}} + d\sqrt{\frac{N}{K}}
\Bigg).
\end{align}
\end{Corollary}
Wee note that ZO-PAGE-EXTRA achieves performance comparable to ZO-PAGE-ED.

\begin{Corollary}[\textbf{ZO-PAGE-GT}]
\label{corollary:offline:page+atc-gt}
Under Assumptions \ref{main:assumption:costfunction}-\ref{main:assumptions:combinationmatrix} and the matrix condition in Lemma \ref{appendix:lemma:transformed_recursion},
we consider $\bu^j_{k,i} \equiv 0$  and the ATC-GT strategy shown in Table \ref{tab:matrix_choices}.
We set the hyperparameters as follows
\begin{align}
 B &= N, b =b_0 = \sqrt{\frac{N}{K}}, p = \frac{1}{\sqrt{NK}}, \beta = 0, \notag\\
 \eta_x &= \mathcal{O}\Big(\frac{(1-\lambda)^{2}}{\kappa^2}\Big), \eta_y = \mathcal{O}\Big((1-\lambda)^{2}\Big), \notag\\
 \delta_x &= \mathcal{O}
 \Bigg(
 \frac{1}{d^{\frac{1}{2}}_1 (1-\lambda)T^{\frac{1}{2}}}
 \Bigg),  \delta_y = \mathcal{O}
 \Bigg(
 \frac{1}{d^{\frac{1}{2}}_2 (1-\lambda)T^{\frac{1}{2}}}
 \Bigg).
\end{align}
The performance bound is then given by 
\begin{align}
&\frac{1}{T}\sum_{i=0}^{T-1}\Big(\mE\|\nabla_x J(\bx_{c,i}, \by_{c,i})\|^2 + \mE\|\nabla_y J(\bx_{c,i}, \by_{c,i})\|^2\Big) \notag \\
&\le \mathcal{O}
\Bigg(
\frac{\kappa^2(\Delta'_0+1)}{(1-\lambda)^{2}T}
+\frac{\kappa^2(1-\lambda)\zeta^2_0}{T}
+\frac{\kappa^2\lambda^4\sqrt{K}\sigma^2}{\sqrt{N} T}
+\frac{\kappa^2\sigma^2}{T} 
\Bigg).
\end{align}
The CC and FC are given by 
\begin{align}
CC &\approx \mathcal{O}\Bigg(\frac{\kappa^2\varepsilon^{-2}}{(1-\lambda)^{2}}
+ \kappa^2\varepsilon^{-2} \Bigg), 
\\
FC &\approx 
\Big(CC \times p \times B + b_0\Big)\times 2d + CC \times (1-p) \times b \times 4d \approx \mathcal{O}\Bigg(
\frac{d\kappa^2\sqrt{N}\varepsilon^{-2}}{\sqrt{K}(1-\lambda)^{2}} + d\sqrt{\frac{N}{K}}
\Bigg).
\end{align}
\end{Corollary}

\begin{Corollary}[\textbf{ZO-L2S-ED}]
\label{corollary:offline:L2S+ed}
Under Assumptions \ref{main:assumption:costfunction}-\ref{main:assumptions:combinationmatrix} and the matrix condition in Lemma \ref{appendix:lemma:transformed_recursion},
we consider $\bu^j_{k,i} \equiv 0$ and   the ED strategy shown in in Table \ref{tab:matrix_choices}.
We set the hyperparameters as follows
\begin{align}
 B &= N, b = \mathcal{O}(1), b_0 = \sqrt{\frac{N}{K}}, p = \frac{1}{N}, \beta = 0, \notag\\
 \eta_x &= \mathcal{O}\Bigg(\frac{\sqrt{K}}{\kappa^2\sqrt{N}}\Bigg), \eta_y = \mathcal{O}\Bigg(
 \frac{\sqrt{K}}{\sqrt{N}}
 \Bigg), \notag\\
 \delta_x &= \mathcal{O}
 \Bigg(
 \frac{K^{\frac{1}{4}}}{d^{\frac{1}{2}}_1 T^{\frac{1}{2}}N^{\frac{1}{4}}}
 \Bigg),  \delta_y = \mathcal{O}
 \Bigg(
 \frac{K^{\frac{1}{4}}}{d^{\frac{1}{2}}_2T^{\frac{1}{2}} N^{\frac{1}{4}}}
 \Bigg).
\end{align}
The performance bound is given by 
\begin{align}
&\frac{1}{T}\sum_{i=0}^{T-1}\Big(\mE\|\nabla_x J(\bx_{c,i}, \by_{c,i})\|^2 + \mE\|\nabla_y J(\bx_{c,i}, \by_{c,i})\|^2\Big) \notag \\
&\le \mathcal{O}\Bigg(
\frac{\kappa^2\sqrt{N}(\Delta'_0+1)}{\sqrt{K}T}
+\frac{\kappa^2\zeta^2_0}{T(1-\lambda)^2}
+\frac{\kappa^2\lambda^2\sqrt{K}(\sigma^2+1)}{\sqrt{N}T(1-\lambda)^3}  + \frac{\kappa^2\sqrt{N}\sigma^2}{\sqrt{K}T}
\Bigg).
\end{align}
The CC and FC are given by 
\begin{align}
CC &\approx 
\mathcal{O}\Bigg(
\frac{\kappa^2\sqrt{N}\varepsilon^{-2}}{\sqrt{K}} 
+\frac{\kappa^2\epsilon^{-2}}{(1-\lambda)^2} + \frac{\kappa^2\sqrt{K}\varepsilon^{-2}}{\sqrt{N}(1-\lambda)^3}
\Bigg), \\
FC &\approx 
\Big(CC \times p \times B + b_0\Big)\times 2d + CC \times (1-p) \times b  \times 4d \notag\\
&\approx \mathcal{O}\Bigg(
\frac{d\kappa^2\sqrt{N}\varepsilon^{-2}}{\sqrt{K}} + \frac{d\kappa^2\epsilon^{-2}}{(1-\lambda)^2} + \frac{d\kappa^2\sqrt{K}\varepsilon^{-2}}{\sqrt{N}(1-\lambda)^3} + d\sqrt{\frac{N}{K}}
\Bigg).
\end{align}
\end{Corollary}
\begin{Corollary}[\textbf{ZO-L2S-GT}]
\label{corollary:offline:L2S+atc-gt}
Under Assumptions \ref{main:assumption:costfunction}-\ref{main:assumptions:combinationmatrix} and the matrix condition in Lemma \ref{appendix:lemma:transformed_recursion},
we consider $\bu^j_{k,i} \equiv 0$ and the ATC-GT shown  in Table \ref{tab:matrix_choices}.
We set the hyperparameters as follows
\begin{align}
 B &= N, b = \mathcal{O}(1), b_0 = \sqrt{\frac{N}{K}}, p = \frac{1}{N}, \beta = 0,  \notag\\
 \eta_x &= \mathcal{O}\Bigg(\frac{\sqrt{K}}{\kappa^2\sqrt{N}}\Bigg),\eta_y = \mathcal{O}\Bigg(
 \frac{\sqrt{K}}{\sqrt{N}}
 \Bigg), \notag\\
 \delta_x &= \mathcal{O}
 \Big(
 \frac{K^{\frac{1}{4}}}{d^{\frac{1}{2}}_1 T^{\frac{1}{2}}N^{\frac{1}{4}}}
 \Big),  \delta_y = \mathcal{O}
 \Big(
 \frac{K^{\frac{1}{4}}}{d^{\frac{1}{2}}_2T^{\frac{1}{2}} N^{\frac{1}{4}}}
 \Big).
\end{align}
The performance bound is given by 
\begin{align}
&\frac{1}{T}\sum_{i=0}^{T-1}\Big(\mE\|\nabla_x J(\bx_{c,i}, \by_{c,i})\|^2 + \mE\|\nabla_y J(\bx_{c,i}, \by_{c,i})\|^2\Big) \notag \\
&\le \mathcal{O}\Bigg(
\frac{\kappa^2\sqrt{N}(\Delta'_0+1)}{\sqrt{K}T}
+\frac{\kappa^2\zeta^2_0}{T(1-\lambda)^3}
+\frac{\kappa^2\lambda^4\sqrt{K}(\sigma^2+1)}{\sqrt{N}T(1-\lambda)^4} 
+ \frac{\kappa^2\sqrt{N}\sigma^2}{\sqrt{K}T}
\Bigg).
\end{align}
The CC and FC are given by 
\begin{align}
&CC \approx 
\mathcal{O}\Bigg(
\frac{\kappa^2\sqrt{N}\varepsilon^{-2}}{\sqrt{K}} 
+\frac{\kappa^2\epsilon^{-2}}{(1-\lambda)^3}  + \frac{\kappa^2\sqrt{K}\varepsilon^{-2}}{\sqrt{N}(1-\lambda)^4}
\Bigg),\\
&FC \approx 
\Big(CC \times p \times B +b_0 \Big)\times 2d + CC \times (1-p) \times b  \times 4d \notag \\
&\approx \mathcal{O}\Bigg(
\frac{d\kappa^2\sqrt{N}\varepsilon^{-2}}{\sqrt{K}} +\frac{d\kappa^2\varepsilon^{-2}}{(1-\lambda)^3} + \frac{d\kappa^2\sqrt{K}\varepsilon^{-2}}{\sqrt{N}(1-\lambda)^4}+ d\sqrt{\frac{N}{K}}
\Bigg).
\end{align}
\end{Corollary}

\section{Proof of hybrid ZO estimator}
We prove the convergence guarantees for the hybrid ZO estimator presented in Theorem \ref{main:theorem2}.
We start by establishing some key lemmas.
\subsection{Key Lemmas}
\begin{Lemma}[\textbf{Network gradient error}]
\label{appendix:lemma:coorinate_Nxc_Nyc}
Under Assumptions \ref{main:assumption:Lipschitz}, \ref{main:assumption:boundedvariance},  and \ref{main:asssumption:RD-ZO},
we consider \textbf{ZOMA} with 
$\bu^j_{k,i} \not= 0$ during $\bpi_i=0$, and choose
$\bar{\beta} \triangleq p +\beta - p \beta \le 1, \beta \le \frac{\sqrt{d}}{2\sigma_1}$, it holds that
\begin{align}
&\frac{1}{T}
\sum_{i=0}^{T-1}\mE\|\mS_{z,i}\|^2 \le 
\frac{2\mS_0}{T\bar{\beta}} + \frac{38dKL^2_fv^2_1v^2_2}{b\bar{\beta}T}
\sum_{i=1}^T
 \mE\|\mhE_{z,i}\|^2\notag \\
&
 +\frac{18dKL^2_fv^2_1v^2_2\mhE_0}{b\bar{\beta}T}+ \frac{20dKL^2_f}{b\bar{\beta}T}
 \sum_{i=0}^{T-1}
(\eta^2_x\mE\|\bg^x_{c,i}\|^2 + \eta^2_y\mE\|\bg^y_{c,i}\|^2)+ \frac{8K\beta^2\sigma^2_1}{b\bar{\beta}T}
\sum_{i=0}^{T-1}
 \mE\|\nabla^z_{c,i}\|^2  
+\frac{KC_2}{\bar{\beta}} +\frac{KC_1}{\bar{\beta}}. 
\label{appendix:proof:NxNy_statement}
\end{align}
where 
 $\mS_0 \triangleq \mE\|\mS_{z,0}\|^2$, $\mhE_0 \triangleq \mE\|\mhE_{z,0}\|^2 $ , $d \triangleq d_1+d_2$ and 
 \begin{align}
C_1 &\triangleq
3pL^2_f(d_1\delta^2_x+d^2_1\mu^2_x +d_2\delta^2_y +d^2_2\mu^2_y)
 + \frac{3L^2_f(d^2_1\mu^2_x + d^2_2\mu^2_y)}{b}.
 \notag\\
 C_2 &\triangleq \frac{6pC_0}{B}\mathbb{I}(B<N)+\frac{4\beta^2\sigma^2_0}{b}.
\end{align}
\end{Lemma}
\begin{proof}
Note that  $\forall k \in [K]$,
we have 
\begin{align}
 &\mE\|\bg^x_{k,i} - \nabla_x J^{\mu_x}_{k}(\bz_{k,i})\|^2\overset{(a)}{=}
 p\mE\Big\|
 \bq^x_{i,0}(\bz_{k,i};B)-\nabla_xJ^{\mu_x}_{k}(\bz_{k,i})  
 \Big\|^2 +(1-p)
 \mE\Big\|
 (1-\beta)\Big(\bg^x_{k,i-1} - \bq^x_{i}(\bz_{k,i-1};b)\Big)  \notag \\
 &\quad +  \bq^x_{i}(\bz_{k,i};b)
  - \nabla_x J^{\mu_x}_{k}(\bz_{k,i}) 
 \Big\|^2,
 \label{proof:appendix:NxNy_step1}
\end{align}
where $(a)$ follows from the fact that each agent employs a probabilistic gradient estimator. 
The first error term can be bounded as
\begin{align}
&\mE\Big\|
 \bq^x_{i,0}(\bz_{k,i};B)- \nabla_x J^{\mu_x}_{k}(\bz_{k,i})
 \Big\|^2 \notag\\
&=\mE\Big\|
 \bq^x_{i,0}(\bz_{k,i};B)- 
 \hnabla_x J_k(\bz_{k,i})
 +\hnabla_x J_k(\bz_{k,i})-\nabla^x_{k,i}+\nabla^x_{k,i}
 -
 \nabla_x J^{\mu_x}_{k}(\bz_{k,i})
 \Big\|^2 \notag \\
 &\le 
 3\mE\Big\|\bq^x_{i,0}(\bz_{k,i};B)- 
 \hnabla_x J_k(\bz_{k,i})\Big\|^2  + 3\mE\|\hnabla_x J_k(\bz_{k,i})
 -\nabla^x_{k,i}\|^2+ 3\mE\|
 \nabla_x J^{\mu_x}_k(\bz_{k,i})-\nabla^x_{k,i} \|^2 \notag \\
 &\overset{(a)}{\le} 
 \frac{3C_0}{B}\mathbb{I}(B<N)
 +3L^2_f(d_1\delta^2_x+d^2_1\mu^2_x),
\label{proof:appendix:NxNy_step2}
\end{align}
where $(a)$ follows from Lemmas \ref{appendix:lemma:randomsmoothing}, \ref{appendix:lemma:coordinate_zero_first}, and \ref{appendix:lemma:coordinate_large_bound}.
For the second term, we can derive 
\begin{align}
 &\mE\Big\|
 (1-\beta)\Big(\bg^x_{k,i-1} - \bq^x_{i}(\bz_{k,i-1};b) \Big)  +  \bq^x_{i}(\bz_{k,i};b)
 - \nabla_x J^{\mu_x}_{k}(\bz_{k,i}) 
 \Big\|^2 \notag \\
 &=\mE\Big\|
 (1-\beta)\Big(\bg^x_{k,i-1} - \nabla_x J^{\mu_x}_{k}(\bz_{k,i-1})\Big) - (1-\beta)\Big(
 \bq^x_{i}(\bz_{k,i-1};b) - \bq^x_{i}(\bz_{k,i};b)
 +\nabla_x J^{\mu_x}_{k}(\bz_{k,i})
 \notag\\
 &\quad - \nabla_x J^{\mu_x}_{k}(\bz_{k,i-1})\Big) 
 + \beta \Big( \bq^x_{i}(\bz_{k,i};b) 
 - \nabla_x J^{\mu_x}_{k}(\bz_{k,i}) \Big) 
 \Big\|^2 \notag \\
 &\overset{(a)}{\le} 
 (1-\beta)^2\mE\|\bg^x_{k,i-1} - \nabla_x J^{\mu_x}_{k}(\bz_{k,i-1})\|^2 +2\mE\Big\|\bq^x_{i}(\bz_{k,i-1};b)
  -
 \bq^x_{i}(\bz_{k,i};b)+\nabla_x J^{\mu_x}_{k}(\bz_{k,i})- \nabla_x J^{\mu_x}_{k}(\bz_{k,i-1})
 \Big\|^2  
 \notag\\
 &\quad
 + 2\beta^2\mE\Big\|
 \bq^x_{i}(\bz_{k,i};b)- \nabla_x J^{\mu_x}_{k}(\bz_{k,i})
 \Big\|^2 \notag \\
 &\overset{(b)}{\le} 
 (1-\beta)\mE\|\bg^x_{k,i-1} - \nabla_x J^{\mu_x}_{k}(\bz_{k,i-1})\|^2
 +\frac{2}{b^2}\sum_{j=1}^b\mE\Big\| \bq^x_{i}(\bz_{k,i-1};\bxi^j_{k,i}, \bu^j_{k,i}) - \bq^x_{i}(\bz_{k,i};\bxi^j_{k,i},\bu^j_{k,i})
 \Big\|^2
 \notag\\
 &\quad 
 +\frac{2\beta^2}{b} (\sigma^2_1\mE\|\nabla_x J(\bz_{k,i})\|^2
 +\sigma^2_0).\label{proof:appendix:NxNy_step3}
\end{align}
where $(a)$ follows from the Jensen's inequality and the fact that 
\begin{align}
\mE_{\bu^j_{k,i}} \mE_{\bxi^j_{k,i}}[\bq^x_{i}(\bz_{k,i};\bxi^j_{k,i},\bu^j_{k,i})] &= \nabla_x J^{\mu_x}_{k}(\bz_{k,i}), \notag \\
\mE_{\bu^j_{k,i}} \mE_{\bxi^j_{k,i}}[\bq^x_{i}(\bz_{k,i-1};\bxi^j_{k,i},\bu^j_{k,i})] &= \nabla_x J^{\mu_x}_{k}(\bz_{k,i-1}),
\end{align}
and $(b)$ follows from the fact that
$\{\bxi^j_{k, i}\}, \{\bu^j_{k,i}\}$
are i.i.d across $j$ as well as the inequality $\mE\|\bxi - \mE\bxi\|^2 \le \mE\|\bxi\|^2$ and Assumption \ref{main:asssumption:RD-ZO}.
Note that
\begin{align}
&\mE\Big\|\bq^x_{i}(\bz_{k,i};\bxi^j_{k,i}, \bu^j_{k,i}) - \bq^x_{i}(\bz_{k,i-1};\bxi^j_{k,i},\bu^j_{k,i})
 \Big\|^2 \notag \\
 &=\mE\Big\|
\frac{d_1(Q_k(\bx_{k,i}+\mu_x\bu^{j}_{k,i}, \by_{k,i};\bxi^j_{k, i}) - Q_k(\bz_{k,i};\bxi^j_{k, i}))}{\mu_x} \bu^j_{k,i} -
 \frac{d_1(Q_k(\bx_{k,i-1}+\mu_x\bu^{j}_{k,i},\by_{k,i-1};\bxi^j_{k, i}) - Q_k(\bz_{k,i-1};\bxi^j_{k, i}))}{\mu_x} \bu^j_{k,i} 
 \Big\|^2 \notag \\
&\overset{(a)}{=} d^2_1\mE\Big\| \Big(
\frac{Q_k(\bx_{k,i}+\mu_x\bu^{j}_{k,i},\by_{k,i};\bxi^j_{k, i}) - Q_k(\bz_{k,i};\bxi^j_{k, i}) -\langle \nabla_x Q_k(\bz_{k,i};\bxi^j_{k, i}),\mu_x \bu^j_{k,i} \rangle}{\mu_x} \bu^j_{k,i} 
 \Big)
\notag\\
&\quad - \Big(\frac{Q_k(\bx_{k,i-1}+\mu_x\bu^{j}_{k,i},\by_{k,i-1};\bxi^j_{k, i}) -Q_k(\bz_{k,i-1};\bxi^j_{k, i}) - \langle \nabla_x Q_k(\bz_{k,i-1};\bxi^j_{k, i}),\mu_x \bu^j_{k,i} \rangle}{\mu_x} \bu^j_{k,i} \Big)
 \notag\\
 &\quad+
\Big\langle \nabla_x Q_k(\bz_{k,i};\bxi^j_{k, i}) - \nabla_x Q_k(\bz_{k,i-1};\bxi^j_{k, i}), \bu^j_{k,i} \Big\rangle\bu^j_{k,i}
 \Big\|^2 \notag \\
 &\overset{(b)}{\le} 
\frac{3L^2_fd^2_1\mu^2_x}{2} +3d^2_1\mE 
(\nabla_x Q_k(\bz_{k,i};\bxi^j_{k, i}) - \nabla_x Q_k(\bz_{k,i-1};\bxi^j_{k, i}))^\top \bu^j_{k,i}(\bu^j_{k,i})^\top (\nabla_x Q_k(\bz_{k,i};\bxi^j_{k, i}) - \nabla_x Q_k(\bz_{k,i-1};\bxi^j_{k, i}))  \notag \\
&\overset{(c)}{\le} 
\frac{3L^2_fd^2_1\mu^2_x}{2} +3d_1\mE
\|\nabla_x Q_k(\bz_{k,i};\bxi^j_{k, i}) - \nabla_x Q_k(\bz_{k,i-1};\bxi^j_{k, i})\|^2  \notag \\
&\le 
\frac{3L^2_fd^2_1\mu^2_x}{2} +3d_1L^2_f\mE
\|\bz_{k,i} - \bz_{k,i-1}\|^2,
\label{proof:appendix:NxNy_step4}
\end{align}
where $(a)$ follows by inserting some terms, $(c)$ is due to $\mE\bu^j_{k,i} (\bu^j_{k,i})^\top = \frac{1}{d_1} \mathrm{I}_{d_1}$ (see, e.g.,  \cite{gao2018information}),
and  $(b)$ follows from $(\bu^j_{k,i})^\top \bu^j_{k,i} = 1$ and
\begin{align}
&\mE|Q_k(x+\mu_x \bu,y; \bxi)  - 
Q_k(z;\bxi) - \langle \nabla_x Q_k(z;\bxi), \mu_x \bu\rangle|^2 \le \frac{L^2_f\mu^4_x}{4}. \label{proof:inegrable}
\end{align}
The proof of relation \eqref{proof:inegrable} is provided below. For convenience, let 
\begin{align}
R \triangleq Q_k(x+\mu_x \bu,y; \bxi)  - 
Q_k(z;\bxi) - \langle \nabla_x Q_k(z;\bxi), \mu_x \bu \rangle,
\end{align}
where $\bu, \bxi$ are independent of $z$.
Using the mean value theorem, we have 
\begin{align}
&Q_{k}(x+\mu_x\bu, y;\bxi) - Q_k(z;\bxi) = \int_{0}^1 \langle 
\nabla_x Q_k(x+t\mu_x\bu,y;\bxi), \mu_x \bu
\rangle dt.
\end{align}
It follows that
\begin{align}
R = \int_{0}^1 \langle 
\nabla_x Q_k(x+t\mu_x\bu,y;\bxi) - \nabla_x Q_{k}(z;\bxi), \mu_x \bu
\rangle dt.
\end{align}
Taking the absolute value on both sides, have 
\begin{align}
|R| &= |\int_{0}^1 \langle 
\nabla_x Q_k(x+t\mu_x\bu,y;\bxi) - \nabla_x Q_{k}(z;\bxi), \mu_x \bu
\rangle dt| \notag \\
&\le \int_{0}^1 |\langle 
\nabla_x Q_k(x+t\mu_x\bu,y;\bxi) - \nabla_x Q_{k}(z;\bxi), \mu_x \bu
\rangle| dt \\
&\overset{(a)}{\le} 
\int_{0}^1 \|\nabla_x Q_k(x+t\mu_x\bu,y;\bxi) - \nabla_x Q_{k}(z;\bxi)\|\|\mu_x \bu\| dt \notag \\
&=\mu_x \int_{0}^1 \|\nabla_x Q_k(x+t\mu_x\bu, y;\bxi) - \nabla_x Q_{k}(z;\bxi)\| dt,
\end{align}
where  $(a)$ follows from Cauchy-Schwarz inequality.
Taking the square, we have 
\begin{align}
|R|^2 &= \mu^2_x \int_{0}^{1} \int_{0}^1 \|\nabla_x Q_k(x+t\mu_x\bu,y;\bxi) - \nabla_x Q_{k}(z;\bxi)\|\|\nabla_x Q_k(x+s\mu_x\bu,y;\bxi) - \nabla_x Q_{k}(z;\bxi)\|dtds,
\end{align}
where the integrand is nonnegative and measurable.
By Tonelli's theorem, we can switch the order of the integral and expectation, which gives
\begin{align}
&\mE|R|^2 \\
&\overset{(a)}{\le} 
\mu^2_x \int_{0}^{1} \int_{0}^1 (\mE\|\nabla_x Q_k(x+t\mu_x\bu,y;\bxi) - \nabla_x Q_{k}(z;\bxi)\|^2)^{\frac{1}{2}}
(\mE\|\nabla_x Q_k(x+s\mu_x\bu,y;\bxi) - \nabla_x Q_{k}(z;\bxi)\|^2)^{\frac{1}{2}}dtds \notag \\
&\overset{(b)}{\le} \mu^2_x
\int_{0}^{1} \int_{0}^1
L_f\|t\mu_x \bu\|L_f\|s\mu_x \bu\| dtds \notag\\
&
= L^2_f\mu^4_x \int_{0}^{1} \int_{0}^1 st dsdt = L^2_f\mu^4_x (\int_{0}^1 t dt)^2 = \frac{L^2_f\mu^4_x}{4},
\end{align}
where $(a)$ follows from $\mE\bx \by \le \sqrt{\mE\bx^2} \sqrt{\mE\by^2}$, $(b)$ follows from the expected $L_f$-smooth assumption.
Putting results together and choosing $p \le 1$, we get 
\begin{align}
&\mE\|\bg^x_{k,i} - \nabla_x J^{\mu_x}_{k}(\bz_{k,i})\|^2 \le 
 \frac{3pC_0}{B}\mathbb{I}(B<N)
+3pL^2_f(d_1\delta^2_x+d^2_1\mu^2_x)
 +(1-p)(1-\beta)\mE\|\bg^x_{k,i-1} - \nabla_x J^{\mu_x}_{k}(\bz_{k,i-1})\|^2
  \notag \\
 &\quad  + \frac{3L^2_fd^2_1\mu^2_x}{b} + \frac{6d_1L^2_f}{b}\mE
\|\bz_{k,i} - \bz_{k,i-1}\|^2
 + \frac{2\beta^2\sigma^2_1}{b}\mE\|\nabla_x J(\bz_{k,i})\|^2+ \frac{2\beta^2\sigma^2_0}{b}.
\end{align}
We then bound 
$\mE\|\bg^y_{k,i} - \nabla_y J^{\mu_y}_{k}(\bz_{k,i})\|^2$ using a symmetric argument.
Putting results together  $\forall k \in [K]$
and denoting $\bar{\beta} \triangleq p+\beta -p\beta \le 1$, we get
\begin{align}
&\mE\|\mS_{z,i}\|^2\le
(1-\bar{\beta})\mE\|\mS_{z,i-1}\|^2
 + KC_2
+KC_1+ \frac{6dL^2_f}{b}\mE
\|\mZ_i-\mZ_{i-1}\|^2 + \frac{2\beta^2\sigma^2_1}{b}\sum_{k=1}^{K}(\mE\|\nabla_x J(\bz_{k,i})\|^2+\mE\|\nabla_y J(\bz_{k,i})\|^2),
\end{align}
where $ d \triangleq d_1+d_2$ and
\begin{align}
C_1 &\triangleq
3pL^2_f(d_1\delta^2_x+d^2_1\mu^2_x +d_2\delta^2_y +d^2_2\mu^2_y)
 + \frac{3L^2_f(d^2_1\mu^2_x + d^2_2\mu^2_y)}{b}, \\
 C_2 &\triangleq \frac{6pC_0}{B}\mathbb{I}(B<N)+\frac{4\beta^2\sigma^2_0}{b}.
\end{align}
We then have
 \begin{align}
&\mE\|\mS_{z,i}\|^2\notag \\
&\le
(1-\bar{\beta})\mE\|\mS_{z,i-1}\|^2
 + \frac{6dL^2_f}{b}\mE
\|\mZ_i-\mZ_{i-1}\|^2
 +\frac{2\beta^2\sigma^2_1}{b}\sum_{k=1}^{K}\Big(\mE\|\nabla_x J(\bx_{k,i},\by_{k,i})- \nabla^x_{c,i} +\nabla^x_{c,i} \|^2 \notag\\
 &\quad 
  +\mE\|\nabla_y J(\bx_{k,i},\by_{k,i})-\nabla^y_{c,i}+ \nabla^y_{c,i}  \|^2 \Big) 
+KC_2 + KC_1 \notag \\ 
&\overset{(a)}{\le}
(1-\bar{\beta})\mE\|\mS_{z,i-1}\|^2 
 + \frac{6dL^2_f}{b}\Big(3Kv^2_1v^2_2
 (\mE\|\mhE_{z,i}\|^2
   + \mE\|\mhE_{z,i-1}\|^2)
+3K(\eta^2_x\mE\|\bg^x_{c,i-1}\|^2 + \eta^2_y\mE\|\bg^y_{c,i-1}\|^2)\Big) 
\notag \\
 &\quad +\frac{8\beta^2L^2_f\sigma^2_1}{b} \mE\|\mZ_{i} - \mZ_{c,i}\|^2 +\frac{4K\beta^2\sigma^2_1}{b}
 (\mE\|\nabla^x_{c,i}\|^2 
+ \mE\|\nabla^y_{c,i}\|^2)
+KC_2 + KC_1\notag 
\\
&\overset{(b)}{\le}
(1-\bar{\beta})\mE\|\mS_{z,i-1}\|^2 
 + \frac{18dKL^2_fv^2_1v^2_2}{b}
 (
 \mE\|\mhE_{z,i}\|^2
 +
 \mE\|\mhE_{z,i-1}\|^2)
 +\frac{18dKL^2_f}{b}
 (\eta^2_x\mE\|\bg^x_{c,i-1}\|^2 + \eta^2_y\mE\|\bg^y_{c,i-1}\|^2)
\notag\\
&\quad +\frac{8K\beta^2L^2_fv^2_1v^2_2\sigma^2_1}{b}
\mE\|\mhE_{z,i}\|^2 +\frac{8K\beta^2\sigma^2_1}{b}
 (\mE\|\nabla^x_{c,i-1}\|^2
 + \mE\|\nabla^y_{c,i-1}\|^2) + \frac{8K\beta^2L^2_f\sigma^2_1}{b}
(\eta^2_x\mE\|\bg^x_{c,i-1}\|^2 
\notag \\
&\quad +\eta^2_y\mE\|\bg^x_{c,i-1}\|^2)
+KC_2+KC_1,
\label{proof:appendix:NxNy_step5}
\end{align}
where $(a)$ follows from inequality \eqref{appendix:proof:incremental}, in $(b)$
we insert the true gradient $\nabla^x_{c,i-1}$ and apply Jensen's inequality. We further choose
\begin{align}
\frac{8K\beta^2L^2_f\sigma^2_1}{b} \le \frac{2dKL^2_f}{b} &\Longrightarrow \beta \le  \frac{\sqrt{d}}{2\sigma_1}, \\
\frac{8K\beta^2L^2_fv^2_1v^2_2\sigma^2_1}{b} \le \frac{2dKL^2_fv^2_1v^2_2}{b} &\Longrightarrow \beta \le \frac{\sqrt{d}}{2\sigma_1}.
\end{align}
Applying the above relation \eqref{proof:appendix:NxNy_step5} recursively for $i = 1, \ldots, T$ and averaging the results together, we obtain
\begin{align}
&\frac{1}{T}
\sum_{i=1}^{T}\mE\|\mS_{z,i}\|^2 \overset{(a)}{\le} 
\frac{\mS_0}{T\bar{\beta}} 
+ \frac{38dKL^2_fv^2_1v^2_2}{b\bar{\beta}T}
\sum_{i=1}^T 
 \mE\|\mhE_{z,i}\|^2
 +\frac{18dKL^2_fv^2_1v^2_2\mhE_0}{b\bar{\beta}T}+ \frac{20dKL^2_f}{b\bar{\beta}T}
 \sum_{i=0}^{T-1}
(\eta^2_x\mE\|\bg^x_{c,i}\|^2 + \eta^2_y\mE\|\bg^y_{c,i}\|^2)\notag \\
 &\quad 
 + \frac{8K\beta^2\sigma^2_1}{b\bar{\beta}T}
\sum_{i=0}^{T-1}
 (\mE\|\nabla^x_{c,i}\|^2 + \mE\|\nabla^y_{c,i}\|^2) +\frac{KC_2}{\bar{\beta}} +\frac{KC_1}{\bar{\beta}}, 
\label{appendix:proof:NxNy_last}
\end{align}
where $(a)$ follows from the useful inequality \eqref{proof:useful:recursion_first}, $\mS_0 \triangleq \mE\|\mS_{z,0}\|^2$ and $\mhE_0 \triangleq \mE\|\mhE_{z,0}\|^2$.
We can further deduce that
$\frac{1}{T}
\sum_{i=0}^{T-1}
\mE\|\mS_{z,i}\|^2
 \le \frac{1}{T}
\sum_{i=0}^{T}
\mE\|\mS_{z,i}\|^2  \le \frac{1}{T}
\sum_{i=1}^{T}
\mE\|\mS_{z,i}\|^2 + \mS_0/(T\bar{\beta})$. 
Plugging the results \eqref{appendix:proof:NxNy_last} into this relation, we can complete the proof.
\end{proof}
\begin{Lemma}[\textbf{Averaged gradient estimation error}]
\label{appendix:lemma:coorinate_variance_nxny}
Under Assumptions \ref{main:assumption:Lipschitz}, \ref{main:assumption:boundedvariance},  and \ref{main:asssumption:RD-ZO},
we consider \textbf{ZOMA} with 
$\bu^j_{k,i} \not= 0$ during $\bpi_i=0$, and choose
$\bar{\beta} \triangleq p +\beta - p \beta \le 1$
and $\beta \le \frac{\sqrt{d}}{2\sigma_1}$, it follows that:
\begin{align}
&\frac{1}{T}
\sum_{i=0}^{T-1}\mE\|\bs^z_{c,i}\|^2 \le 
\frac{2\bs_{c,0}}{T\bar{\beta}} 
+ \frac{38dL^2_fv^2_1v^2_2}{b\bar{\beta}KT}
\sum_{i=1}^T 
\mE\|\mhE_{z,i}\|^2+\frac{18dL^2_fv^2_1v^2_2\mhE_0}{b\bar{\beta}KT}+ \frac{20dL^2_f}{b\bar{\beta}KT}
 \sum_{i=0}^{T-1}
\eta^2_x\mE\|\bg^x_{c,i}\|^2\notag \\
&\quad+ 
\frac{20dL^2_f \eta^2_y}{b\bar{\beta}KT}
 \sum_{i=0}^{T-1} (1-(1-\bar{\beta})^{T-i})
\mE\|\bg^y_{c,i}\|^2 
 + \frac{8\beta^2\sigma^2_1}{b\bar{\beta}KT}
\sum_{i=0}^{T-1}
 (\mE\|\nabla^x_{c,i}\|^2 
+ \mE\|\nabla^y_{c,i}\|^2)  +\frac{C_2}{K\bar{\beta}} +\frac{C_1}{K\bar{\beta}}, 
\label{appendix:proof:Nxc_Nyc_last}
\end{align}
where  $\bs_{c,0}\triangleq \mE\|\bs^z_{c,0}\|^2$, $\mhE_0 \triangleq \mE\|\mhE_{z,0}\|^2 $ and $C_1, C_2$ are defined in Lemma \ref{appendix:lemma:coorinate_Nxc_Nyc}.
\end{Lemma}
\begin{proof}
The proof is similar to Lemma \ref{appendix:lemma:coorinate_Nxc_Nyc}.
\end{proof}
\begin{Lemma}[\textbf{Coupled consensus error}]
\label{appendix:lemma:random:exey}
Under Assumptions \ref{main:assumption:Lipschitz}, \ref{main:assumption:boundedvariance},  and \ref{main:asssumption:RD-ZO},
we consider \textbf{ZOMA} with $\bu^j_{k,i} \not= 0$ during $\bpi_i=0$ and choose
\begin{align}
\eta_x &\le \eta_y, \eta_y \le \min \Bigg\{\frac{(1-\rho)\underline{\lambda_b}}{15d^{\frac{1}{2}}L_fv_1v_2\lambda_a}, 
\frac{b^{\frac{1}{2}}(1-\rho)\underline{\lambda_b}}{7L_fv_1v_2\sigma_1 \lambda_a},  \frac{(1-\rho)\underline{\lambda_b}}{40(d)^{\frac{1}{2}}L_fv_1v_2 \lambda_a}\Bigg\},
\beta \le 1, p+\beta \le 1, \beta +bp \le b, b \ge 1.\notag
\end{align}
We can bound the coupled consensus error as follows:
\begin{align}
&\frac{1}{T}  \sum_{i=0}^{T-1}
\mE\|\mhE_{z,i}\|^2  \quad  \Big(\text{ or } \frac{1}{T}  \sum_{i=1}^{T}
\mE\|\mhE_{z,i}\|^2  \Big) \\
&\le
\frac{5\mhE_0}{T(1-\rho)}
+ I_0 
+L_0\sum_{i=0}^{T-1}(\eta^2_x\mE\|\bg^x_{c,i}\|^2
+\eta^2_y\mE\|\bg^y_{c,i}\|^2)+
 L_1\sum_{i=0}^{T-1}(\mE\|\nabla^x_{c,i}\|^2+\mE\|\nabla^y_{c,i}\|^2) 
+C_4 + \frac{4\eta^2_y\lambda^2_a( 5C_1+C_3)}{(1-\rho)^2\underline{\lambda^2_b}}
, \notag
\end{align}
where $\mS_0 \triangleq \mE\|\mS_{z,0}\|^2$, $\mhE_0 \triangleq \mE\|\mhE_{z,0}\|^2$,  and 
\begin{align}
I_0 &\triangleq  \frac{20\eta^2_y\lambda^2_a(p+\beta^2)}{(1-\rho)^2\underline{\lambda^2_b}K}
\Bigg(
\frac{2\mS_{0}}{T\bar{\beta}}  +\frac{18dKL^2_fv^2_1v^2_2\mhE_{0}}{b\bar{\beta}T}  \Bigg),\notag\\
L_0 &\triangleq \frac{508dL^2_f\lambda^2_a\eta^2_y}{(1-\rho)^2\underline{\lambda^2_b}T}, 
L_1 \triangleq \frac{184\lambda^2_a\sigma^2_1\eta^2_y\beta^2}{(1-\rho)^2\underline{\lambda^2_b}bT},\notag\\
 C_3 &\triangleq 
5p(d_1L^2_f\delta^2_x + d_2L^2_f\delta^2_y 
+d^2_1L^2_f\mu^2_x +d^2_2L^2_f\mu^2_y)
+ \frac{9L^2_f(d^2_1\mu^2_x+d^2_2\mu^2_y)}{2}   , \notag\\
C_4 &\triangleq 
\frac{4\eta^2_y\lambda^2_a}{(1-\rho)^2\underline{\lambda^2_b}}\Big(\frac{40pC_0}{B}\mathbb{I}(B<N) + \frac{26\beta^2\sigma^2_0}{b} \Big).
\end{align}
\end{Lemma}
\begin{proof}
Invoking Lemma \ref{appendix:lemma:transformed_recursion} and using $\eta_x \le \eta_y, \|\hat{\mU}^\top_x\|\le1, \|\hat{\mU}^\top_y\|\le1$, we can establish that 
\begin{align}
&\|\mhE_{z,i+1}\|^2  \le  
\rho\|\mhE_{z,i}\|^2 + \frac{\eta^2_y\lambda^2_{a}\|\mG_{z,i+1} -\mG_{z,i}\|^2}{(1-\rho)\underline{\lambda^2_{b}}K}. \label{appendix:proof:lemma_exey_random_step1}
\end{align}
$\forall k \in [K]$, it follows that
\begin{align}
&\mE\|\bg^x_{k,i+1} - \bg^x_{k,i}\|^2 =
p\mE\Big\|\bq^x_{i+1,0}(\bz_{k,i+1};B) - \bg^x_{k,i}\Big\|^2+ (1-p)\mE\Big\|
(1-\beta)\Big(\bg^x_{k,i} - \bq^x_{i+1}(\bz_{k,i};b)\Big)
+ \bq^x_{i+1}(\bz_{k,i+1}; b) - \bg^x_{k,i}
\Big\|^2.
\end{align}
The first term corresponds to the event $\bpi_{i+1} = 1$, hence
\begin{align}
&\mE\Big\|\bq^x_{i+1,0}(\bz_{k,i+1};B) - \bg^x_{k,i}\Big\|^2  \notag \\
&=
\mE\Big\|\bq^x_{i+1,0}(\bz_{k,i+1};B) - \hnabla_x J_k(\bz_{k,i+1})
+
\hnabla_x J_k(\bz_{k,i+1})- \nabla^x_{k,i+1}
+\nabla^x_{k,i+1}-\nabla^x_{k,i}+
\nabla^x_{k,i}-\nabla_x J^{\mu_x}_{k}(\bz_{k,i}) 
\notag \\
& \quad+\nabla_x J^{\mu_x}_{k}(\bz_{k,i}) 
- \bg^x_{k,i}\Big\|^2 \notag \\
&\overset{(a)}{\le}
\frac{5C_0}{B}\mathbb{I}(B<N) + 5(d_1L^2_f\delta^2_x+d^2_1L^2_f\mu^2_x)
+5L^2_f\mE\|\bz_{k,i+1}  -\bz_{k,i}\|^2
+5\mE\|\nabla_x J^{\mu_x}_{k}(\bz_{k,i}) 
- \bg^x_{k,i}\|^2.
\end{align}
where $(a)$ follows from 
Lemmas \ref{appendix:lemma:randomsmoothing}, \ref{appendix:lemma:coordinate_zero_first}, \ref{appendix:lemma:coordinate_large_bound} and Assumption \ref{main:assumption:Lipschitz}.
The second term can be bounded as follows 
\begin{align}
&\mE\Big\|
(1-\beta)\Big(\bg^x_{k,i} - \bq^x_{i+1}(\bz_{k,i};b)\Big)+ \bq^x_{i+1}(\bz_{k,i+1};b) - \bg^x_{k,i}
\Big\|^2 \notag \\
&=
\mE\Big\|
-\beta(\bg^x_{k,i} - \nabla_x J^{\mu_x}_{k}(\bz_{k,i}))
-
\Big(\bq^x_{i+1}(\bz_{k,i};b)
 - \bq^x_{i+1}(\bz_{k,i+1};b)\Big)
+\beta \Big(
\bq^x_{i+1}(\bz_{k,i};b) - \nabla_x J^{\mu_x}_{k}(\bz_{k,i})
\Big)
\Big\|^2 \notag \\
&\le 3\beta^2
\mE\|\bg^x_{k,i} - \nabla_x J^{\mu_x}_{k}(\bz_{k,i})\|^2
+3 \mE\Big\|
\bq^x_{i+1}(\bz_{k,i+1};b) 
 - \bq^x_{i+1}(\bz_{k,i};b)
\Big\|^2+3\beta^2 \mE\Big\|
\bq^x_{i+1}(\bz_{k,i};b) - \nabla_x J^{\mu_x}_{k}(\bz_{k,i})
\Big\|^2 \notag \\
&\overset{(a)}{\le} 
3\beta^2
\mE\|\bg^x_{k,i} - \nabla_x J^{\mu_x}_{k}(\bz_{k,i})\|^2
+
\frac{9L^2_fd^2_1\mu^2_x}{2} +9d_1L^2_f
\mE
\|\bz_{k,i+1} - \bz_{k,i}\|^2+
\frac{3\beta^2}{b}\Big( \sigma^2_1
\mE\|\nabla_x J(\bz_{k,i})\|^2 +\sigma^2_0
\Big) \notag \\
&\overset{(b)}{\le} 
3\beta^2
\mE\|\bg^x_{k,i} - \nabla_x J^{\mu_x}_{k}(\bz_{k,i})\|^2
+
\frac{9L^2_fd^2_1\mu^2_x}{2} 
 +9d_1L^2_f\mE
\|\bz_{k,i+1}- \bz_{k,i}\|^2 +
\frac{3\beta^2}{b}\Big( 2L^2_f\sigma^2_1\mE\|\bz_{k,i}-\bz_{c,i}\|^2   \notag\\
&\quad+2\sigma^2_1
\mE\|\nabla^x_{c,i}\|^2+\sigma^2_0
\Big) 
\end{align}
where $(a)$ follows from
\eqref{proof:appendix:NxNy_step4} and Assumption \ref{main:asssumption:RD-ZO},
and $(b)$ follows from
\begin{align}
&\|\nabla_x J(\bz_{k,i})\|^2 \le 
2\|\nabla^x_{c,i}\|^2
+2 \|\nabla_x J(\bz_{k,i})- \nabla^x_{c,i}\|^2 \notag \\
&\le 2\|\nabla^x_{c,i}\|^2
+2L^2_f\|\bz_{k,i} -\bz_{c,i}\|^2. 
\end{align}
% Combining the results $\forall k \in [K]$, we have 
% \begin{align}
% &\mE\|\mG_{x,i+1} - \mG_{x,i}\|^2\le 
% \frac{5pKC_0}{B}\mathbb{I}(B<N) + 5pK(d_1L^2_f\delta^2_x+d^2_1L^2_f\mu^2_x)\notag \\
% &
% +5pL^2_f\mE\|\mZ_{i+1} - \mZ_{i}\|^2 +5p\mE\|\mS_{x,i}\|^2 +
% 3(1-p)\beta^2\mE\|\mS_{x,i}\|^2\notag \\
% & 
% +\frac{9(1-p)KL^2_fd^2_1\mu^2_x}{2}
% +9(1-p)d_1L^2_f
% \mE\|\mZ_{i+1} -\mZ_{i}\|^2 
% \notag\\
% & +
% \frac{6(1-p)\beta^2L^2_f\sigma^2_1}{b}
% \mE\|\mZ_i -\mZ_{c,i}\|^2  + \frac{6K\beta^2\sigma^2_1(1-p)}{b} \mE\|\nabla^x_{c,i}\|^2
% \notag\\
% &\quad+\frac{3K\beta^2\sigma^2_0(1-p)}{b}.
% \end{align}
We can bound 
$\mE\|\bg^y_{k,i+1} -\bg^y_{k,i}\|^2$ using a symmetric argument. Putting results $\forall k \in [K]$ together,
we get 
\begin{align}
&\mE\|\mG_{z,i+1} - \mG_{z,i}\|^2 \notag \\
&
\overset{(a)}{\le} \frac{10pKC_0}{B}\mathbb{I}(B<N)
+5pK(d_1L^2_f\delta^2_x + d_2L^2_f\delta^2_y+d^2_1L^2_f\mu^2_x + d^2_2L^2_f\mu^2_y) 
+9(d_1+d_2)L^2_f\mE\|\mZ_{i+1} - \mZ_{i}\|^2 \notag\\
&\quad +5(p+\beta^2)\mE\|\mS_{z,i}\|^2
+ \frac{9KL^2_f(d^2_1\mu^2_x+d^2_2\mu^2_y)}{2}  +\frac{6K\beta^2\sigma^2_0}{b} +\frac{12\beta^2L^2_f\sigma^2_1}{b}\mE\|\mZ_i - \mZ_{c,i}\|^2
+ \frac{6K\beta^2\sigma^2_1}{b}
\mE\|\nabla^z_{c,i}\|^2 \notag \\
&
\overset{(b)}{\le} \frac{10pKC_0}{B}\mathbb{I}(B<N) 
+KC_3
+27dKL^2_fv^2_1v^2_2\Big(\mE\|\mhE_{z,i+1}\|^2
+\mE\|\mhE_{z,i}\|^2\Big)
+27dKL^2_f(\eta^2_x\mE\|\bg^x_{c,i}\|^2 + \eta^2_y\mE\|\bg^y_{c,i}\|^2)
\notag \\&\quad +5(p+\beta^2)\mE\|\mS_{z,i}\|^2
  +\frac{6K\beta^2\sigma^2_0}{b}+\frac{12K\beta^2L^2_fv^2_1v^2_2\sigma^2_1}{b} \mE\|\mhE_{z,i}\|^2+ \frac{6K\beta^2\sigma^2_1}{b}
(\mE\|\nabla^x_{c,i}\|^2 +\mE\|\nabla^y_{c,i}\|^2).
\end{align}
where $(a)$ follows from $(1-p)+p =1, d_1 \ge 1$, and $(b)$ follows from \eqref{proof:useful:consensus}, \eqref{appendix:proof:incremental}, and
\begin{align}
C_3 &\triangleq 
5p(d_1L^2_f\delta^2_x + d_2L^2_f\delta^2_y 
+d^2_1L^2_f\mu^2_x +d^2_2L^2_f\mu^2_y)
+ \frac{9L^2_f(d^2_1\mu^2_x+d^2_2\mu^2_y)}{2}.
\end{align}
As a result, we have
\begin{align}
&\mE\|\mhE_{z,i+1}\|^2 \le  
\rho\mE\|\mhE_{z,i}\|^2 + \frac{\eta^2_y\lambda^2_{a}}{(1-\rho)\underline{\lambda^2_b}K}\Bigg(27dKL^2_fv^2_1v^2_2
(\mE\|\mhE_{z,i+1}\|^2 +\mE\|\mhE_{z,i}\|^2)+27dKL^2_f(\eta^2_x\mE\|\bg^x_{c,i}\|^2
 + \eta^2_y\mE\|\bg^y_{c,i}\|^2)\notag\\
&\quad +5(p+\beta^2)\mE\|\mS_{z,i}\|^2
+\frac{12K\beta^2L^2_fv^2_1v^2_2\sigma^2_1}{b}\mE\|\mhE_{z,i}\|^2+ \frac{6K\beta^2\sigma^2_1}{b}
(\mE\|\nabla^x_{c,i}\|^2  +\mE\|\nabla^y_{c,i}\|^2)+\frac{10pKC_0}{B}\mathbb{I}(B<N) 
\notag \\&\quad 
+ \frac{6K\beta^2\sigma^2_0}{b} + KC_3\Bigg) .
\end{align}
Setting  
\begin{align}
&\frac{27dKL^2_fv^2_1v^2_2 \eta^2_y\lambda^2_a}{(1-\rho)\underline{\lambda^2_b}K}
\le \frac{1-\rho}{8} \Longrightarrow \eta_y \le \frac{(1-\rho)\underline{\lambda_b}}{15d^{\frac{1}{2}}L_fv_1v_2\lambda_a} \notag \\  &\frac{\eta^2_y \lambda^2_a}{(1-\rho)\underline{\lambda^2_b}K} \times 
\frac{12KL^2_fv^2_1v^2_2\sigma^2_1\beta^2}{b} \le \frac{1-\rho}{4} \Longrightarrow \beta \le 1, 
\eta_y \le \frac{b^{\frac{1}{2}}(1-\rho)\underline{\lambda_b}}{7L_fv_1v_2\sigma_1 \lambda_a} ,
\end{align}
and using $\rho + \frac{1-\rho}{8} + \frac{1-\rho}{4}  \le  \frac{1+\rho}{2}$,
it follows that 
\begin{align}
&\mE\|\mhE_{z,i+1}\|^2 \le 
\frac{1+\rho}{2}\mE\|\mhE_{z,i}\|^2 +  \frac{5\eta^2_y\lambda^2_{a}(p+\beta^2)}{(1-\rho)\underline{\lambda^2_b}K}\mE\|\mS_{z,i}\|^2+ 
 \frac{1-\rho}{8}\mE\|\mhE_{z,i+1}\|^2
+\frac{27dL^2_f\lambda^2_a\eta^2_y}{(1-\rho)\underline{\lambda^2_b}}(\eta^2_x\mE\|\bg^x_{c,i}\|^2 + \eta^2_y\mE\|\bg^y_{c,i}\|^2)\notag\\
&\quad
+ \frac{6\lambda^2_a\eta^2_y\beta^2\sigma^2_1}{(1-\rho)\underline{\lambda^2_b}b}
\mE\|\nabla^z_{c,i}\|^2 + \frac{\eta^2_y\lambda^2_a}{(1-\rho)\underline{\lambda^2_b}}\Big(\frac{10pC_0}{B}\mathbb{I}(B<N) 
 + \frac{6\beta^2\sigma^2_0}{b} + C_3\Big).
\end{align}
Applying the above relation recursively for 
$i=0, \ldots, T-1$
and averaging the results,
we can establish that 
\begin{align}
&\frac{1}{T}  \sum_{i=0}^{T-1}
\mE\|\mhE_{z,i+1}\|^2 =\frac{1}{T}  \sum_{i=1}^{T}
\mE\|\mhE_{z,i}\|^2 \notag \\
&\overset{(a)}{\le} 
\frac{2\mhE_0}{T(1-\rho)}
+ \frac{10\eta^2_y\lambda^2_a(p+\beta^2)}{(1-\rho)^2\underline{\lambda^2_b}KT}
\sum_{i=0}^{T-1}
\mE\|\mS_{z,i}\|^2+\frac{1}{4T}
\sum_{i=1}^{T}
\mE\|\mhE_{z,i}\|^2
+\frac{54dL^2_f\lambda^2_a\eta^2_y}{(1-\rho)^2\underline{\lambda^2_b}T}\sum_{i=0}^{T-1}(\eta^2_x\mE\|\bg^x_{c,i}\|^2+\eta^2_y\mE\|\bg^y_{c,i}\|^2)\notag \\
&\quad 
+\frac{12\lambda^2_a\eta^2_y\beta^2\sigma^2_1}{(1-\rho)^2\underline{\lambda^2_b}bT}
\sum_{i=0}^{T-1}
\mE\|\nabla^z_{c,i}\|^2 + \frac{2\eta^2_y\lambda^2_a}{(1-\rho)^2\underline{\lambda^2_b}}\Big(\frac{10pC_0}{B}\mathbb{I}(B<N)
 + \frac{6\beta^2\sigma^2_0}{b} + C_3\Big) \notag \\
&\overset{(b)}{\le} 
\frac{2\mhE_0}{T(1-\rho)}
+ \frac{10\eta^2_y\lambda^2_a(p+\beta^2)}{(1-\rho)^2\underline{\lambda^2_b}K}
\Bigg(
\frac{2\mS_{0}}{T\bar{\beta}} 
+ \frac{38dKL^2_fv^2_1v^2_2}{b\bar{\beta}T}
\sum_{i=1}^T
\mE\|\mhE_{z,i}\|^2+\frac{18dKL^2_fv^2_1v^2_2\mhE_0}{b\bar{\beta}T}+ \frac{20dKL^2_f}{b\bar{\beta}T}
 \notag \\
 &\quad 
\times
 \sum_{i=0}^{T-1}
(\eta^2_x\mE\|\bg^x_{c,i}\|^2 + \eta^2_y\mE\|\bg^y_{c,i}\|^2) + \frac{8K\beta^2\sigma^2_1}{b\bar{\beta}T}
\sum_{i=0}^{T-1}
 \mE\|\nabla^z_{c,i}\|^2 
+\underbrace{\frac{6pKC_0}{\bar{\beta}B}\mathbb{I}(B<N)+\frac{4K\beta^2\sigma^2_0}{b\bar{\beta}}}_{KC_2/\bar{\beta}} +\frac{KC_1}{\bar{\beta}}\Bigg)
\notag \\
&\quad +\frac{1}{4T}
\sum_{i=1}^{T}\mE\|\mhE_{z,i}\|^2
+\frac{54dL^2_f\lambda^2_a\eta^2_y}{(1-\rho)^2\underline{\lambda^2_b}T}\sum_{i=0}^{T-1}(\eta^2_x\mE\|\bg^x_{c,i}\|^2
+\eta^2_y\mE\|\bg^y_{c,i}\|^2)
+\frac{12\lambda^2_a\eta^2_y\beta^2\sigma^2_1}{(1-\rho)^2\underline{\lambda^2_b}bT}
\sum_{i=0}^{T-1}\mE\|\nabla^z_{c,i}\|^2
 \notag \\
&\quad
+ \frac{2\eta^2_y\lambda^2_a}{(1-\rho)^2\underline{\lambda^2_b}}\Big(\frac{10pC_0}{B}\mathbb{I}(B<N) + \frac{6\beta^2\sigma^2_0}{b} + C_3\Big),
\end{align}
where $(a)$ and $(b)$ follow from the useful inequality \eqref{proof:useful:recursion_first} and Lemma \ref{appendix:lemma:coorinate_Nxc_Nyc}.
We next choose 
\begin{align}
&\frac{10\eta^2_y\lambda^2_a(p+\beta^2)}{(1-\rho)^2\underline{\lambda^2_b}K}
    \times \frac{38dKL^2_fv^2_1v^2_2}{b\bar{\beta}} \le \frac{1}{4} \Longrightarrow 
    \eta_y \le \frac{(1-\rho)\underline{\lambda_b}}{40(d)^{\frac{1}{2}}L_fv_1v_2 \lambda_a}\sqrt{\frac{b\bar{\beta}}{p+\beta^2}} \\
&\frac{10\eta^2_y\lambda^2_a(p+\beta^2)}{(1-\rho)^2\underline{\lambda^2_b}K}
\times \frac{20dKL^2_f}{b\bar{\beta}}
\le \frac{200dL^2_f\lambda^2_a\eta^2_y}{(1-\rho)^2\underline{\lambda^2_b}} \Longrightarrow  p+\beta^2 \le b\bar{\beta} \Longrightarrow b \ge1, \beta +bp\le b \\
&\frac{10\eta^2_y\lambda^2_a(p+\beta^2)}{(1-\rho)^2\underline{\lambda^2_b}K}
\times \frac{8K\beta^2\sigma^2_1}{b\bar{\beta}} \le \frac{80\lambda^2_a\sigma^2_1\eta^2_y\beta^2}{(1-\rho)^2\underline{\lambda^2_b}b} \Longrightarrow 
p+\beta^2 \le \bar{\beta} \Longrightarrow p+\beta \le 1 \\
&\frac{10\eta^2_y\lambda^2_a(p+\beta^2)}{(1-\rho)^2\underline{\lambda^2_b}K}
\times 
\frac{6pKC_0}{\bar{\beta} B} \le \frac{2\eta^2_y\lambda^2_a}{(1-\rho)^2\underline{\lambda^2_b}} \times \frac{30pC_0}{B} \Longrightarrow
p+\beta^2 \le \bar{\beta} \Longrightarrow p+\beta \le 1 \\
&\frac{10\eta^2_y\lambda^2_a(p+\beta^2)}{(1-\rho)^2\underline{\lambda^2_b}K}
\times \frac{4K\beta^2\sigma^2_0}{b\bar{\beta}} \le 
\frac{2\eta^2_y\lambda^2_a}{(1-\rho)^2\underline{\lambda^2_b}} \times 
\frac{20\beta^2\sigma^2_0}{b} \Longrightarrow  p+\beta \le 1 \\
&\frac{10\eta^2_y\lambda^2_a(p+\beta^2)}{(1-\rho)^2\underline{\lambda^2_b}K}
\times \frac{K}{\bar{\beta}}
\le \frac{2\eta^2_y\lambda^2_a \times 5}{(1-\rho)^2\underline{\lambda^2_b}} \Longrightarrow
p+\beta \le 1.
\end{align}
Using $b\bar{\beta} \ge p +\beta^2$,
the above first condition can be simplified into 
\begin{align}   \eta_y \le \frac{(1-\rho)\underline{\lambda_b}}{40(d)^{\frac{1}{2}}L_fv_1v_2 \lambda_a}.
\end{align}
Using the above results, we get the following results 
% \begin{align}
% &\frac{1}{T}  \sum_{i=1}^{T}
% (\mE\|\mhE_{x,i}\|^2 + \mE\|\mhE_{y,i}\|^2) \notag \\
% &\le 
% \frac{4\mhE_0}{T(1-\rho)}
% + \frac{20\eta^2_y\lambda^2_a(p+\beta^2)}{(1-\rho)^2\underline{\lambda^2_b}K}
% \Bigg(
% \frac{2\mS_{0}}{T\bar{\beta}}  +\frac{18dKL^2_fv^2_1v^2_2\mhE_{0}}{b\bar{\beta}T}  \Bigg)
% \notag \\
% &\quad 
% +\frac{508dL^2_f\lambda^2_a\eta^2_y}{(1-\rho)^2\underline{\lambda^2_b}T}\sum_{i=0}^{T-1}(\eta^2_x\mE\|\bg^x_{c,i}\|^2
% +\eta^2_y\mE\|\bg^y_{c,i}\|^2)
% +\frac{184\lambda^2_a\sigma^2_1\eta^2_y\beta^2}{(1-\rho)^2\underline{\lambda^2_b}bT}\sum_{i=0}^{T-1}(\mE\|\nabla_x J(\bx_{c,i}, \by_{c,i})\|^2+\mE\|\nabla_y J(\bx_{c,i}, \by_{c,i})\|^2) 
% \notag\\
% &\quad 
% + \frac{4\eta^2_y\lambda^2_a}{(1-\rho)^2\underline{\lambda^2_b}}\Big(\frac{40pC_0}{B}\mathbb{I}(B<N) + \frac{26\beta^2\sigma^2_0}{b} + 5C_1+C_3\Big).
% \end{align}
% Finally, we can derive that 
\begin{align}
&\frac{1}{T}  \sum_{i=0}^{T-1}
\mE\|\mhE_{z,i}\|^2 \le \frac{1}{T}  \sum_{i=1}^{T}
\mE\|\mhE_{z,i}\|^2  + \frac{\mE\|\mhE_{z,0}\|^2}{T(1-\rho)} \notag\\
&\le
\frac{5\mhE_0}{T(1-\rho)}
+ \frac{20\eta^2_y\lambda^2_a(p+\beta^2)}{(1-\rho)^2\underline{\lambda^2_b}K}
\Bigg(
\frac{2\mS_{0}}{T\bar{\beta}}  +\frac{18dKL^2_fv^2_1v^2_2\mhE_{0}}{b\bar{\beta}T}  \Bigg)
+\frac{508dL^2_f\lambda^2_a\eta^2_y}{(1-\rho)^2\underline{\lambda^2_b}T}\sum_{i=0}^{T-1}(\eta^2_x\mE\|\bg^x_{c,i}\|^2
+\eta^2_y\mE\|\bg^y_{c,i}\|^2)
\notag\\
&\quad +\frac{184\lambda^2_a\sigma^2_1\eta^2_y\beta^2}{(1-\rho)^2\underline{\lambda^2_b}bT}\sum_{i=0}^{T-1}\mE\|\nabla^z_{c,i}\|^2
+ \frac{4\eta^2_y\lambda^2_a}{(1-\rho)^2\underline{\lambda^2_b}}\Big(\frac{40pC_0}{B}\mathbb{I}(B<N) + \frac{26\beta^2\sigma^2_0}{b} + 5C_1+C_3\Big). \notag
\end{align}
\end{proof}
\begin{Lemma}[\textbf{Descent relation}]
\label{apendix:lemma:RD:cost}
Under Assumption \ref{main:assumption:Lipschitz}, let $P_{\mu_x}(x) \triangleq \mE_{\bu \sim {\rm Unil}(\mathbb{B}^{d_1})} [P(x+\mu_x \bu)]$, where $P(x) = \max_y J(x, y)$, it holds that 
\begin{align}
& P_{\mu_x}(\bx_{c,i+1})\le  P_{\mu_x}(\bx_{c,i})
- \frac{\eta_x}{4}
\|\nabla P(\bx_{c,i})\|^2-\frac{\eta_x(1-L\eta_x)}{2}\|\bg^x_{c,i}\|^2
 +6\eta_x\kappa L_f
\Delta^y_{c,i}+9\eta_x L^2_fv^2_1v^2_1\|\mhE_{z,i}\|
 +3\eta_x\|\bs^x_{c,i}\|^2 \notag\\
 &\quad 
 +2\eta_x\mu^2_xL^2d^2_1.
\label{appendix:lemma:random_cost:statement}
\end{align}
\end{Lemma}
\begin{proof}
According to Lemma \ref{appendix:lemma:randomsmoothing}, given a $L$-smooth function $P(z)$, then $P_{\mu_x}(x) \triangleq \mE_{\bu \sim {\rm Unil}(\mathbb{B}^{d_1})}P(x+\mu_x \bu)$ is $L^\prime (\le L)$-smooth,  it follows that
\begin{align}
&P_{\mu_x}(\bx_{c,i+1})
\le P_{\mu_x}(\bx_{c,i})
+\langle  \nabla P_{\mu_x}(\bx_{c,i}), \bx_{c,i+1}- \bx_{c,i}\rangle \notag + \frac{L\eta^2_x}{2}\|\bg^x_{c,i}\|^2\\
&\le P_{\mu_x}(\bx_{c,i})
- \eta_x\langle  \nabla P_{\mu_x}(\bx_{c,i}), \bg^x_{c,i}\rangle + \frac{L\eta^2_x}{2}\|\bg^x_{c,i}\|^2 \notag \\
&\le 
P_{\mu_x}(\bx_{c,i})
-\frac{\eta_x}{2}
\|\nabla P_{\mu_x}(\bx_{c,i})\|^2
- \frac{\eta_x}{2}\|\bg^x_{c,i}\|^2
+\frac{\eta_x}{2}
\|\nabla P_{\mu_x}(\bx_{c,i}) - \bg^x_{c,i}\|^2
+ \frac{L\eta^2_x}{2}\|\bg^x_{c,i}\|^2.
\label{appendix:lemma:random_cost:step1}
\end{align}
Note that
\begin{align}
&\|\nabla P(\bx_{c,i})\|=\|\nabla P(\bx_{c,i})
-\nabla P_{\mu_x}(\bx_{c,i}) +
\nabla P_{\mu_x}(\bx_{c,i})\|^2 \notag\\
&\le 2\|\nabla P_{\mu_x}(\bx_{c,i})
-\nabla P(\bx_{c,i}) \|^2 +2\|\nabla P_{\mu_x}(\bx_{c,i})\|^2,
\end{align}
we have 
\begin{align}
-\|\nabla P_{\mu_x}(\bx_{c,i})\|^2
&\le \|\nabla P_{\mu_x}(\bx_{c,i})
-\nabla P(\bx_{c,i}) \|^2 -\frac{\|\nabla P(\bx_{c,i})\|^2}{2}\le \frac{\mu^2_xL^2d^2_1}{4}-\frac{\|\nabla P(\bx_{c,i})\|^2}{2}.
\end{align}
where the second inequality is due to Lemma \ref{appendix:lemma:randomsmoothing}.
Plugging the above results into 
\eqref{appendix:lemma:random_cost:step1}
and denoting 
\begin{align}
J^{\mu_x}(\bz_{c,i}) &\triangleq
\frac{1}{K}\sum_{k=1}^{K}J^{\mu_x}_{k}(\bz_{c,i}),
\end{align}
we can derive that 
\begin{align}
& P_{\mu_x}(\bx_{c,i+1}) \notag\\
&\le 
P_{\mu_x}(\bx_{c,i})
-\frac{\eta_x}{4}
\|\nabla P(\bx_{c,i})\|^2 + \frac{\eta_x\mu^2_xL^2d^2_1}{8}
- \frac{\eta_x}{2}\|\bg^x_{c,i}\|^2 
+\frac{\eta_x}{2}
\|\nabla P_{\mu_x}(\bx_{c,i}) - \bg^x_{c,i}\|^2
+ \frac{L\eta^2_x}{2}\|\bg^x_{c,i}\|^2  \notag \\
&\overset{(a)}{\le} P_{\mu_x}(\bx_{c,i})
- \frac{\eta_x}{4}
\|\nabla P(\bx_{c,i})\|^2
+ \frac{\eta_x\mu^2_xL^2d^2_1}{8}
-\frac{\eta_x}{2}
\|\bg^x_{c,i}\|^2
+3\eta_x\Big(\|
\nabla P_{\mu_x}(\bx_{c,i}) - 
\nabla P(\bx_{c,i})
\|^2+\|
\nabla P(\bx_{c,i}) - 
\nabla^x_{c,i}\|^2 \notag \\
&\quad 
+\|
\nabla^x_{c,i}
- \nabla_x J^{\mu_x}(\bz_{c,i})
\|^2+\Big\|\nabla_x J^{\mu_x}(\bz_{c,i})- \frac{1}{K}\sum_{k=1}^{K} \nabla_x 
J^{\mu_x}_k(\bz_{k,i})\Big\|^2 +
\Big\| \frac{1}{K}\sum_{k=1}^{K} \nabla_x J^{\mu_x}_k(\bz_{k,i})
- \bg^x_{c,i}
\Big\|^2\Big)  + \frac{L\eta^2_x}{2}\|\bg^x_{c,i}\|^2  \notag \\
&\overset{(b)}{\le} 
P_{\mu_x}(\bx_{c,i})
- \frac{\eta_x}{4}
\|\nabla P(\bx_{c,i})\|^2
-\frac{\eta_x(1-L\eta_x)}{2}
\|\bg^x_{c,i}\|^2 
+3\eta_xL^2_f\|
\by^o(\bx_{c,i}) - \by_{c,i}\|^2\notag +\frac{3\eta_xL^2_f}{K}\|\mZ_{i} - \mZ_{c,i}\|^2  \\
&\quad +3\eta_x\|
\bs^x_{c,i}
\|^2
  + \frac{3}{2}\eta_x \mu^2_xL^2d^2_1 +\frac{\eta_x\mu^2_xL^2d^2_1}{8} \notag \\
&\overset{(c)}{\le} 
 P_{\mu_x}(\bx_{c,i})
- \frac{\eta_x}{4}
\|\nabla P(\bx_{c,i})\|^2
-\frac{\eta_x(1-L\eta_x)}{2}
\|\bg^x_{c,i}\|^2 
+6\eta_x\kappa L_f
\Delta^y_{c,i} + 9\eta_xL^2_fv^2_1v^2_2\|\mhE_{z,i}\|^2
 +3\eta_x\|\bs^x_{c,i}\|^2
 +2\eta_x\mu^2_xL^2d^2_1,
\label{appendix:lemma:random_cost:step2}
\end{align}
where $(a)$ follows from Jensen's inequality, $(b)$ follows from 
Lemma \ref{appendix:lemma:randomsmoothing} and $L_f \le L$, 
and $(c)$ follows from 
\eqref{proof:useful:consensus} and
\begin{align}
\|\by^o(\bx_{c,i}) - \by_{c,i}\|^2 \le \frac{2 (P(\bx_{c,i}) -J(\bx_{c,i},\by_{c,i})) }{\nu} = \frac{2 \Delta^y_{c,i}}{\nu}.\notag
\end{align}
\end{proof}
\begin{Lemma}[\textbf{Duality gap}]
\label{lemma:optimality_gap}
Under Assumptions \ref{main:assumption:costfunction}---\ref{main:assumption:Lipschitz}, choosing step sizes
$\eta_x \le \min\{\frac{\eta_y}{16\kappa^2}, \frac{1}{32L}\}, \eta_y \le \min\{\frac{1}{\nu}, \frac{1}{2L_f}\}$, 
we can bound the optimality gap $\Delta^y_{c,i}$  as follows
\begin{align}
&\frac{1}{T}\sum_{i=0}^{T-1}\Delta^y_{c,i} \le 
\frac{3}{T\nu\eta_y}
\Delta^y_{c,0}
+
\frac{\eta_x}{4\nu\eta_yT} 
\sum_{i=0}^{T-1}
\|\bg^x_{c,i}\|^2-\frac{\eta_y}{4T}\sum_{i=1}^{T}
\sum_{j=0}^{i-1}
\Big(
1- \frac{\nu\eta_y}{2}\Big)^{i-j-1}
\|\bg^y_{c,j}\|^2+
\frac{4\kappa L_fv^2_1v^2_2}{T}
\sum_{i=0}^{T-1}
\|\mhE_{z,i}\|^2 \notag \\
&\quad 
+\frac{8}{T\nu}
\sum_{i=0}^{T-1}
\|\bs^y_{c,i}\|^2 +2\kappa L_fd^2_2\mu^2_y. 
\end{align}
\end{Lemma}
\begin{proof}
The proof follows from an argument similar to that used in  \cite[Lemma 5]{cai2025dama2}.
\end{proof}
\subsection{Proof of Theorem \ref{main:theorem2}}
\label{appendix:subsection:theorem2}
\begin{Theorem}[Restatement of Theorem \ref{main:theorem2}]
\label{appendix:restatement:theorem2}
Under Assumptions
\ref{main:assumption:costfunction}---\ref{main:asssumption:RD-ZO}, we consider \textbf{ZOMA} with  $\bu^j_{k,i} \not= 0$  during $\bpi_i=0$ and choose sufficiently small hyperparameters $\eta_x,\eta_y, \beta, p, \mu_x, \mu_y, \delta_x, \delta_y$ satisfying
\begin{align}
\eta_x &\le \min \Big\{
\frac{\sqrt{A_1}}{130\sqrt{d}\kappa L_f}, \frac{1}{32L},\frac{\eta_y}{16\kappa^2} 
\Big\}, \quad \beta \le \min \Big\{ 
\frac{\sqrt{A_1}}{200\kappa \sigma_1},
\frac{\sqrt{A_1}}{90\sqrt{d}L_f}, \eta_y, \frac{\sqrt{d}}{2\sigma_1}
\Big\}, p+\beta
 \le 1, \beta + bp \le b, b\bar{\beta} \le \frac{55d}{K}, \notag \\
\eta_y &\le \min \Bigg\{ \frac{\sqrt{A_1}}{90\sqrt{d}L_f}, \frac{(1-\rho)\underline{\lambda_b}}{15\sqrt{d}A_3}, 
\frac{b^{\frac{1}{2}}(1-\rho)\underline{\lambda_b}}{7\sigma_1 A_3},  \frac{(1-\rho)\underline{\lambda_b}}{40\sqrt{d}A_3}, 
\frac{(1-\rho)^{\frac{1}{2}}\underline{\lambda^{\frac{1}{2}}_b}(A_1)^{\frac{1}{4}}}{57
 (d\kappa A_2)^\frac{1}{2} L_f} , \frac{1}{\nu}, \frac{1}{2L_f},\frac{(1-\rho)^{\frac{1}{2}}\underline{\lambda^{\frac{1}{2}}_b}(bA_1)^{\frac{1}{4}}}{70 d^{\frac{1}{4}}(\kappa A_3 \sigma_1)^{\frac{1}{2}}}, 
  \frac{(1-\rho)^{\frac{2}{3}}
\underline{\lambda^{\frac{2}{3}}_b}(A_1)^{\frac{1}{3}}}{138L_f \kappa^{\frac{1}{3}}  (d A_2)^{\frac{2}{3}}}\Bigg\}, \notag \\
 &\quad \delta_x =\mu_x \le \frac{\sigma}{L_f\sqrt{d_1}}, \delta_y = \mu_y \le \frac{\sigma}{L_f\sqrt{d_2}}, \bar{\beta} \le \frac{\nu \eta_y}{2},
\end{align}
where
\begin{align}
A_1 \triangleq bK\bar{\beta}, A_2 \triangleq v_1v_2\lambda_a, A_3 \triangleq L_fv_1v_2\lambda_a
\end{align}
We obtain the following bound
\begin{align}
& \frac{1}{T}
\sum_{i=0}^{T-1}
(\mE\|\nabla_x J(\bx_{c,i},\by_{c,i})\|^2 + \mE\|\nabla_y J(\bx_{c,i},\by_{c,i})\|^2)  \notag \\
&\le\mathcal{O}
\Bigg(
\underbrace{\Pi_0}_{\text{initial gap}} + \underbrace{\frac{d\kappa^2\lambda^2_a\eta^2_y(p+\beta^2)\sigma^2}{bb_0K\bar{\beta}^2T(1-\rho)^2\underline{\lambda^2_b}}}_{\text{network noise}} + \underbrace{\Pi_1 \sigma^2 \mathbb{I}(B<N)}_{\text{large-batch effect}}
+ \underbrace{\frac{\kappa^2\sigma^2}{b_0K\bar{\beta}T}}_{\text{initial noise}}
+ \underbrace{\Pi_2 \sigma^2_0}_{\text{ZO momentum error}} 
+ \underbrace{\Pi_3}_{\text{ZO bias}}
\Bigg),
\end{align}
where
\begin{align}
\Pi_0 &\triangleq 
\frac{\mE\Delta_{c,0}}{T\eta_x} 
+ 
\frac{\kappa^2\mE\Delta^y_{c,0}}{T\eta_y}
+\frac{d\kappa^2\eta^2_y\zeta^2_0}{b\bar{\beta}KT(1-\rho)\underline{\lambda^2_b}}, \notag \\
\Pi_1 &\triangleq \frac{d\kappa^2\lambda^2_a\eta^2_yp}{b\bar{\beta}K(1-\rho)^2\underline{\lambda^2_bB}}+ \frac{\kappa^2p}{BK\bar{\beta}},\notag\\
\Pi_2 &\triangleq 
\frac{d\kappa^2\lambda^2_a\eta^2_y\beta^2}{b^2\bar{\beta}K(1-\rho)^2\underline{\lambda^2_b}} + \frac{\kappa^2\beta^2}{Kb\bar{\beta}},\notag\\
\Pi_3 &\triangleq\frac{d\kappa^2\lambda^2_a\eta^2_y}{b\bar{\beta}K(1-\rho)^2\underline{\lambda^2_b}} (d^2_1\mu^2_x+d^2_2\mu^2_y)
+\frac{\kappa^2(d^2_1\mu^2_x+d^2_2\mu^2_y)}{bK\bar{\beta}}\notag\\
&\quad+\kappa^2d^2_2\mu^2_y+\mu^2_xd^2_1.
\end{align}
\end{Theorem}
\begin{proof}
From Lemma \ref{apendix:lemma:RD:cost}, we can derive 
\begin{align}
&\frac{1}{T}
\sum_{i=0}^{T-1}
(\mE\|\nabla P(\bx_{c,i})\|^2 + \kappa L_f\mE\Delta^y_{c,i}) \notag\\ 
&\le \sum_{i=0}^{T-1}
\frac{4\mE(P_{\mu_x}(\bx_{c,i}) - P_{\mu_x}(\bx_{c,i+1}))}{T\eta_x}
-\frac{2(1-L\eta_x)}{T}\sum_{i=0}^{T-1}
\mE\|\bg^x_{c,i}\|^2
+ \frac{25\kappa L_f}{T}
\sum_{i=0}^{T-1}\mE\Delta^y_{c,i}
+\frac{36L^2_fv^2_1v^2_2}{T}\sum_{i=0}^{T-1}\mE\|\mhE_{z,i}\|^2
\notag\\
&\quad + \frac{12}{T}
\sum_{i=0}^{T-1}\mE\|\bs^x_{c,i}\|^2
+8\mu^2_xL^2d^2_1 \notag \\
&\overset{(a)}{\le} \frac{4(\mE P_{\mu_x}(\bx_{c,0}) -P^\star)}{T\eta_x} - \frac{1}{T}\sum_{i=0}^{T-1}
\mE\|\bg^x_{c,i}\|^2
+ \frac{25\kappa L_f}{T}
 \sum_{i=0}^{T-1}\mE\Delta^y_{c,i}+\frac{36L^2_fv^2_1v^2_2}{T}\sum_{i=0}^{T-1}\mE\|\mhE_{z,i}\|^2
+ \frac{12}{T}
\sum_{i=0}^{T-1}\mE\|\bs^x_{c,i}\|^2
+8\mu^2_xL^2d^2_1,
\end{align}
where $(a)$
follows by choosing $\eta_x \le \frac{1}{2L}$, telescoping the first term and 
establishing the relation 
\begin{align}
\inf_x P_{\mu_x}(x) &= \inf_x \frac{1}{{\rm Vol}(\mathbb{B}^{d_1})}
\int_{\mathbb{B}^{d_1}}
P(x +\mu_x \bu) d\bu \ge 
\frac{1}{{\rm Vol}(\mathbb{B}^{d_1})}
\int_{\mathbb{B}^{d_1}}
 \inf_x P(x +\mu_x \bu) d\bu = P^\star.
\end{align}
Denoting $\mE\Delta_{c,0} \triangleq \mE P_{\mu_x}(\bx_{c,0}) -P^\star$ and invoking Lemma \ref{lemma:optimality_gap}, 
we have 
\begin{align}
&\frac{1}{T}
\sum_{i=0}^{T-1}
(\mE\|\nabla P(\bx_{c,i})\|^2 + \kappa L_f\mE\Delta^y_{c,i}) \notag\\ 
&\le \frac{4\mE\Delta_{c,0}}{T\eta_x} - \frac{1}{T}\sum_{i=0}^{T-1}
\mE\|\bg^x_{c,i}\|^2
+ 
\Big(
\frac{75\kappa^2}{T\eta_y}
\mE\Delta^y_{c,0}
+
\frac{25\kappa^2\eta_x}{4\eta_yT} 
\sum_{i=0}^{T-1} \mE\|\bg^x_{c,i}\|^2
-\frac{25\kappa L_f\eta_y}{4T}\sum_{i=1}^{T}
\sum_{j=0}^{i-1}\Big(
1-  \frac{\nu\eta_y}{2}\Big)^{i-j-1}\mE\|\bg^y_{c,j}\|^2\notag \\
&\quad 
+
\frac{100\kappa^2 L^2_fv^2_1v^2_2}{T}
\sum_{i=0}^{T-1}
\mE\|\mhE_{z,i}\|^2
+\frac{200\kappa^2}{T}
\sum_{i=0}^{T-1}
\mE\|\bs^y_{c,i}\|^2
+50\kappa^2L^2_fd^2_2\mu^2_y
\Big)
+\frac{36L^2_fv^2_1v^2_2}{T}\sum_{i=0}^{T-1}\mE\|\mhE_{z,i}\|^2\notag\\
&\quad + \frac{12}{T}
\sum_{i=0}^{T-1}\mE\|\bs^x_{c,i}\|^2
+8\mu^2_xL^2d^2_1 \notag 
\\
&\overset{(a)}{\le} \frac{4\mE\Delta_{c,0}}{T\eta_x} - \frac{1}{2T}\sum_{i=0}^{T-1}
\mE\|\bg^x_{c,i}\|^2
+ 
\Big(
\frac{75\kappa^2}{T\eta_y}
\mE\Delta^y_{c,0}
-\frac{25\kappa L_f\eta_y}{4T}\sum_{i=1}^{T}
\sum_{j=0}^{i-1}\Big(
1-  \frac{\nu\eta_y}{2}\Big)^{i-j-1} \mE\|\bg^y_{c,j}\|^2 +
\frac{136\kappa^2 L^2_fv^2_1v^2_2}{T}
\sum_{i=0}^{T-1}
\mE\|\mhE_{z,i}\|^2
\Big)\notag \\
&\quad 
+\frac{200\kappa^2}{T}
\sum_{i=0}^{T-1}
\mE\|\bs^z_{c,i}\|^2
+ C_5\notag 
\\
&\overset{(b)}{\le} 
\frac{4\mE\Delta_{c,0}}{T\eta_x} - \frac{1}{2T}\sum_{i=0}^{T-1}
\mE\|\bg^x_{c,i}\|^2
+ 
\Big(
\frac{75\kappa^2}{T\eta_y}
\mE\Delta^y_{c,0}-\frac{25\kappa L_f\eta_y}{4T} \sum_{i=1}^{T}
\sum_{j=0}^{i-1}\Big(
1-  \frac{\nu\eta_y}{2}\Big)^{i-j-1} \mE\|\bg^y_{c,j}\|^2+
\frac{136\kappa^2 L^2_fv^2_1v^2_2}{T}
\sum_{i=0}^{T-1}
\mE\|\mhE_{z,i}\|^2
\Big)
\notag\\
&\quad + 200\kappa^2
\Bigg(
\frac{2\bs_{c,0}}{T\bar{\beta}} + \frac{38dL^2_fv^2_1v^2_2}{b\bar{\beta}KT}\sum_{i=1}^T
\mE\|\mhE_{z,i}\|^2
 +\frac{18dL^2_fv^2_1v^2_2\mhE_0}{b\bar{\beta}KT}+ \frac{20dL^2_f}{b\bar{\beta}KT}
 \sum_{i=0}^{T-1}
\eta^2_x\mE\|\bg^x_{c,i}\|^2
\notag \\
 &\quad 
 + 
\frac{20dL^2_f \eta^2_y}{b\bar{\beta}KT}
 \sum_{i=0}^{T-1} (1-(1-\bar{\beta})^{T-i})
\mE\|\bg^y_{c,i}\|^2 
 + \frac{8\beta^2\sigma^2_1}{b\bar{\beta}KT}
\sum_{i=0}^{T-1}
 (\mE\|\nabla^x_{c,i}\|^2 + \mE\|\nabla^y_{c,i}\|^2)  
+\frac{C_1+C_2}{K\bar{\beta}}
\Bigg)
+C_5 \notag \\
&\overset{(c)}{\le} 
\frac{4\mE\Delta_{c,0}}{T\eta_x} - \frac{1}{4T}\sum_{i=0}^{T-1}
\mE\|\bg^x_{c,i}\|^2
+ 
\Big(
\frac{75\kappa^2}{T\eta_y}
\mE\Delta^y_{c,0}
-\frac{25\kappa L_f\eta_y}{4T}\sum_{i=1}^{T}
\sum_{j=0}^{i-1}\Big(
1-  \frac{\nu\eta_y}{2}\Big)^{i-j-1} \mE\|\bg^y_{c,j}\|^2\Big)\notag \\
&\quad 
+ \frac{200\times 38 d \kappa^2 L^2_fv^2_1v^2_2}{b\bar{\beta}KT} \Big(
\sum_{i=1}^T 
\mE\|\mhE_{z,i}\|^2 + \sum_{i=0}^{T-1}
\mE\|\mhE_{z,i}\|^2\Big)
+ 200\kappa^2
\Bigg(
\frac{2\bs_{c,0}}{T\bar{\beta}} +\frac{18dL^2_fv^2_1v^2_2\mhE_0}{b\bar{\beta}KT} + 
\frac{20dL^2_f \eta^2_y}{b\bar{\beta}KT}
\notag \\
 &\quad \times
 \sum_{i=0}^{T-1} (1-(1-\bar{\beta})^{T-i})
\mE\|\bg^y_{c,i}\|^2+ \frac{8\beta^2\sigma^2_1}{b\bar{\beta}KT}
\sum_{i=0}^{T-1}
 (\mE\|\nabla^x_{c,i}\|^2 + \mE\|\nabla^y_{c,i}\|^2) 
+\frac{C_1+C_2}{K\bar{\beta}}
\Bigg)
+C_5,
\end{align}
where $(a)$ follows by choosing $\frac{25\kappa^2\eta_x}{4\eta_y} \le \frac{1}{2} \Longrightarrow \frac{\eta_x}{\eta_y} \le \frac{2}{25\kappa^2}$
and noting $36 < 36\kappa^2, 12 < 200\kappa^2$ and denoting $C_5 \triangleq 50\kappa^2L^2_fd^2_2\mu^2_y+8\mu^2_xL^2d^2_1$, $(b)$
follows from Lemma \ref{appendix:lemma:coorinate_variance_nxny},
$(c)$ follows by choosing 
\begin{align}
& \frac{136\kappa^2L^2_fv^2_1v^2_2}{T} \le    200\kappa^2 \times \frac{38dL^2_fv^2_1v^2_2}{b\bar{\beta}KT} \Longrightarrow 
 b\bar{\beta} \le \frac{55d}{K} ,\\  &\frac{200\times20 d \kappa^2L^2_f \eta^2_x}{b\bar{\beta}KT}\le \frac{1}{4T}\Longrightarrow \eta_x \le 
 \frac{\sqrt{bK\bar{\beta}}}{130\sqrt{d}\kappa L_f}.
\end{align}
Moreover, we have 
\begin{align}
&\frac{\eta_y\kappa L_f}{T}
\sum_{i=1}^{T}\sum_{j=0}^{i-1}\Big(1 - \frac{\nu \eta_y}{2}\Big)^{i-j-1}
\mE\|\bg^y_{c,j}\|^2 
\notag\\
&=
\frac{2\kappa^2}{T}\sum_{i=0}^{T-1} \Big(1- \Big(1-\frac{\nu\eta_y}{2}\Big)^{T-i}\Big)
\mE\|\bg^y_{c,i}\|^2.
\label{proof:appendix:eliminate_gyc2}
\end{align}
We further choose
\begin{align}
&200 \kappa^2 \times \frac{20dL^2_f\eta^2_y}{b\bar{\beta}K} \le 2 \kappa^2 \Longrightarrow \eta_y \le  \frac{\sqrt{bK\bar{\beta}}}{90\sqrt{d}L_f}.\\
&(1-(1-\bar{\beta})^{T-i}) \le \Big(1- \Big(1-\frac{\nu\eta_y}{2}\Big)^{T-i}\Big)  \Longrightarrow
\bar{\beta} \le \frac{\nu \eta_y}{2}, \bar{\beta} \le 1, \eta_y \le \frac{2}{\nu}.
\end{align}
We can cancel part of the term on $\mE\|\bg^y_{c,i}\|$, leading to
\begin{align}
&\frac{1}{T}
\sum_{i=0}^{T-1}
(\mE\|\nabla P(\bx_{c,i})\|^2 + \kappa L_f\mE\Delta^y_{c,i}) \notag\\ 
&\le  
\frac{4\mE\Delta_{c,0}}{T\eta_x} - \frac{1}{4T}\sum_{i=0}^{T-1}
\mE\|\bg^x_{c,i}\|^2
+ 
\Big(
\frac{75\kappa^2}{T\eta_y}
\mE\Delta^y_{c,0}
-\frac{21\kappa L_f\eta_y}{4T}\sum_{i=1}^{T}
\sum_{j=0}^{i-1}
\Big(
1-  \frac{\nu\eta_y}{2}\Big)^{i-j-1} \mE\|\bg^y_{c,j}\|^2\Big)+ \frac{200\times 38 d \kappa^2 L^2_fv^2_1v^2_2}{b\bar{\beta}KT} \notag \\
&\quad 
\times\Big(
\sum_{i=1}^T
\mE\|\mhE_{z,i}\|^2 + \sum_{i=0}^{T-1}
\mE\|\mhE_{z,i}\|^2\Big)+ 200\kappa^2
\Bigg(
\frac{2\bs_{c,0}}{T\bar{\beta}} 
 +\frac{18dL^2_fv^2_1v^2_2\mhE_0}{b\bar{\beta}KT}
+ \frac{8\beta^2\sigma^2_1}{b\bar{\beta}KT}
\sum_{i=0}^{T-1}
 (\mE\|\nabla^x_{c,i}\|^2  + \mE\|\nabla^y_{c,i}\|^2) + \frac{C_1+C_2}{K\bar{\beta}}
\Bigg)
\notag 
\\
&\quad 
+C_5.
\end{align}
Invoking Lemma 
\ref{appendix:lemma:random:exey} and using inequality
\eqref{proof:useful:gradientnorm}, we have 
\begin{align}
& \Big(1 
-  \frac{9600\kappa^2\beta^2\sigma^2_1}{b\bar{\beta}K} \Big)\frac{1}{T}
\sum_{i=0}^{T-1}
(\mE\|\nabla P(\bx_{c,i})\|^2 + \kappa L_f\mE\Delta^y_{c,i})  \notag \\
&\le 
\frac{4\mE\Delta_{c,0}}{T\eta_x} - \frac{1}{4T}\sum_{i=0}^{T-1}
\mE\|\bg^x_{c,i}\|^2
+ 
\Big(
\frac{75\kappa^2}{T\eta_y}
\mE\Delta^y_{c,0}
-\frac{21\kappa L_f\eta_y}{4T}\sum_{i=1}^{T}
\sum_{j=0}^{i-1}\Big(
1-  \frac{\nu\eta_y}{2}\Big)^{i-j-1} \mE\|\bg^y_{c,j}\|^2\Big)\notag \\
&\quad 
+ \frac{400\times 38 d \kappa^2 L^2_fv^2_1v^2_2}{b\bar{\beta}K} \Bigg(
\frac{5\mhE_0}{T(1-\rho)}
+ I_0+L_0\sum_{i=0}^{T-1}
(\eta^2_x\mE\|\bg^x_{c,i}\|^2
+\eta^2_y\mE\|\bg^y_{c,i}\|^2)
+L_1\sum_{i=0}^{T-1}(6\kappa L_f\mE\Delta^y_{c,i} 
\notag\\
&\quad +  6\mE\|\nabla P(\bx_{c,i})\|^2) 
 +C_4 + \frac{4\eta^2_y\lambda^2_a( 5C_1+C_3)}{(1-\rho)^2\underline{\lambda^2_b}}\Bigg)
+ 200\kappa^2
\Bigg(
\frac{2\bs_{c,0}}{T\bar{\beta}} +\frac{18dL^2_fv^2_1v^2_2\mhE_0}{b\bar{\beta}KT}
+\frac{C_1+C_2}{K\bar{\beta}} \Bigg)+C_5.
\end{align}
Recall the definition of $L_0, L_1$ in Lemma \ref{appendix:lemma:random:exey}, we choose
\begin{align}
&\frac{400\times 38 d\kappa^2 L^2_f v^2_1v^2_2}{b\bar{\beta} K}
\times \frac{508dL^2_f\lambda^2_a\eta^2_y\eta^2_x}{(1-\rho)^2\underline{\lambda^2_b}} \le \frac{1}{4} \Longrightarrow
\eta_y \le \frac{(1-\rho)^{\frac{1}{2}}\underline{\lambda^{\frac{1}{2}}_b}(b\bar{\beta}K)^{\frac{1}{4}}}{75
 (d\kappa v_1v_2 \lambda_a)^\frac{1}{2} L_f
}, \\
& \frac{9600\kappa^2\beta^2\sigma^2_1}{bK\bar{\beta}} \le \frac{1}{4}  \Longrightarrow 
\beta \le \frac{\sqrt{b\bar{\beta}K}}{200\kappa \sigma_1}, \\
&\frac{400\times 38 d\kappa^2 L^2_f v^2_1v^2_2}{b\bar{\beta} K}
\times\frac{184\lambda^2_a\sigma^2_1\eta^2_y\beta^2}{(1-\rho)^2\underline{\lambda^2_b}b} \times 6\le \frac{1}{4}
\Longrightarrow 
\eta_y \le \frac{(1-\rho)^{\frac{1}{2}}\underline{\lambda^{\frac{1}{2}}_b}(b^2\bar{\beta}K)^{\frac{1}{4}}}{91d^{\frac{1}{4}}(\kappa L_f v_1 v_2 \lambda_a \sigma_1)^{\frac{1}{2}}}, \beta \le \eta_y .
\end{align}
We can get 
\begin{align}
& \frac{1}{2T}
\sum_{i=0}^{T-1}
(\mE\|\nabla P(\bx_{c,i})\|^2 + \kappa L_f\mE\Delta^y_{c,i})  \notag \\
&\le 
\frac{4\mE\Delta_{c,0}}{T\eta_x} 
+ 
\Big(
\frac{75\kappa^2}{T\eta_y}
\mE\Delta^y_{c,0}
-\frac{21\kappa L_f\eta_y}{4T}\sum_{i=1}^{T}
\sum_{j=0}^{i-1} \Big(
1-  \frac{\nu\eta_y}{2}\Big)^{i-j-1} \mE\|\bg^y_{c,j}\|^2\Big)
+ \frac{400\times 38 d \kappa^2 L^2_fv^2_1v^2_2}{b\bar{\beta}K} \notag\\
&\quad \times \Bigg(
\frac{5\mhE_0}{T(1-\rho)}
+ I_0+ L_0\eta^2_y\sum_{i=0}^{T-1}\mE\|\bg^y_{c,i}\|^2 +C_4
 + \frac{4\eta^2_y\lambda^2_a( 5C_1+C_3)}{(1-\rho)^2\underline{\lambda^2_b}}\Bigg)
+ 200\kappa^2
\Bigg(
\frac{2\bs_{c,0}}{T\bar{\beta}} +\frac{18dL^2_fv^2_1v^2_2\mhE_0}{b\bar{\beta}KT}+\frac{C_1+C_2}{K\bar{\beta}}\Big) \notag\\
&\quad 
+C_5.
\end{align}
Note that
\begin{align}
&\frac{\eta_y\kappa L_f}{T}
\sum_{i=1}^{T}\sum_{j=0}^{i-1}\Big(1 - \frac{\nu \eta_y}{2}\Big)^{i-j-1} \mE\|\bg^y_{c,i}\|^2 \notag\\
&= \frac{2\kappa^2 }{T}\sum_{i=0}^{T-1}\Big(
1-(1-\frac{\nu\eta_y}{2})^{T-i}
\Big)\mE\|\bg^y_{c,i}\|^2 
\notag\\
&\ge 
 \frac{2\kappa^2 }{T}\sum_{i=0}^{T-1}\Big(
1-(1-\frac{\nu\eta_y}{2})
\Big)\mE\|\bg^y_{c,i}\|^2 \quad (\eta_y \le \frac{1}{\nu}) \notag\\
&\ge \frac{\kappa^2 \nu \eta_y}{T}\sum_{i=0}^{T-1}\mE\|\bg^y_{c,i}\|^2 .
\end{align}
We can choose
\begin{align}
&\frac{400\times 38 d \kappa^2 L^2_fv^2_1v^2_2}{b\bar{\beta}K}  \times 
\frac{508dL^2_f\lambda^2_a\eta^4_y}{(1-\rho)^2\underline{\lambda^2_b}T}
\le \frac{\kappa^2\nu \eta_y }{T}
\Longrightarrow 
\eta_y  \le \frac{(1-\rho)^{\frac{2}{3}}
\underline{\lambda^{\frac{2}{3}}_b}(bK\bar{\beta})^{\frac{1}{3}}}{198 L_f  \kappa^{\frac{1}{3}} (d v_1v_2 \lambda_a)^{\frac{2}{3}}}.
\end{align}
We can conclude
\begin{align}
& \frac{1}{T}
\sum_{i=0}^{T-1}
(\mE\|\nabla P(\bx_{c,i})\|^2 + \kappa L_f\mE\Delta^y_{c,i})  \notag \\
&\le 
\frac{8 \mE\Delta_{c,0}}{T\eta_x} 
+ 
\frac{150\kappa^2\mE\Delta^y_{c,0}}{T\eta_y}
 + \frac{400\times 76 d \kappa^2 L^2_fv^2_1v^2_2}{b\bar{\beta}K} \Bigg(
\frac{5\mhE_0}{T(1-\rho)}
+ I_0 +C_4  +\frac{4\eta^2_y\lambda^2_a( 5C_1+C_3)}{(1-\rho)^2\underline{\lambda^2_b}}\Bigg) + 400\kappa^2
\Bigg(
\frac{2\bs_{c,0}}{T\bar{\beta}} \notag\\
&\quad 
+\frac{18dL^2_fv^2_1v^2_2\mhE_0}{b\bar{\beta}KT}
+\frac{C_1+C_2}{K\bar{\beta}}
\Bigg)
+2C_5
\notag \\
&\overset{(a)}{\le}\mathcal{O}
\Bigg(
\frac{\mE\Delta_{c,0}}{T\eta_x} 
+ 
\frac{\kappa^2\mE\Delta^y_{c,0}}{T\eta_y}
+\frac{d\kappa^2\eta^2_y\zeta^2_0}{b\bar{\beta}KT(1-\rho)\underline{\lambda^2_b}}
+ \frac{d\kappa^2\lambda^2_a\eta^2_y(p+\beta^2)\sigma^2}{bb_0K\bar{\beta}^2T(1-\rho)^2\underline{\lambda^2_b}}
+\frac{d^2\kappa^2\lambda^2_a\eta^4_y(p+\beta^2)\zeta^2_0}{b^2K\bar{\beta}^2T(1-\rho)^2\underline{\lambda^4_b}} \notag \\
&\quad 
+\frac{d\kappa^2\lambda^2_a\eta^2_y}{b\bar{\beta}K(1-\rho)^2\underline{\lambda^2_b}} \Big(
\frac{pC_0}{B}\mathbb{I}(B<N)+ \frac{\beta^2\sigma^2_0}{b} +C_1 +C_3
\Big) 
+\frac{\kappa^2\sigma^2}{b_0K\bar{\beta}T}
+\frac{d\kappa^2\eta^2_y\zeta^2_0}{b\bar{\beta}KT\underline{\lambda^2_b}}
+\frac{\kappa^2(C_1+C_2)}{K\bar{\beta}}
+C_5\Bigg) \notag\\
&\overset{(b)}{\le}\mathcal{O}
\Bigg(
\frac{\mE\Delta_{c,0}}{T\eta_x} 
+ 
\frac{\kappa^2\mE\Delta^y_{c,0}}{T\eta_y}
+\frac{d\kappa^2\eta^2_y\zeta^2_0}{b\bar{\beta}KT(1-\rho)\underline{\lambda^2_b}}
+ \frac{d\kappa^2\lambda^2_a\eta^2_y(p+\beta^2)\sigma^2}{bb_0K\bar{\beta}^2T(1-\rho)^2\underline{\lambda^2_b}}
+\frac{d^2\kappa^2\lambda^2_a\eta^4_y(p+\beta^2)\zeta^2_0}{b^2K\bar{\beta}^2T(1-\rho)^2\underline{\lambda^4_b}} \notag \\
&\quad 
+\frac{d\kappa^2\lambda^2_a\eta^2_y}{b\bar{\beta}K(1-\rho)^2\underline{\lambda^2_b}} \Big(
\frac{p\sigma^2}{B}\mathbb{I}(B<N)+ \frac{\beta^2\sigma^2_0}{b} +d^2_1\mu^2_x+d^2_2\mu^2_y
\Big) 
 +\frac{\kappa^2\sigma^2}{b_0K\bar{\beta}T}
+\frac{d\kappa^2\eta^2_y\zeta^2_0}{b\bar{\beta}KT\underline{\lambda^2_b}}
+\frac{\kappa^2(d^2_1\mu^2_x+d^2_2\mu^2_y)}{bK\bar{\beta}}\notag\\
&\quad 
+\frac{\kappa^2p\sigma^2}{BK\bar{\beta}}\mathbb{I}(B<N)+\frac{\kappa^2\beta^2\sigma^2_0}{Kb\bar{\beta}}
+\kappa^2d^2_2\mu^2_y+\kappa^2d^2_1\mu^2_x\Bigg),
\end{align}
where in $(a)$ we suppress the constants $v^2_1,v^2_2,L^2_f$ and use the relations \eqref{appendix:useful_Sxy}, \eqref{appendix:useful_Exy} and definition of $I_0, C_4$, $(b)$ is because we 
choose sufficiently small $p$ and $\delta_x, \mu_x, \delta_y, \mu_y$ with $\mu_x =\delta_x, \mu_y =\delta_y$ and $b \ge 1$ such that
\begin{align}
C_1 &= \mathcal{O}\Big(
\frac{d^2_1\mu^2_x+d^2_2\mu^2_y}{b}
\Big), \quad 
C_2 = \mathcal{O}\Big( \frac{pC_0}{B}\mathbb{I}(B<N) + \frac{\beta^2\sigma^2_0}{b}\Big), \notag\\
C_3 &= \mathcal{O}(d^2_1\mu^2_x+d^2_2\mu^2_y), \quad C_0 = \mathcal{O}(\sigma^2), \quad
C_5 = \mathcal{O}(\kappa^2d^2_2\mu^2_y+ \kappa^2d^2_1\mu^2_x).
\end{align}
We next derive a compact form of the performance bound by absorbing several minor terms. Specifically, choosing $\eta_y$ sufficiently small, we have
\begin{align}
&\mathcal{O}\Big(
\frac{d\kappa^2\lambda^2_a\eta^2_y(p+\beta^2)\sigma^2}{bb_0K\bar{\beta}^2T(1-\rho)^2\underline{\lambda^2_b}}
+\frac{d^2\kappa^2\lambda^2_a\eta^2_y(p+\beta^2)\zeta^2_0}{b^2K\bar{\beta}^2T(1-\rho)^2\underline{\lambda^2_b}} \times \underbrace{\frac{\eta^2_y}{\underline{\lambda^2_b}}}_{\text{small}}\Big) = \mathcal{O}\Big(
\frac{d\kappa^2\lambda^2_a\eta^2_y(p+\beta^2)\sigma^2}{bb_0K\bar{\beta}^2T(1-\rho)^2\underline{\lambda^2_b}} 
\Big).
\end{align}
Invoking Lemma \ref{appendix:lemma:transformed_recursion}, we have $(1-\rho)\underline{\lambda^2_b} < 1$ and
\begin{align}
\mathcal{O} \Big(\frac{d\kappa^2\eta^2_y\zeta^2_0}{b\bar{\beta}KT(1-\rho)\underline{\lambda^2_b}}+\frac{d\kappa^2\eta^2_y\zeta^2_0}{b\bar{\beta}KT\underline{\lambda^2_b}}
\Big) &= \mathcal{O} \Big(\frac{d\kappa^2\eta^2_y\zeta^2_0}{b\bar{\beta}KT(1-\rho)\underline{\lambda^2_b}}
\Big).
\end{align}
Note that we choose $\eta_y, \beta$ sufficiently small in stochastic setting and choose $\beta = 0$ in the finite-sum setting, we can write 
\begin{align}
\underbrace{\frac{\eta^2_y}{b}}_{\text{small}}\times\frac{d\lambda^2_a}{(1-\rho)^2\underline{\lambda^2_b}} \times \frac{\kappa^2\beta^2}{b\bar{\beta}K}+ \frac{\kappa^2\beta^2}{Kb\bar{\beta}}= \mathcal{O}\Big(\frac{\kappa^2\beta^2}{Kb\bar{\beta}}\Big)
\end{align}
Furthermore, 
we introduce 
\begin{align}
\Pi_0 &\triangleq 
\frac{\mE\Delta_{c,0}}{T\eta_x} 
+ 
\frac{\kappa^2\mE\Delta^y_{c,0}}{T\eta_y}
+\frac{d\kappa^2\eta^2_y\zeta^2_0}{b\bar{\beta}KT(1-\rho)\underline{\lambda^2_b}}, \notag \\
\Pi_1 &\triangleq \frac{d\kappa^2\lambda^2_a\eta^2_yp}{b\bar{\beta}K(1-\rho)^2\underline{\lambda^2_b}B}+ \frac{\kappa^2p}{BK\bar{\beta}},\quad
\Pi_2 \triangleq 
 \frac{\kappa^2\beta^2}{Kb\bar{\beta}},\notag\\
\Pi_3 &\triangleq\frac{d\kappa^2\lambda^2_a\eta^2_y}{b\bar{\beta}K(1-\rho)^2\underline{\lambda^2_b}} (d^2_1\mu^2_x+d^2_2\mu^2_y)
+\frac{\kappa^2(d^2_1\mu^2_x+d^2_2\mu^2_y)}{bK\bar{\beta}}+\kappa^2d^2_2\mu^2_y+d^2_1\mu^2_x.
\end{align}
 We can conclude that
\begin{align}
& \frac{1}{T}
\sum_{i=0}^{T-1}
(\mE\|\nabla P(\bx_{c,i})\|^2 + \kappa L_f\mE\Delta^y_{c,i})  \notag \\
&\le\mathcal{O}
\Bigg(
\underbrace{\Pi_0}_{\text{initial gap}} + \underbrace{\frac{d\kappa^2\lambda^2_a\eta^2_y(p+\beta^2)\sigma^2}{bb_0K\bar{\beta}^2T(1-\rho)^2\underline{\lambda^2_b}}}_{\text{network noise}} + \underbrace{\Pi_1 \sigma^2 \mathbb{I}(B<N)}_{\text{large-batch effect}}
+ \underbrace{\frac{\kappa^2\sigma^2}{b_0K\bar{\beta}T}}_{\text{initial noise}}
+ \underbrace{\Pi_2 \sigma^2_0}_{\text{ZO momentum error}} 
+ \underbrace{\Pi_3}_{\text{ZO-RU bias}}
\Bigg).
\end{align}
We further denote 
\begin{align}
A_1 \triangleq bK\bar{\beta}, A_2 \triangleq v_1v_2\lambda_a, A_3 \triangleq L_fv_1v_2\lambda_a.
\end{align}
We summarize the choice of hyperparameters as follows
\begin{align}
\eta_x &\le \min \Big\{
\frac{\sqrt{A_1}}{130\sqrt{d}\kappa L_f}, \frac{1}{32L},\frac{\eta_y}{16\kappa^2} 
\Big\}, \quad \beta \le \min \Big\{ 
\frac{\sqrt{A_1}}{200\kappa \sigma_1},
\frac{\sqrt{A_1}}{90\sqrt{d}L_f}, \eta_y, \frac{\sqrt{d}}{2\sigma_1}
\Big\}, p+\beta
 \le 1, \beta + bp \le b, b\bar{\beta} \le \frac{55d}{K}, \notag \\
\eta_y &\le \min \Bigg\{ \frac{\sqrt{A_1}}{90\sqrt{d}L_f}, \frac{(1-\rho)\underline{\lambda_b}}{15\sqrt{d}A_3}, 
\frac{b^{\frac{1}{2}}(1-\rho)\underline{\lambda_b}}{7\sigma_1 A_3},  \frac{(1-\rho)\underline{\lambda_b}}{40\sqrt{d}A_3}, 
\frac{(1-\rho)^{\frac{1}{2}}\underline{\lambda^{\frac{1}{2}}_b}(A_1)^{\frac{1}{4}}}{75
 (d\kappa A_2)^\frac{1}{2} L_f} , \frac{1}{\nu}, \frac{1}{2L_f},\frac{(1-\rho)^{\frac{1}{2}}\underline{\lambda^{\frac{1}{2}}_b}(bA_1)^{\frac{1}{4}}}{91 d^{\frac{1}{4}}(\kappa A_3 \sigma_1)^{\frac{1}{2}}},
  \frac{(1-\rho)^{\frac{2}{3}}
\underline{\lambda^{\frac{2}{3}}_b}(A_1)^{\frac{1}{3}}}{198L_f \kappa^{\frac{1}{3}}  (d A_2)^{\frac{2}{3}}}\Bigg\}, \notag\\
&\quad\delta_x =\mu_x \le \frac{\sigma}{L_f\sqrt{d_1}},  \delta_y = \mu_y \le \frac{\sigma}{L_f\sqrt{d_2}}, \bar{\beta} \le \frac{\nu \eta_y}{2}
\end{align}
The proof is completed by using the relation  \eqref{appendix:useful:bound_gradient}.
\end{proof}

\subsection{Corollaries of Theorem \ref{main:theorem2}}
\label{appendix:corollary:case2}
We now specialize Theorem \ref{main:theorem2} by considering two settings: the online stochastic setting, where $N$ may be infinitely large, and the finite-sum setting, where $N<+\infty$.

$\bullet$ \textbf{Online stochastic scenario}

\begin{Corollary}[\textbf{ZO-STORM-ED}]
\label{corollary:online:storm+ed:rd}
Under Assumptions \ref{main:assumption:costfunction}-\ref{main:asssumption:RD-ZO} and the matrix condition in Lemma \ref{appendix:lemma:transformed_recursion}, 
we consider $\bu^j_{k,i} \not= 0$ during $\bpi_i=0$.
We consider the ED strategy shown in Table \ref{tab:matrix_choices}, and set hyperparameters as
\begin{align}
\eta_x &= \mathcal{O}\Bigg(
\frac{K^{\frac{2}{3}}}{\kappa^2 d^{\frac{1}{2}} T^{\frac{1}{3}}}
\Bigg), \eta_y = \mathcal{O}\Bigg(
\frac{K^{\frac{2}{3}}}{d^{\frac{1}{2}}T^{\frac{1}{3}}}
\Bigg), p = 0,  \notag\\
\quad \beta &=   \mathcal{O}\Bigg(
\frac{K^{\frac{1}{3}}}{T^{\frac{2}{3}}}
\Bigg), b= \mathcal{O}(1),
\delta_x = \mu_x=\mathcal{O}\Bigg(
\frac{K^{\frac{1}{3}}}{d^{\frac{3}{4}}_1T^{\frac{2}{3}}}
\Bigg),
\delta_y =\mu_y = \mathcal{O}\Bigg(
\frac{K^{\frac{1}{3}}}{d^{\frac{3}{4}}_2T^{\frac{2}{3}}} 
\Bigg) , b_0 = \mathcal{O}\Bigg(
\frac{T^{\frac{1}{3}}}{K^{\frac{2}{3}}}
\Bigg).
\end{align}
The above hyperparameter choices satisfy the hyperparameter conditions for {\em sufficiently} large $T$, i.e., 
\begin{align}
\eta_x & \le \min \Big\{ 
\frac{1}{32L}, \mathcal{O}\Bigg(
\frac{K^{\frac{2}{3}}}{\kappa^2d^{\frac{1}{2}}T^{\frac{1}{3}}}\Bigg), \mathcal{O}\Bigg(
\frac{K^{\frac{2}{3}}}{\kappa d^{\frac{1}{2}} T^{\frac{1}{3}}}
\Bigg)
\Bigg\},  \notag\\
\delta_x &= \mu_x\le \frac{\sigma^2}{L_f\sqrt{d_1}}, \beta \le \min\Bigg\{ \mathcal{O} \Bigg(
\frac{K^{\frac{2}{3}}}{\kappa T^{\frac{1}{3}}}
\Bigg), \mathcal{O}\Bigg(
\frac{K^{\frac{2}{3}}}{d^{\frac{1}{2}}T^{\frac{1}{3}}}\Bigg),\frac{\sqrt{d}}{2\sigma_1} \Bigg\} ,  \delta_y = \mu_y \le \frac{\sigma^2}{L_f\sqrt{d_2}}, \notag \\
 \eta_y &
\le \min\Big\{
\mathcal{O}\Bigg(
\frac{K^{\frac{2}{3}}}{\sqrt{d}T^{\frac{1}{3}}} \Bigg),\mathcal{O}
\Bigg(
\frac{(1-\lambda)\sqrt{1-\lambda}}{d^{\frac{1}{2}}}\Bigg), \mathcal{O}(1)
\Bigg\}, \bar{\beta} = \beta \le \eta_y, \bar{\beta} \le \mathcal{O}\Big(\frac{d}{K}\Big). \notag 
\end{align}
The performance bound is given by 
\begin{align}
&\frac{1}{T}\sum_{i=0}^{T-1}\Big(\mE\|\nabla_x J(\bx_{c,i}, \by_{c,i})\|^2 + \mE\|\nabla_y J(\bx_{c,i}, \by_{c,i})\|^2\Big) \notag \\
&\le \mathcal{O}
\Bigg(
\frac{\sqrt{d}\kappa^2\Delta_{0}}{(TK)^{\frac{2}{3}}}
+\frac{\kappa^2\zeta^2_0}{T(1-\lambda)^2}
+ \frac{\kappa^2K\lambda^2\sigma^2}{T^2(1-\lambda)^3} +\frac{\kappa^2(\sigma^2+\sigma^2_0)}{(TK)^{\frac{2}{3}}}
 + \frac{\sqrt{d}\kappa^2\lambda^2K^{\frac{2}{3}}}{T^{\frac{4}{3}}(1-\lambda)^3}
+ \frac{\sqrt{d}\kappa^2}{(TK)^{\frac{2}{3}}} + \frac{\sqrt{d}\kappa^2K^{\frac{2}{3}}}{T^{\frac{4}{3}}}
\Bigg),
\end{align}
where $\Delta_0  \triangleq \mE(\Delta_{c,0}+\Delta^y_{c,0})$.
The dominant communication complexity (CC) and function query complexity (FC) are given by 
\begin{align}
CC &= \mathcal{O}\Bigg(\frac{ d^{\frac{3}{4}}\kappa^3\varepsilon^{-3}}{K} + \frac{\kappa^2\varepsilon^{-2}}{(1-\lambda)^2} \Bigg),  FC = 
CC \times 8b +b_0 \times 2d\approx \mathcal{O}\Bigg(\frac{d^{\frac{3}{4}}\kappa^3\varepsilon^{-3}}{K}+ \frac{\kappa^2\varepsilon^{-2}}{(1-\lambda)^2} + \frac{d^{\frac{5}{4}}\kappa\varepsilon^{-1}}{K} \Bigg).
\end{align}
Furthermore, the transient time in achieving a linear speedup in the number of agent $K$ is given by 
\begin{align}
\max \Bigg\{ 
\mathcal{O} \Bigg(\frac{K^2}{d^{\frac{3}{2}}(1-\lambda)^6} \Bigg), 
\mathcal{O} \Bigg( 
\frac{K^{\frac{5}{4}}}{d^{\frac{3}{8}}(1-\lambda)^{\frac{9}{4}}}
\Bigg), \mathcal{O} \Bigg(\frac{K^2}{(1-\lambda)^{4.5}}\Bigg)
\Bigg\},
\end{align}
where we used $\lambda < 1$.
\end{Corollary}

\begin{Corollary}[\textbf{ZO-STORM-GT}]
\label{corollary:online:atc+gt:rd}
Under Assumptions \ref{main:assumption:costfunction}-\ref{main:asssumption:RD-ZO}  and the matrix condition in \ref{appendix:lemma:transformed_recursion}, 
we consider $\bu^j_{k,i} \not= 0$ during $\bpi_i=0$.
We consider the GT strategy shown in Table \ref{tab:matrix_choices}, and then choose the same hyperparameters as Corollary \ref{corollary:online:storm+ed:rd}.
The performance bound is given by 
\begin{align}
&\frac{1}{T}\sum_{i=0}^{T-1}\Big(\mE\|\nabla_x J(\bx_{c,i}, \by_{c,i})\|^2 + \mE\|\nabla_y J(\bx_{c,i}, \by_{c,i})\|^2\Big) \notag \\
&\le \mathcal{O}
\Bigg(
\frac{ \sqrt{d}\kappa^2\Delta_0}{(TK)^{\frac{2}{3}}}
+\frac{\kappa^2\zeta^2_0}{T(1-\lambda)^3}
+ \frac{\kappa^2K\lambda^4\sigma^2}{T^2(1-\lambda)^4} + \frac{\kappa^2(\sigma^2+\sigma^2_0)}{(TK)^{\frac{2}{3}}} 
+ \frac{\sqrt{d}\kappa^2\lambda^4K^{\frac{2}{3}}}{(1-\lambda)^4T^{\frac{4}{3}}}
+ \frac{\sqrt{d}\kappa^2}{(TK)^{\frac{2}{3}}} + \frac{\sqrt{d}\kappa^2K^{\frac{2}{3}}}{T^{\frac{4}{3}}}
\Bigg)
\end{align}
The CC and FC are given by 
\begin{align}
CC &= \mathcal{O}\Bigg(\frac{ d^{\frac{3}{4}}\kappa^3\varepsilon^{-3}}{K} + \frac{\kappa^2\varepsilon^{-2}}{(1-\lambda)^3} \Bigg), \quad  FC = 
CC \times 8b +b_0\times 2d \approx \mathcal{O}\Bigg(\frac{d^{\frac{3}{4}}\kappa^3\varepsilon^{-3}}{K} + \frac{\kappa^2\varepsilon^{-2}}{(1-\lambda)^3} + \frac{d^{\frac{5}{4}}\kappa\varepsilon^{-1}}{K} \Bigg).
\end{align}
Furthermore, the transient time in achieving linear speedup in the number of agent $K$ is given by 
\begin{align}
\max \Bigg\{ 
\mathcal{O} \Bigg(\frac{K^2}{d^{\frac{3}{2}}(1-\lambda)^9} \Bigg), 
\mathcal{O} \Bigg( 
\frac{K^{\frac{5}{4}}}{d^{\frac{3}{8}}(1-\lambda)^{3}}\Bigg), \mathcal{O} \Bigg( \frac{K^2}{(1-\lambda)^{6}}\Bigg)
\Bigg\},
\end{align}
where we used $\lambda < 1$.
\end{Corollary}

$\bullet$ \textbf{Finite-sum scenario}

We next consider the finite-sum setting, focusing on the two regimes
$N\ge \tilde{\mathcal{O}}(\varepsilon^{-2})$
and $N\le \tilde{\mathcal{O}}(\varepsilon^{-2})$.
Compared with the naive choice $B =N$,
we show that the complexity bound can be further improved when $N \ge \tilde{\mathcal{O}}(\epsilon^{-2})$.

\begin{Corollary}[\textbf{ZO-PAGE-ED}]
\label{corollary:offline:page+ed:rd}
Under Assumptions \ref{main:assumption:costfunction}-\ref{main:asssumption:RD-ZO} and the matrix condition in Lemma \ref{appendix:lemma:transformed_recursion},
we consider $\bu^j_{k,i} \not= 0$ during $\bpi_i=0$.
We consider the ED strategy shown in Table \ref{tab:matrix_choices}.
We set the hyperparameters as follows
\begin{align}
\eta_x &= \mathcal{O}\Bigg(\frac{(1-\lambda)^{1.5}}{\sqrt{d}\kappa^2}\Bigg),\eta_y =  \mathcal{O}\Bigg(\frac{(1-\lambda)^{1.5}}{\sqrt{d}}\Bigg), \notag\\
b &=b_0 = \sqrt{\frac{Nd^{1-c}}{K}}, p = \frac{1}{\sqrt{NKd^{1-c}}}, \beta = 0,  B=\frac{N}{d^{c}} \notag\\
\delta_x& = \mu_x= \mathcal{O}\Big(\frac{1}{d^{\frac{3}{4}}_1 (1-\lambda)^{\frac{3}{4}}T^{\frac{1}{2}}}
\Big),  \delta_y = \mu_y = \mathcal{O}
 \Big(
\frac{1}{d^{\frac{3}{4}}_2 (1-\lambda)^{\frac{3}{4}}T^{\frac{1}{2}}}
 \Big),
\end{align}
where $c=1$ if $N \ge \mathcal{O}(d\kappa^2\varepsilon^{-2}/K)$; otherwise $c=0$ (full-batch).
We obtain the performance bound given by 
\begin{align}
&\frac{1}{T}\sum_{i=0}^{T-1}\Big(\mE\|\nabla_x J(\bx_{c,i}, \by_{c,i})\|^2 + \mE\|\nabla_y J(\bx_{c,i}, \by_{c,i})\|^2\Big) \notag \\
&\le \mathcal{O}
\Bigg(
\frac{\sqrt{d}\kappa^2\Delta_0}{T(1-\lambda)^{1.5}}
+\frac{\kappa^2(1-\lambda)\zeta^2_0}{T}
+\frac{\kappa^2\lambda^2\sqrt{K}\sigma^2}{\sqrt{Nd^{1-c}} T} + \varepsilon^2\mathbb{I}(c=1) 
+ \frac{\kappa^2\sigma^2}{T} +\frac{\sqrt{d}\kappa^2}{T(1-\lambda)^{1.5}} 
\Bigg).
\end{align}
The dominant CC and FC are given by 
\begin{align}
&CC \approx \mathcal{O}\Bigg(\frac{\sqrt{d}\kappa^2\varepsilon^{-2}}{(1-\lambda)^{1.5}}
+ \kappa^2\varepsilon^{-2} \Bigg), FC \approx 
(CC \times p  B + b_0)2d + CC \times (1-p) \times 8b \approx \mathcal{O}\Bigg(
\frac{d^{1-\frac{c}{2}}\kappa^2\sqrt{N}\varepsilon^{-2}}{\sqrt{K}(1-\lambda)^{1.5}} + d^{\frac{3-c}{2}}\sqrt{\frac{N}{K}}
\Bigg).
\end{align}
\end{Corollary}
We note that ZO-PAGE-EXTRA achieves performance comparable to ZO-PAGE-ED.

\begin{Corollary}[\textbf{ZO-PAGE-GT}]
\label{corollary:offline:page+atc-gt:rd}
Under Assumptions \ref{main:assumption:costfunction}-\ref{main:asssumption:RD-ZO}  and the matrix condition in Lemma \ref{appendix:lemma:transformed_recursion},
we consider $\bu^j_{k,i} \not= 0$ during $\bpi_i=0$.
We consider the ATC-GT strategy shown on Table \ref{tab:matrix_choices}.
We set the hyperparameters as follows
\begin{align}
 \eta_x &= \mathcal{O}\Bigg(\frac{(1-\lambda)^{2}}{\sqrt{d}\kappa^2}\Bigg), \eta_y = \mathcal{O}\Bigg(\frac{(1-\lambda)^{2}}{\sqrt{d}}\Bigg),B = \frac{N}{d^c}, b =b_0 = \sqrt{\frac{Nd^{1-c}}{K}},\notag \\
  p &= \frac{1}{\sqrt{NKd^{1-c}}}, \beta = 0,  \delta_x =
\mu_x = \mathcal{O}
 \Bigg(
 \frac{1}{d^{\frac{3}{4}}_1 (1-\lambda)T^{\frac{1}{2}}}
 \Bigg),  \delta_y=\mu_y = \mathcal{O}
 \Bigg(
 \frac{1}{d^\frac{3}{4}_2 (1-\lambda)T^{\frac{1}{2}}}
 \Bigg).
\end{align}
where $c=1$ if $N \ge \mathcal{O}(d\kappa^2\varepsilon^{-2}/K)$; otherwise $c=0$ (full-batch).
We obtain the performance bound given by 
\begin{align}
&\frac{1}{T}\sum_{i=0}^{T-1}\Big(\mE\|\nabla_x J(\bx_{c,i}, \by_{c,i})\|^2 + \mE\|\nabla_y J(\bx_{c,i}, \by_{c,i})\|^2\Big) \notag \\
&\le \mathcal{O}
\Bigg(
\frac{\sqrt{d}\kappa^2\Delta_0}{(1-\lambda)^{2}T}
+\frac{\kappa^2(1-\lambda)\zeta^2_0}{T}
+\frac{\kappa^2\lambda^4\sqrt{K}\sigma^2}{\sqrt{Nd^{1-c}} T}
+ \varepsilon^2\mathbb{I}(c=1)+\frac{\kappa^2\sigma^2}{T} +\frac{\sqrt{d}\kappa^2}{(1-\lambda)^{2}T} 
\Bigg)
\end{align}
The dominant CC and FC are given by 
\begin{align}
&CC \approx \mathcal{O}\Bigg(\frac{\sqrt{d}\kappa^2\varepsilon^{-2}}{(1-\lambda)^{2}}
+ \kappa^2\varepsilon^{-2} \Bigg), FC \approx 
\Big(CC \times p B +b_0\Big) 2d + CC \times (1-p) \times 8b \approx \mathcal{O}\Bigg(
\frac{d^{1-\frac{c}{2}}\kappa^2\sqrt{N}\varepsilon^{-2}}{\sqrt{K}(1-\lambda)^{2}} + d^{\frac{3-c}{2}}\sqrt{\frac{N}{K}}
\Bigg).
\end{align}
\end{Corollary}

\begin{Corollary}[\textbf{ZO-L2S-ED}]
\label{corollary:offline:L2S+ed:rd}
Under Assumptions \ref{main:assumption:costfunction}-\ref{main:asssumption:RD-ZO}  and the matrix condition in Lemma \ref{appendix:lemma:transformed_recursion},
we consider $\bu^j_{k,i} \not = 0$ during $\bpi_i=0$.
We consider the ED strategy shown on  Table \ref{tab:matrix_choices}.
We set the hyperparameters as follows
\begin{align}
\eta_x &= \mathcal{O}\Bigg(\frac{K^{\frac{c+1}{2}}}{\kappa^2\sqrt{N}d^{1-\frac{c}{2}}}\Bigg), \eta_y = \mathcal{O}\Bigg(
 \frac{K^{\frac{c+1}{2}}}{\sqrt{N}d^{1-\frac{c}{2}}}
 \Bigg), \notag\\
  b &= \mathcal{O}(1), b_0 = \frac{\sqrt{N}}{K^{\frac{c+1}{2}}}, B = \frac{N}{(dK)^{c}}, p = \frac{K^{c}}{Nd^{1-c}}, \beta = 0, \\
 \delta_x &= \mu_x= \mathcal{O}
 \Bigg(
 \frac{1}{d_1 T^{\frac{1}{2}}N^{\frac{1}{4}}}
 \Bigg),  \delta_y=\mu_y = \mathcal{O}
 \Bigg(
 \frac{1}{d_2T^{\frac{1}{2}} N^{\frac{1}{4}}}
 \Bigg).
\end{align}
where $c=1$ if $N \ge \mathcal{O}(\max\{d\kappa^2\varepsilon^{-2}, \frac{\sqrt{d}K\kappa\varepsilon^{-1}}{(1-\lambda)^{1.5}}\})$; otherwise $c=0$ (full-batch).
We obtain the performance bound given by 
\begin{align}
&\frac{1}{T}\sum_{i=0}^{T-1}\Big(\mE\|\nabla_x J(\bx_{c,i}, \by_{c,i})\|^2 + \mE\|\nabla_y J(\bx_{c,i}, \by_{c,i})\|^2\Big) \notag \\
&\le \mathcal{O}\Bigg(
\frac{d^{1-\frac{c}{2}}\kappa^2\sqrt{N}\Delta_0}{K^{\frac{c+1}{2}}T}
+\frac{\kappa^2\zeta^2_0}{T(1-\lambda)^2}+\frac{\kappa^2\lambda^2(K^{\frac{c+1}{2}}\sigma^2+1)}{\sqrt{N}T(1-\lambda)^3}+\varepsilon^2\mathbb{I}(c=1)+ \frac{d^{1-c}\kappa^2\sqrt{N}\sigma^2}{K^{\frac{c+1}{2}}T} + \frac{d^{1-c}\kappa^2\sqrt{N}}{K^{c+1}T}
  + \frac{\kappa^2}{\sqrt{N}T}
\Bigg).
\end{align}
The dominant CC and FC are given by 
\begin{align}
CC &\approx 
\mathcal{O}\Bigg(
\frac{d^{1-\frac{c}{2}}\kappa^2\sqrt{N}\varepsilon^{-2}}{K^{\frac{c+1}{2}}} 
+\frac{\kappa^2\epsilon^{-2}}{(1-\lambda)^2}+ \frac{\kappa^2K^{\frac{c+1}{2}} \varepsilon^{-2}}{\sqrt{N}(1-\lambda)^3}
\Bigg), \notag\\
FC &\approx 
\Big(CC \times p \times B +b_0\Big)\times 2d + CC \times (1-p) \times 8b   \notag\\
\approx \mathcal{O}&\Bigg(
\frac{d^{1-\frac{c}{2}}\kappa^2\sqrt{N}\varepsilon^{-2}}{K^{\frac{c+1}{2}}} +\frac{\kappa^2\varepsilon^{-2}}{(1-\lambda)^2}+ \frac{\kappa^2K^{\frac{c+1}{2}} \varepsilon^{-2}}{\sqrt{N}(1-\lambda)^3} + \frac{d\sqrt{N}}{K^{\frac{c+1}{2}}} 
\Bigg).
\end{align}
\end{Corollary}
\begin{Corollary}[\textbf{ZO-L2S-GT}]
\label{corollary:offline:L2S+atc-gt:rd}
Under Assumptions \ref{main:assumption:costfunction}-\ref{main:asssumption:RD-ZO} and the matrix condition in Lemma \ref{appendix:lemma:transformed_recursion},
we consider $\bu^j_{k,i} \not= 0$ during $\bpi_i=0$.
We consider ATC-GT strategy shown in Table \ref{tab:matrix_choices}.
We set the hyperparameters the same as Corollary \ref{corollary:offline:L2S+ed:rd} and choose  $c=1$ if $N \ge \mathcal{O}(\max\{d\kappa^2\varepsilon^{-2}, \frac{\sqrt{d}\kappa K\varepsilon^{-1}}{(1-\lambda)^{2}}\})$; otherwise $c=0$ (full-batch).
We obtain the performance bound given by 
\begin{align}
&\frac{1}{T}\sum_{i=0}^{T-1}\Big(\mE\|\nabla_x J(\bx_{c,i}, \by_{c,i})\|^2 + \mE\|\nabla_y J(\bx_{c,i}, \by_{c,i})\|^2\Big) \notag \\
&\le \mathcal{O}\Bigg(
\frac{d^{1-\frac{c}{2}}\kappa^2\sqrt{N}\Delta_0}{K^{\frac{c+1}{2}}T}
+\frac{\kappa^2\zeta^2_0}{T(1-\lambda)^3}+\frac{\kappa^2\lambda^4(K^{\frac{c+1}{2}}\sigma^2+1)}{\sqrt{N}T(1-\lambda)^4} 
+\varepsilon^2\mathbb{I}(c=1)+ \frac{d^{1-c}\kappa^2\sqrt{N}\sigma^2}{K^{\frac{c+1}{2}}T} + \frac{d^{1-c}\kappa^2\sqrt{N}}{K^{c+1}T} + \frac{\kappa^2}{\sqrt{N}T}
\Bigg).
\end{align}
The dominant CC and FC are given by 
\begin{align}
CC &\approx 
\mathcal{O}\Bigg(
\frac{d^{1-\frac{c}{2}}\kappa^2\sqrt{N}\varepsilon^{-2}}{K^{\frac{c+1}{2}}} 
+\frac{\kappa^2\epsilon^{-2}}{(1-\lambda)^3}+ \frac{\kappa^2K^{\frac{c+1}{2}} \varepsilon^{-2}}{\sqrt{N}(1-\lambda)^4}
\Bigg), \notag\\
FC &\approx 
\Big(CC \times p \times B +b_0\Big)\times 2d + CC \times (1-p) \times 8b   \notag\\
\approx \mathcal{O} &\Bigg(
\frac{d^{1-\frac{c}{2}}\kappa^2\sqrt{N}\varepsilon^{-2}}{K^{\frac{c+1}{2}}} +\frac{\kappa^2\varepsilon^{-2}}{(1-\lambda)^3}+ \frac{\kappa^2K^{\frac{c+1}{2}} \varepsilon^{-2}}{\sqrt{N}(1-\lambda)^4} + \frac{d\sqrt{N}}{K^{\frac{c+1}{2}}} 
\Bigg).
\end{align}
\end{Corollary}

\end{document}